\documentclass[a4paper]{amsart}
\input{preamble}

\begin{document}
\title{Relative dendroidal Rezk nerve and applications}
\author{K.\ Arakawa, V.\ Carmona, and F.\ Pratali}
\keywords{relative $\infty$-operads, localization, higher algebra, operadic nerve}
\subjclass{18Nxx, 18N55, 18N70, 18M60, 18M75, 55U35}
\begin{abstract}
We extend the dendroidal Rezk nerve to the setting of relative $\infty$-operads. 
Our main theorem relates it to localization of $\infty$-operads, generalizing a theorem of Mazel-Gee. 
By exploiting the relation, we obtain a surprisingly effective tool to prove localization results in operadic contexts. 
As applications, we obtain a number of new results on operadic localizations, including a generalization of Willwacher's recent result on cyclic operads and operadic modules, and a description of locally constant factorization algebras on spheres in terms of discrete geometry.
\end{abstract}

\maketitle

\tableofcontents


\section*{Introduction}
%
%
%
%

A recurring theme in higher category theory is that a non-trivial $\infty$-category
$\cat C$ is presented as the localization of a more tractable $\infty$-
(or even ordinary) category $\cat D$. This perspective is extremely
useful, as it unlocks a powerful universal property that reduces various
categorical constructions associated with $\cat C$ to those of $\cat D$. 

There is a parallel story in the theory of higher operads, where many
nontrivial $\infty$-operads are obtained from a more tractable $\infty$-operad
by inverting a collection of unary operations \cite{arakawa_relative_2025,basterra_inverting_2018,benini_equivalence_2026,benini_c-categorical_2026,benini_strictification_2023,benini_operads_2025,calaque_not_2026, calaque_algebras_2026,carmona_additivity_2025,lurie_higher_nodate,karlsson_assembly_2026,pratali_root_2025}.
The utility of localization is perhaps even greater in the operadic
setting, because constructing maps of $\infty$-operads is even more
difficult than constructing functors of $\infty$-categories. Because
of this, there is now a growing demand to identify $\infty$-operads
as localizations of other $\infty$-operads. However, the theory of
operadic localization is still under development, and proving
localization results of $\infty$-operads can be quite challenging.

This paper contributes to the theory of operadic localization by developing techniques to analyze localization of $\infty$-operads.
This will be made possible by understanding the relation between localization and the \emph{dendroidal Rezk nerve}. Recall that this is the fully faithful functor from the $\infty$-category of $\infty$-operads into that of dendroidal spaces, defined by the following formula:
\begin{align*}
	\dN\from\Op_{\infty} & \xhookrightarrow{\quad}\d\cS=\Fun(\Den^{\op},\cS)\\
\cat O & \longmapsto\pr{T\mapsto\Map_{\Op_\infty}\pr{T,\cat O}}.
\end{align*}
We can generalize this
to \emph{relative $\infty$-operads} $\pr{\cat O,\cat W}$, which are pairs consisting of an $\infty$-operad $\cat O$ equipped with a wide subcategory
$\cat W$ of the $\infty$-category of unary operations of $\cat O$.
For such a pair, we define the \emph{relative dendroidal Rezk nerve}
$\dN\pr{\cat O,\cat W}$ by
\[
\dN\pr{\cat O,\cat W}_{T}=\abs{\Alg_{T}\pr{\cat O}\!\,^{\cat W}},
\]
where $\abs -\from\Cat_{\infty}\to\cS$ denotes the groupoid
completion, and $\Alg_{T}\pr{\cat O}\!\,^{\cat W}$ denotes the category
of operad maps $T\to\cat O$ and natural transformations whose components
belong to $\cat W$. Our result asserts that, by applying the left adjoint $\OO\from \d \cat S \to \Opinfty$ of $\dN$ to the  map $\dN (\cat O) \to \dN(\cat O, \cat W)$, we get a localization:

\begin{customthm}[Theorem {\ref{thm:MG}}]
\label{thm:main}

The map
\[
\cat O\simeq\OO\circ\dN\pr{\cat O}\longrightarrow\OO\circ\dN\pr{\cat O,\cat W}
\]
exhibits its target as the localization of $\cat O$ at $\cat W$.

\end{customthm}

Theorem \ref{thm:main} is a direct generalization of Mazel-Gee's
theorem \cite[Theorem 3.8]{mazel-gee_universality_2019}, which establishes a version of Theorem \ref{thm:main} for
$\infty$-categories and simplicial spaces. But we remark that his original argument does not readily extend to the operadic
setting. Instead, our proof of Theorem \ref{thm:main} relies on a
recent advance in the understanding of the mechanism underlying his
theorem, due to the first author and Bastiaan Cnossen \cite{arakawa_short_2026}. 

Theorem \ref{thm:main} has profound implications in the study of localization of $\infty$-operads. 
Essentially, it provides an alternative presentation of localization, which looks very different from the one given by its universal property. 
The main thrust of this paper is that, by carefully exploiting this presentation, one can deduce a number of new localization results and simplify the proof of existing results.
We demonstrate this by providing numerous applications, which we summarize in the (non-exhaustive) list below.
\begin{enumerate}[wide, labelwidth=!, label = (\arabic*), itemsep = 1\baselineskip, topsep=1\baselineskip]
	\item Willwacher recently discovered that, over characteristic 0, the homotopy category of dg cyclic operads embeds fully faithfully into that of pointed operadic modules \cite{willwacher_cyclic_2024}. We show that this phenomenon happens not only over positive characteristics, but also in any symmetric monoidal $\infty$-category:

\begin{customcor}[Corollary {\ref{cor:cyc_opd_mod_infty}}]\label{cor:cycmod}
	For any symmetric monoidal $\infty$-category $\cat C$, the forgetful functor from the $\infty$-category of cyclic operads in $\cat{C}$ into that of pointed operadic modules in $\cat{C}$ is fully faithful. We can also identify the essential image.
\end{customcor}
\noindent
We prove Corollary \ref{cor:cycmod} by showing that the operad governing cyclic operads is the localization of the operad governing pointed operadic modules.

\item We establish a number of new results on prefactorization algebras, a rich source of localization problems. For instance, we will show that the continuous geometry of spheres can be described by the discrete geometry of lattices: 

\begin{customcor}[Corollary {\ref{cor: First and a half Lattice Localization}} and Proposition {\ref{prop: Second Lattice Localization}}]
	\label{cor:sphere_discrete}
		Locally constant (pre)factorization algebras over the $n$-sphere $\mathbb{S}^n$ are equivalent to those defined over cones on  $\mathbb{Z}^{n+1}$.
\end{customcor}

\noindent
We will also relate the little cubes operads with Lorentzian geometry, where disjointness is replaced by causal independence:

\begin{customcor}[Proposition {\ref{prop: Second Lorentzian localization}}]
	\label{cor:little_minkowski}
Locally constant prefactorization algebras on $m$-dimensional Minkowski spacetime defined over the Lorentzian analog of open disks are equivalent to $\mathbb{E}_{m-1}$-algebras.
\end{customcor}

\noindent
We remark that Corollaries \ref{cor:sphere_discrete} and \ref{cor:little_minkowski} have direct implications in the theory of superselection sectors for lattice quantum systems \cite{benini_c-categorical_2026}, and for algebraic quantum field theories on Minkowski spacetime \cite{benini_prefactorization_2026}.

In addition to these results, we also solve localization problems related to the time-slice axiom of various kinds of prefactorization algebras defined over a fixed spacetime (Proposition \ref{prop: First Lorentzian localization} and Corollary \ref{cor: CLFs for Lorentzian pFAs}). Our conclusions imply various strictification theorems for the time-slice axiom in the spirit of \cite{benini_strictification_2023}.

\item The underlying $\infty$-category of a symmetric monoidal model category $\mathbf{V}$ admits a symmetric monoidal structure, given by the derived tensor product. A theorem of Nikolaus--Scholze characterizes this symmetric monoidal $\infty$-category by a universal property \cite[Theorem A.7]{nikolaus_topological_2018}. We give a concise proof of this theorem:
\begin{customcor}[Theorem \ref{thm:NS}]
	\label{cor:NS}
	The underlying symmetric monoidal $\infty$-category of $\mathbf{V}$ has the universal property of lax symmetric monoidal localization.
\end{customcor}
We also record a version of Corollary \ref{cor:NS} for more general class of symmetric monoidal $\infty$-categories, called symmetric monoidal $\infty$-categories with cofibrations and weak equivalences (Theorem \ref{thm:NS}).

\item We generalize the classical theory of calculus of fractions \cite{gabriel_calculus_1967} to the setting of operads:
	\begin{customcor}[Theorem \ref{thm:lcalc}]
	\label{cor:lcalc}
	Let $(\cat O, \cat W)$ be a relative (1-)operad admitting the calculus of left or right fractions. Then the $\infty$-operadic localization $\cat O [\cat W^{-1}]$ is the (1-)operad obtained by identifying zigzags of arrows (or fractions) using natural equivalence relations on them.
\end{customcor}
\end{enumerate}

In essence, Corollary \ref{cor:lcalc} says that localizations of certain discrete operads carry no higher homotopical information (and admit an explicit point-set model).
This is important, because such information can have far-reaching consequences. 
For example, the $\infty$-categorical equivalence between time-orderable prefactorization algebras and AQFTs (algebraic quantum field theories), which is still open, can be reduced to the discreteness of a localization of an operad controlling time-orderable prefactorization algebras \cite[Open Problem 5.6]{benini_equivalence_2026}.
Also, in \cite{benini_equivalence_2026}, the authors obtained a strictification theorem for the time-slice axiom for AQFTs by proving the discreteness of the operad controlling them \cite[Theorem 5.1 and Proposition 5.5]{benini_equivalence_2026}.
The key ingredient there was the operadic calculus of left fractions, which the authors were the first to consider.
Thus, calculus of fractions remains to be an important tool in the setting of $\infty$-operads, and Corollary \ref{cor:lcalc} contributes to a greater understanding of this tool.

\subsection*{Outline of the paper}

This paper consists of 6 sections and 2 appendices. 

The first two sections are the theoretical core of this paper. In Section \ref{sec:mainthm},
we prove the main theorem. In Section \ref{sec:criteria}, we use
the main theorem to establish several criteria to detect localizations
of $\infty$-operads and, more generally, maps of relative $\infty$-operads inducing
equivalences between the localizations. Together, the results in these
sections lay the foundation of understanding operadic localization
via the relative dendroidal Rezk nerve.

The next four sections will provide sample applications of the results
in the previous sections. In Section \ref{sec:modelcats}, we give a short proof of Nikolaus--Scholze's result on derived tensor
product. In Section \ref{sec:CLFs}, we develop operadic calculus
of fractions. Section \ref{sec:prefact} is about applications to
prefactorization algebras. In Section \ref{sec:cyclicoperads}, we
extend Willwacher's result to arbitrary symmetric monoidal $\infty$-categories.
These sections are mostly independent of each other, so readers can
jump into any section that interests them right after reading Sections
\ref{sec:mainthm} and \ref{sec:criteria}. The situation is summarized
in Figure \ref{fig:toc}.

%
%
\begin{figure}
	\begin{center}
		\begin{tikzpicture}[node distance =1.5cm]
			\node (s1) [sec] {1. Main theorem};
			\node (s2) [sec, below of=s1] {2. Criterion};
			\node (s3) [sec, below of=s2, xshift=-4cm] {3. Derived tensor};
			\node (s4) [sec, below of=s3, xshift=2cm] {4. Fractions};
			\node (s6) [sec, below of=s2, xshift=4cm] {6. Cyclic operads};
			\node (s5) [sec, below of=s6,xshift=-2cm] {5. PFAs};

			\draw [arrow] (s1) -- (s2);
			\draw [arrow] (s2) -- (s3);
			\draw [arrow] (s2) -- (s4);
			\draw [arrow, dashed] (s4) -- (s5);
			\draw [arrow] (s2) -- (s6);
			\draw [arrow] (s2) -- (s5);
		\end{tikzpicture}
	\end{center}
	\caption{}
	\label{fig:toc}
\end{figure}

The appendices contain definitions and results that are well-known
but are hard to find in the literature. As such, they should be consulted only when the need arises.

\subsection*{Notation and convention}

Throughout the paper, we make extensive use of \emph{$\infty$-categories}.
In particular, unless stated otherwise, various categorical constructions
(colimits, localizations, etc.) are performed in the realm of $\infty$-categories.
For the most part, readers can ignore which model of $\infty$-categories
is used, but when pressed, we mean quasicategories in the sense of
Joyal \cite{joyal_quasi-categories_2002}, which is the most well-developed model of $\pr{\infty,1}$-categories.
Likewise, we use the term \emph{spaces} or \emph{$\infty$-groupoids} to mean Kan complexes or
topological spaces. 

In a similar spirit, most of our discussion on \emph{$\infty$-operads}
does not depend on a specific model of colored $\pr{\infty,1}$-operads.
We try to be as ``model-agnostic'' as possible in the main body
of the paper, as our argument makes sense in any of the established models
of $\infty$-operads, such as Lurie's model of $\infty$-operads.
(We refer to \cite[$\S$ 1.2]{arakawa_relative_2025} for a review of existing
models of $\pr{\infty,1}$-operads.) 

A special case of $\infty$-operads comes from symmetric monoidal $\infty$-categories.
If $\pr{\cat C,\t}$ is a symmetric monoidal $\infty$-category, we
denote by $\cat C^{\t}$ the associated $\infty$-operad. (In Lurie's
model of $\infty$-operads and symmetric monoidal $\infty$-categories,
$\pr{\cat C,\t}$ is \textit{equal} to $\cat C^{\t}$. We generally
prefer the notation $\pr{\cat C,\t}$ or $\cat C$ for symmetric monoidal
$\infty$-categories, as it is closer to the standard convention in
ordinary category theory.)

If $\cat{C}$ is an $\infty$-category, its maximal subgroupoid will be denoted by $\cat{C}^{\simeq}$.

The $\infty$-categories of $\infty$-categories, $\infty$-groupoids
(or spaces), and $\infty$-operads, are denoted by $\Cat_{\infty}, \cat{S}$, and $\Op_{\infty},$
respectively. We write $\mathrm{Op}$ for the (ordinary) category
of ordinary operads. 
We also write $\Rel\Cat_\infty\subset \Fun([1],\Cat_\infty)$ for the full subcategory spanned by the essentially surjective monomorphisms in $\Cat_\infty$, and refer to its objects as \emph{relative $\infty$-categories}. We typically denote its objects as pairs $(\cat{C},\cat{W})$ instead of $\cat{W}\hookrightarrow \cat{C}$.

If $\cat O$ is an $\infty$-operad, we define
its \emph{underlying $\infty$-category} $U\pr{\cat O}$ as the
$\infty$-category consisting of the unary operations of $\cat O$. The functor $U$ is the right adjoint to the inclusion of $\infty$-categories into $\infty$-operads: 
\[
\begin{tikzcd}
	{\Cat_\infty}  & 
	{\Opinfty \from U.} \ar[l, shift right=-1.8,"\scalebox{0.8}{$\perp$}"']\ar[l, shift right=1.8, hookleftarrow]
\end{tikzcd}
\]
Also, if $\cat O$ and $\cat P$ are $\infty$-operads, the $\infty$-category
of maps of operads $\cat O\to\cat P$ will be denoted by $\Alg_{\cat O}\pr{\cat P}$. 


\subsection*{Acknowledgments} We are grateful to G.\ Horel for suggesting that Willwacher's theorem on cyclic operads and operadic modules might be an instance of operadic localization. We also thank M.\ Benini and A.\ Schenkel for many interesting conversations which justified looking at some of the localization examples in this work regarding prefactorization algebras. Finally, we would like to thank B.\ Cnossen for suggestions that improved the exposition, G.\ Heuts for fruitful exchanges, and I.\ Moerdijk for spotting an oversight in an earlier draft of this paper.

K.A.\ was supported by JSPS KAKENHI Grant Number 24KJ1443. V.C.\ was partially supported by the project PID2024-157173NB-I00 funded by MCIN/AEI/ 10.13039/501100011033 and by FEDER, UE. F.P.\ was supported by the Dutch Research Council (NWO), Grant ID DGZBU69810. Part of this work was carried out during a visit of the first two authors to Utrecht University, supported by NWO. Additionally, the authors received funding from the Max Planck Institute for Mathematics in the Sciences.



\section{Relative dendroidal Rezk nerve}
\label{sec:mainthm}

The following definition is the starting point of this paper.
\begin{defn}
	A morphism $\lambda\from\cat O\to\cat O'$ of $\infty$-operads
	exhibits $\cat O'$ as the \emph{localization} of $\cat O$ at a
	subcategory $\cat W\subset U\pr{\cat O}$ if for each $\infty$-operad
	$\cat P$, the map
	\[
	\lambda^*\colon \Alg_{\cat O'}\pr{\cat P}\longrightarrow\Alg_{\cat O}\pr{\cat P}
	\]
	is fully faithful, with essential image given by those algebras that carry
	each map in $\cat W$ to equivalences. In this case, we write $\cat O'=\cat O[\cat W^{-1}]$.
\end{defn} 
The goal of this Section is to provide a model for the operadic localization $\cat O[\cat W^{-1}]$ in terms of \emph{dendroidal spaces/ complete dendroidal Segal spaces} (Theorem \ref{thm:main}). In order to do this, we first recall the formalism of complete dendroidal Segal spaces. The expert reader can jump directly to Section \ref{subsec:mainthm}.

\subsection{Complete dendroidal Segal spaces}

Write $\Omega$ for Moerdijk-Weiss dendroidal category \cite{moerdijk_inner_2009}: Its objects are (finite, rooted non-planar) trees, each of which freely generates a colored discrete operad whose set of objects are the edges of the trees and the generating operations are the vertices. Morphisms in $\Omega$ are maps of operads, which makes the free operad functor $\Omega\hookrightarrow \mathrm{Op}$ fully faithful.

\begin{notation}
	Given a tree $T$, we will not distinguish notationally between $T$ as an object of $\Omega$, as an operad and as a representable dendroidal presheaf. We write $E(T)$ for the set of edges of a tree, i.e., the objects of the operad generated by $T$.
\end{notation}

We write $\d\cS\coloneqq \mathrm{PSh}(\Omega)=\mathrm{Fun}(\Omega^\text{op},\cS)$ for the $\infty$-category of space-valued presheaves on $\Omega$, called \emph{dendroidal spaces}.

\begin{defn}
	A \emph{complete dendroidal Segal space} is a presheaf $X\in \d\cS$ satisfying the Segal and completeness conditions:
	\begin{enumerate}
		\item For an inner edge $e$ in $T$, let $T^e$ and $T_e$ be the upper and lower part of $T$ obtained by cutting $T$ at $e$, so that $T= T^e\cup_{\eta_e} T_e$ in $\Omega$. Then the natural map $$ X_T\longrightarrow X_{T^e}\times_{X_{\eta_e}} X_{T_e}$$ is an equivalence.
		\item The underlying simplicial Segal space  $i^*X\colon \Delta^\text{op}\to \cS$ is complete, where $i^*$ denotes restriction along the canonical inclusion $i\from \Delta\hookrightarrow\Omega$. 
	\end{enumerate}
\end{defn}
It is by now classical that $\infty$-operads admit a model in terms of presheaves on $\Omega$ satisfying the Segal and completeness conditions (see \cite{cisinski_dendroidal_2013}, or \cite{heuts_simplicial_2022}). In fact, the full subcategory of $\d\cS$ spanned by complete dendroidal Segal spaces is equivalent to the essential image of the fully faithful functor $$ \dN \from \Opinfty\longrightarrow \Fun(\Omega^\text{op},\cS) , \quad \cO \longmapsto \Map_{\Opinfty}(\Omega(-),\cO),$$
called the \emph{dendroidal Rezk nerve functor}. In other words, the functor $\dN$ is the restricted Yoneda embedding along the inclusion $\Omega\hookrightarrow\mathrm{Op}\hookrightarrow\Opinfty$. 

The dendroidal nerve is in fact part of an adjunction
\[
	\begin{tikzcd}
		\OO \from \d\cS \ar[r, shift left=1.8]\ar[r, shift right=1.8,"\scalebox{0.8}{$\perp$}", hookleftarrow] & 
		\Opinfty \from \dN,
	\end{tikzcd}
\]
making $\Opinfty$ a reflective localization of dendroidal spaces. One can implement this adjunction by identifying $\Opinfty$ with complete dendroidal Segal spaces. In this case, the left adjoint $\OO$ can be interpreted as sending a dendroidal space to its \emph{complete Segalification}.

\begin{rem}
	A first version of this nerve was already discussed, in a model-categorical fashion, in \cite{cisinski_dendroidal_2013} and \cite[\S 14]{heuts_simplicial_2022}.
\end{rem}

The $\infty$-category $\Opinfty$ admits a closed symmetric monoidal structure. This can be seen via Lurie’s tensor product for $\infty$-operads (\cite[Proposition 2.2.5.7]{lurie_higher_nodate}), or alternatively, by appealing to Hinich and Moerdijk's construction, which lifts the tensor product on dendroidal sets to obtain an equivalent symmetric monoidal structure on complete dendroidal Segal spaces (\cite[\S 5]{hinich_equivalence_2024} and \cite{hinich_addendum_2025}).
We will write $$  \uALG\!\,_{\scalebox{0.6}{\,$(-)$}}(-)\colon \Opinfty^\op\times \Opinfty\longrightarrow \Opinfty$$ for the internal-hom functor.  
Since the inclusion $\Cat_\infty\to \Opinfty$ is symmetric monoidal, the underlying category functor is lax monoidal, and $\Opinfty$ acquires the structure of an $(\infty,2)$-category, with mapping $\infty$-categories given by $$\Alg_{\cO}(\cP)\coloneqq U\uALG\!\,_{\cO}(\cP)$$ for every pair of $\infty$-operads $\cO,\cP$.
Note in particular that there is an equivalence
$$ \Alg_\cO(\cP)^\simeq \simeq \Map_{\Opinfty}(\cO,\cP) .$$

Finally, observe that, in fact, the closed monoidal structure on $\Opinfty$ has an $(\infty,2)$-universal property: For any choice of $\infty$-operads $\cO,\cP,\cQ$, there is a categorical equivalence  $$\Alg_{\cO\otimes \cP}(\cQ)\simeq \Alg_{\cO}(\uALG\!\,_\cP(\cQ)).$$

We will now introduce the $\infty$-category of \emph{relative $\infty$-operads} $\RelOp_\infty$, where $ \infty$-operads are also equipped with a choice of weak equivalences, and see how the above constructions can be extended to the relative context as well.

\subsection{Relative $\infty$-operads}
Consider the  $\infty$-category of relative $\infty$-operads, defined as
$${\RelOp_\infty \coloneqq \Opinfty\times_{\Catinfty} \RelCatinfty}.$$ Explicitly, an object in $\RelOp_\infty$ is a pair $(\cO,\cW)$, where $\cO$ is an $\infty$-operad and $\cW \subseteq U(\cO)$ is a wide subcategory of the underlying  $\infty$-category of $\cO$. The maps in $\cW$ are thought of as weak equivalences.  A morphism $(\cO,\cW)\to(\cP,\cW')$ in $\RelOp_\infty $ is a map of $\infty$-operads $\cO\to \cP$ preserving weak equivalences.

There is a forgetful functor $\mathrm{fgt}\colon \RelOp_\infty\to \Opinfty, (\cO,\cW)\mapsto\cO$, which admits fully faithful left and right adjoints,
\[\begin{tikzcd}
	\Opinfty && {\RelOp_\infty}
	\arrow[""{name=0, anchor=center, inner sep=0}, "{(-)^\flat}"', bend right = 50, from=1-3, to=1-1, hookleftarrow]
	\arrow[""{name=1, anchor=center, inner sep=0}, "{(-)^\sharp}"', bend right =50, from=1-1, to=1-3, hookrightarrow]
	\arrow[""{name=2, anchor=center, inner sep=0}, "{\mathrm{fgt}}"{description}, from=1-3, to=1-1]
	\arrow["\scalebox{0.8}{$\perp$}"{description}, between={0.2}{0.8}, phantom, from=0, to=2]
	\arrow["\scalebox{0.8}{$\perp$}"{description}, between={0.2}{0.8}, phantom, from=2, to=1]
\end{tikzcd}.\]
The functors $(-)^\flat $ and $ (-)^\sharp $ associate to an $\infty$-operad its minimal and maximal marking, respectively. Explicitly, given an $\infty$-operad $\cO$, we have  $$  \cO^\flat=(\cO, U\cO^\simeq), \quad \quad  \cO^\sharp=(\cO, U\cO).$$ 
The adjunction restricts to an analogous one for the functor $\mathrm{fgt}\colon \RelCatinfty\to\Catinfty$.

	The $\infty$-category $\RelOpinfty$ has a symmetric monoidal structure, where the tensor product is ``almost'' the pointwise tensor product. This does not immediately follow from the definition of $\RelOp_\infty$ as a pullback of symmetric monoidal $\infty$-categories, as the underlying category functor $U\colon \Opinfty\to \Catinfty$ is only lax monoidal.\footnote{The fact that the canonical map  $ U(\cO)\times U(\cP)\to U(\cO\otimes \cP)$ is not always an equivalence can be seen by taking $\cO$ to be a stump and $\cP$ a $2$-corolla: The category $U(\cO)\times U(\cP)$ is discrete, but $U(\cO\otimes \cP)$ is not.} Given two relative $\infty$-operads $(\cO,\cW)$, $(\cP,\cW')$, the underlying $\infty$-operad of $(\cO,\cW)\otimes (\cP,\cW')$ is simply $\cO\otimes\cP$, and the marking is given by the subcategory generated by  $\cW\times \cW'$ under the canonical functor $U(\cO)\times U(\cP)\to U(\cO\otimes \cP)$.
	
The way this symmetric monoidal structure can be deduced is by first observing that $\RelOpinfty$ is the underlying category of an $\infty$-operad, as it is a pullback of $\infty$-operads. The symmetric monoidal structure is given by considering the subspaces of multimorphisms $(\cO_i,\cW_i)_{i\in I} \to (\cP,\cW')$ such that the image of the restriction  $\prod_i\cW_i\hookrightarrow \prod_i U\cO_i \to U(\otimes_i \cO_i)\to U\cP$ lands in $\cW'$. 
	

Most importantly, there is an internal-hom object functor, for which we write $$  \RelOp_\infty^{\text{op}}\times \RelOp_\infty\longrightarrow \RelOp_\infty,\quad ((\cO,\cW),(\cP,\cW'))\longmapsto \uALG\!\,_{(\cO,\cW)}(\cP,\cW')^+.$$ 

The relative $\infty$-operad $ \uALG\!\,_{(\cO,\cW)}(\cP,\cW')^+$ is given by the pair
$$ \uALG\!\,_{(\cO,\cW)}(\cP,\cW')^+\coloneqq \pr{ \uALG\!\,_{(\cO,\cW)}(\cP,\cW'), \Alg_{(\cO,\cW)}(\cP,\cW')^{\cW'}}, $$  where:
\begin{itemize}
	\item 
	The underlying $\infty$-operad $\Alg_{(\cO,\cW)}(\cP,\cW')$  is the full suboperad of $\uALG\!\,_\cO(\cP)$ spanned by the algebras $F\colon \cO\to \cP$ restricting to functors $\cW\to\cW'$. 
	\item 	The subcategory of marked edges is defined as the pullback
	\[
	\begin{tikzcd}
		\Alg_{(\cO,\cW)}(\cP,\cW')^{\cW'} \arrow[r]\arrow[d] \ar[rd, phantom, "\lrcorner" description, near start]& \cW^{\prime\times \mathrm{Ob}(U\cO)} \arrow[d] \\
		\Alg_{(\cO,\cW)}(\cP,\cW')  \arrow[r] & U\cP^{\times\mathrm{Ob}(U\cO)}
	\end{tikzcd}.
	\]
Here, $	\Alg_{(\cO,\cW)}(\cP,\cW') $ is the underlying $\infty$-category of $\uALG\!\,_{(\cO,\cW)}(\cP,\cW')$ and the bottom horizontal arrow is obtained as the composite $$ \Alg_{(\cO,\cW)}(\cP,\cW') \hookrightarrow \Alg_\cO(\cP)\to\Alg_{\mathrm{Ob}(U\cO)}(\cP)\simeq U\cP^{\times \mathrm{Ob}(U\cO)},$$  where the last arrow is the restriction along the inclusion $\mathrm{Ob}(U\cO)\to U(\cO)$. 
	In other words, the chosen weak equivalences are those natural transformations ${F\to G}$ which objectwise lie in $\cW'$. 
\end{itemize}

For the purposes of this paper, the two instances relevant to our discussion correspond to the domain with minimal or maximal marking, as we shall see in the next sections.

\subsection{Main theorem}\label{subsec:mainthm}

We now define the \emph{relative dendroidal Rezk nerve functor}. Our formulation generalizes Mazel-Gee's $\infty$-categorical formulation of the relative Rezk's nerve for relative $\infty$-categories: The main result of this section is to generalize  \cite[Theorem 3.8]{mazel-gee_universality_2019} and show that, up to dendroidal complete Segalification, this functor computes localization of $\infty$-operads. 
\begin{defn}
	For a relative $\infty$-operad $ (\cO,\cW)$ and a tree $T\in\Omega$, we write $\Alg_{T}(\cO)^\cW$ for the $\infty$-category obtained as the pullback 
\[
\begin{tikzcd}
	\Alg_{T}(\cO)^\cW \arrow[r]\arrow[d]\ar[rd, phantom, "\lrcorner" description, near start] &\cW^{\times \mathrm{Ob}(T)} \arrow[d] \\
	\Alg_{T}(\cO) \arrow[r] & U(\cO)^{\times \mathrm{Ob}(T)} \ ,
\end{tikzcd}
\] where the bottom horizontal arrow is induced by the inclusion ${\mathrm{Ob}(T)\to T}$. Recall that the objects of the operad generated by the tree $T$ are just the edges of $T$.
\end{defn}

In other words, the objects of $\Alg_{T}(\cO)^\cW$ are maps of $\infty$-operads $T\to \cO$, and an arrow $\alpha\colon F \to G$ in $\Alg_{T}(\cO)^\cW$ is a natural transformation which objectwise lies in $\cW$. We call any such transformation a \emph{natural weak equivalence}.
Observe that $\Alg_{T}(\cO)^\cW$ is the $\infty$-category of weak equivalences of the relative $\infty$-operad $\uALG\!\,_{T^\flat}(\cO,\cW)^+$, that is, $$ \Alg_{T}(\cO)^\cW \coloneqq \Alg_{T^\flat}(\cO,\cW)^\cW . $$

This construction is natural in $T$, and allows us to define a dendroidal object in $\Catinfty$ in the following way.

\begin{defn}\label{def:relnerve}
	Let $\mathrm{pre}\dN $ be the functor defined as follows, $$ \mathrm{pre}\dN \colon \RelOp_\infty\longrightarrow\mathrm{Fun}(\Omega^\op,\Catinfty), \quad  (\cO,\cW)\longmapsto\Alg_{\scalebox{0.6}{\,$(-)$}}(\cO)\!\,^\cW.$$ The \emph{relative dendroidal Rezk nerve} is the functor $$ \dN\colon\RelOp_\infty\longrightarrow \d\cS$$ defined by the composition $$\dN\colon \RelOp_\infty\xlongrightarrow{\;\mathrm{pre}\dN\;} \Fun(\Omega^\text{op},\Catinfty)\xlongrightarrow{\;|-|_\ast\;} \Fun(\Omega^\text{op},\cS),$$ where $|-|\colon \Catinfty \to \cS$ denotes geometric realization.
\end{defn}
Explicitly, for a tree $T\in\Omega$, the relative dendroidal Rezk nerve of $(\cO,\cW)$ is defined by $$ \dN(\cO,\cW)_T = | \Alg_T(\cO)^{\cW}|.$$

Let us record the following basic but essential remark.

\begin{rem}\label{rmk:nervevsrelnerve}
	There is a canonical equivalence $$\dN(\cO^\flat)\simeq \dN\cO.$$ 
	In particular, $\dN(\cO^\flat)$ is a complete dendroidal Segal space.
\end{rem}

For a general relative $\infty$-operad $(\cO,\cW)$, the dendroidal space $\dN(\cO,\cW)$ need not satisfy either the Segal condition or completeness. The same phenomenon occurs for relative $\infty$-categories. Roughly speaking, this happens because pullbacks and geometric realization do not commute in general. We come back to this discussion in Remark \ref{rem:complete} and systematically in Section \ref{sec:criteria}.

We can now state our main result.

\begin{thm}\label{thm:MG}
	For any relative $\infty$-operad $(\cO,\cW)$, applying the relative nerve to the canonical map $\cO^\flat\to(\cO,\cW)$, the induced morphism $$ \cO\simeq \OO(\dN\cO)\longrightarrow \OO(\dN(\cO,\cW))$$  exhibits the target as the operadic localization of $\cO$ at $\cW$.
\end{thm}

The strategy of the proof relies on the one used in \cite{arakawa_short_2026} to prove the categorical version. In particular, a key step consists of constructing a suitable simplicial resolution for the $\infty$-operad $\OO(\dN(\cO,\cW))$. Let us deal with this now.


\begin{lem}\label{lem:simplres}
	Consider the functor $$ L_\bullet\colon  \RelOp_\infty\longrightarrow \Fun(\Delta^{\op},\Opinfty), \quad (\cO,\cW)\longmapsto L_\bullet(\cO,\cW)\coloneqq \uALG\!\,_{[\bullet]^\sharp}(\cO,\cW). $$  
	There is a natural equivalence of $\infty$-operads $$ \mathrm{colim}_{\Delta^{\op}}L_\bullet(\cO,\cW){\;\simeq\;} \OO(\dN(\cO,\cW)).$$
\end{lem}
\begin{proof}
	
	Consider the following diagram
	\[
	\begin{tikzcd}[column sep=small]
		\RelOp_\infty \arrow[rr, "L_\bullet"] \arrow[ddd, bend right =65, "\N_\d"'] \arrow[d, "\mathrm{pre}\dN"']        &                          & {\Fun(\Delta^\text{op},\Opinfty)} \arrow[ddd, "(\dN)_\ast"] \arrow[dddd,bend left=70,"\mathrm{id}"]                          \\
		{\Fun(\Omega^\text{op},\Catinfty)} \arrow[rd, "\N_\ast"] \arrow[dd, "|-|_\ast"'] & (\ast)                                                                                             &                     \\						    & {\Fun(\Omega^{\text{op}}, \Fun(\Delta^{\text{op}},\cS))} \arrow[ld, "(\colim_{\Delta^{\op}})_\ast"'] \arrow[rd, "\simeq" sloped] &                                                                                                                              \\
		{\Fun(\Omega^\text{op},\cS)} \arrow[d, "\OO"']                             &                      & {\Fun(\Delta^\text{op},\Fun(\Omega^\text{op},\cS))} \arrow[d, "\OO_\ast"] \arrow[ll, "\mathrm{colim}_{\Delta^{\text{op}}}"] \\
		\Opinfty              &                                                                                                    & {\Fun(\Delta^\text{op},\Opinfty).} \arrow[ll, "\mathrm{colim}_{\Delta^{\text{op}}}"]        \end{tikzcd}
 	\]
We will show that this diagram commutes. That all the subdiagrams different from $(\ast)$ commute follows by direct inspection, by recalling that for a category $\cC$, there is an equivalence $|\cC|\simeq \mathrm{colim}_{\Delta^{\text{op}}}\N \cC$ (see \cite{ramzi_mo_2021} for a detailed explanation of this standard result). For the commutativity of the diagram $(\ast)$, we proceed as follows. Given a relative $\infty$-operad $(\cO,\cW)$, evaluating the simplicial dendroidal object $\dN L_\bullet(\cO,\cW)$ at $([n],T)\in \Delta\times\Omega$ gives the equivalence
	$$\dN(L_n(\cO,\cW))_T = \Map_{\Opinfty} (T, \uALG\!\,_{[n]^\sharp}(\cO,\cW))\simeq \Map_{\RelOp_\infty}(T^\flat, \uALG\!\,_{[n]^\sharp}(\cO,\cW)^+),$$ where the last equality exploits the adjunction $(-)^\flat\dashv \mathrm{fgt}$. By using the inner hom-tensor adjunction, we can rewrite the last term as $$ \dN(L_n(\cO,\cW))_T \simeq \Map_{\RelOp_\infty}([n]^\sharp, \uALG\!\,_{T^\flat}(\cO,\cW)^+)\simeq \Map_{\Catinfty}([n],\Alg_{T^\flat}(\cO,\cW)^\cW).$$ As the last term is by definition $ \N(\mathrm{pre}\dN(\cO,\cW)_T)_n$, this concludes the proof.
\end{proof}

We are now able to prove our main result, namely Theorem \ref{thm:MG}.

\begin{proof}[{Proof of Theorem \ref{thm:MG}}]
	By Lemma \ref{lem:simplres},	the thesis reduces to showing that the induced map $$\colim_{\Delta^\text{op}} L_\bullet (\cO,\cW) \longrightarrow \mathrm{colim}_{\Delta^\text{op}} L_\bullet(\cO[\cW^{-1}]^\flat)$$ is an equivalence of $\infty$-operads.
Since the colimit functor is left adjoint to the constant diagram functor, the claim reduces to showing that, for every $\infty$-operad $\cP$, the map of spaces $$ \Theta\colon \Map_{\Fun(\Delta^\text{op}, \Opinfty)}(L_\bullet\cO[\cW^{-1}]^\flat,\cP)\longrightarrow \Map_{\Fun(\Delta^\text{op}, \Opinfty)}(L_\bullet(\cO,\cW),\cP) $$ is an equivalence, where by abuse of notation we write $\cP$ for the constant simplicial diagram at $\cP$. 
	Let us prove this by showing that $(i)$ it is $(-1)$-truncated, and $(ii)$ it is essentially surjective.
	
	For $(i)$, since operadic localization is an epimorphism, by naturality in $[n]$ we just need to show that for every $n\geq 0$ the map $$ L_n(\lambda)\colon L_n(\cO,\cW)\longrightarrow L_n(\cO[\cW^{-1}]^\flat) $$ exhibits the target as the localization of the source at natural weak equivalences.
	
	To this end, we consider the following commutative diagram induced by the map $[n]\to [0]$: 
\[
\begin{tikzcd}
	\cO\simeq &[-11mm] L_0(\cO,\cW) \arrow[r]\arrow[d] & L_0(\cO[\cW^{-1}]^\flat) \arrow[d, "\simeq" sloped] &[-11mm] \simeq \cO[\cW^{-1}] \\
	 & L_n(\cO,\cW) \arrow[r] & L_n(\cO[\cW^{-1}]^\flat)& \simeq \cO[\cW^{-1}].
\end{tikzcd}
\]

As the right vertical arrow is an equivalence and the top horizontal one is a localization, we only need to show that the left vertical map $\cO\to L_n(\cO,\cW)$, where $ L_n(\cO,\cW) = \uALG\!\,_{[n]^\sharp}(\cO,\cW)$, is an equivalence up to localization of the target at natural weak equivalences. But this is true, since evaluation at $[0]$ is a right adjoint and both the unit and the counit are natural weak equivalences.

Observe that we also obtained a characterization of the essential image of $\Theta$: It is given by those $\phi\colon L_\bullet(\cO,\cW)\to \cP$ such that $\phi_0\colon L_0(\cO,\cW)=\cO\to \cP$ (equivalently, $\phi_n$ for every $n$) induces a map of relative $\infty$-operads $L_0(\cO,\cW)^\sharp\simeq (\cO,\cW)\to \cP^\flat$. (Observe the condition is equivalent to asking that $\phi_0$ induces a map  $\cO[\cW^{-1}]\to \cP$.)

In particular, to prove point $(ii)$, we just need to show that this condition is satisfied by any morphism of simplicial diagrams $\phi\colon L_\bullet(\cO,\cW)\to \cP$. By naturality in $[n]$, it is enough to show the condition holds for $n=0$. Consider an arrow $w\colon x \to y$ in $\cW$ and let us prove that $\phi_0(w)$ is an equivalence. In $L_1(\cO,\cW)$ there is a morphism $\alpha\colon w \to \id_y$ of the form \[
\begin{tikzcd}
	x \arrow[r, "w"] \arrow[d, "w"'] & y \arrow[d,"\id_y"] \\
	y \arrow[r, "\id_y"'] & y 
\end{tikzcd}
\] so in particular $\phi_0(w)\simeq\phi_0(d_1(\alpha))\simeq \phi_1(\alpha)\simeq\phi_0(d_0(\alpha))\simeq\phi_0(\id_y).$ Since the latter is clearly an equivalence, this concludes the argument.	
\end{proof}


\begin{rem}The relative nerve construction of Definition \ref{def:relnerve} could be more generally performed for reflective localizations of more general diagram categories. A rich source of meaningful examples is given by appealing to Barwick's (perfect) operator categories \cite{barwick_operator_2018} and more generally to the subcategory of complete Segal objects over algebraic patterns (in the sense of \cite{chu_homotopy-coherent_2021}) receiving a fully faithful functor of algebraic patterns $\Delta^\op\to\cP^\op$. However, the crucial property we have used in the proof is the closed symmetric monoidal structure on $\Opinfty$, and such a structure may not be available in more general cases.

%
\end{rem}
We conclude with a few useful observations about when it is actually necessary to apply the Segal completion functor $\OO$ to the relative nerve, and when this step can be safely omitted.

 Given a relative $\infty$-operad $(\cO,\cW)$, we say that $ (\cO,\cW)$ is \emph{saturated} if the preimage of equivalences under $U\lambda\from U(\cO)\to U(\cO[\cW^{-1}])$ is precisely the subcategory $\cW$. Saturation handles completeness of the relative nerve, whenever this is known to be Segal: More precisely, we can give the following characterization.

\begin{rem}\label{rem:complete}
	Suppose that $\cW$ enjoys the $2$-out-of-$3$ property and that $\dN(\cO,\cW)$ satisfies the Segal condition. Then we claim that $\dN(\cO,\cW)$ is complete if and only if $\pr{\cat O,\cat W}$ is saturated.

Let us now prove our claim. Assume that $\dN(\cO,\cW)$ is complete Segal, and let $f\in \dN(\cO,\cW)_{C_1}$ be a unary operation such that its image in $\dN(\cO[\cW^{-1}])_{C_1}$ is a \emph{homotopy equivalence} (see \cite[\textsection 6]{rezk_model_2001}). Then, by Theorem \ref{thm:MG}, we have an identification of dendroidal spaces $\dN(\cO,\cW)\simeq \dN(\cO[\cW^{-1}])$.
Therefore, $f$ is already a homotopy equivalence in $\dN(\cO,\cW)_{C_1}$. By completeness, $f$ lies in the essential image of the degeneracy map 
$$
s_0\colon \dN(\cO,\cW)_\eta=|\cW|\longrightarrow{\dN(\cO,\cW)}_{C_1}=|\Alg_{C_1}(\cO)^\cW|,
$$
and so $f$ belongs indeed to $\cW$ by $2$-out-of-$3$, as wanted. 
For the converse, suppose that $(\cO,\cW)$ is saturated. Since its relative nerve is Segal by assumption, Theorem \ref{thm:MG} yields a fully faithful and essential surjective map of dendroidal Segal spaces (in the sense of \cite[Definitions 5.9 and 8.9]{cisinski_dendroidal_2013}) $ \dN(\cO,\cW)\to \dN(\cO[\cW^{-1}])$. 
In particular, the subspace $\dN(\cO,\cW)_{\text{hoeq}}\subset \dN(\cO,\cW)_{C_1}$ of homotopy equivalences is the
preimage of $\dN(\cO[\cW^{-1}])_{\text{hoeq}}=\Map([1], U(\cO[\cW^{-1}])^\simeq )$. Since $(\cO,\cW)$ is saturated, this preimage is exactly $|\Fun([1],\cW)|$. Consequently, the map $\dN\pr{\cat O,\cat W}_{\eta}\to\dN\pr{\cat O,\cat W}_{\hoeq}$
	can be identified with the map 
	$
\abs{\cat W}\longrightarrow\abs{\Fun\pr{[1],\cat W}}.
$
	This is evidently a homotopy equivalence, so $\dN\pr{\cat O,\cat W}$
	is complete, as desired. 
\end{rem}

\section{Criteria for localizations and local equivalences} 
\label{sec:criteria}

Our goal in this section is to develop widely applicable criteria to detect operadic localizations based on Theorem \ref{thm:MG}. These criteria are formulated in Propositions \ref{prop:corolla-wise_criterion I} and  \ref{prop:corolla-wise_criterion II}. Moreover, we observe that these statements admit dual versions, e.g., see Proposition \ref{prop:dual corolla-wise_criterion II}. For a discussion comparing our result with Lurie's and Harpaz's theory of weak approximations, see Remark \ref{rem:Weak Approx}.

We start with a direct consequence of Theorem \ref{thm:MG}. Recall that a \emph{local equivalence} between relative $\infty$-operads is a map in $\RelOpinfty$ inducing an equivalence between localizations. In other words, the subcategory of local equivalences in $\RelOpinfty $ is the preimage of $\Op_{\infty}^{\simeq}$ along the operadic localization functor $L\from \RelOpinfty\to \Opinfty$ (see \cite[\textsection 2]{arakawa_relative_2025}).

\begin{cor}\label{coro: Segal-equivalence implies local-equivalence} Let $f\colon (\cO,\cW_{\cO})\to (\cP,\cW_{\cP})$ be a map of relative $\infty$-operads. Then, $f$ is a local equivalence if one of the following conditions is satisfied:
	\begin{itemize}
		\item the functor $f_*\colon \Alg_{T}(\cO)^{\cW_{\cO}}\to \Alg_{T}(\cP)^{\cW_{\cP}}$ is a weak homotopy equivalence for any tree $T\in \Omega$, or
		\item the functor $f_*\colon \Alg_{T}(\cO)^{\cW_{\cO}}\to \Alg_{T}(\cP)^{\cW_{\cP}}$ is a weak homotopy equivalence for any $T\in \{\eta,C_n\}_{n\geq 0}$, and both $\dN(\cO,\cW_{\cO})$ and $\dN(\cP,\cW_{\cP})$ are dendroidal Segal spaces.
	\end{itemize} 
\end{cor}

In light of the previous corollary, we should look for sufficient conditions on the relative $\infty$-operad $(\cO,\cW)\in \RelOp_{\infty}$ ensuring that $\dN(\cO,\cW)$ is a dendroidal Segal space. We provide one instance of this phenomenon in Lemma \ref{lem:dendroidal Rezk nerve is Segal}, by exploiting the cocartesian fibration 
$
\ev_{\mathrm{root}}\from\Alg_{C_{n}}\pr{\cat O}\!\,^{\cat{W}}\to \Alg_{\eta}\pr{\cat O}\!\,^{\cat{W}}\simeq\cat{W}
$
induced by the inclusion of the root $\eta\to C_n$ (see Proposition \ref{prop:root_cc} and Remark \ref{rem:restr_at_root}). The essential ingredient is the following result, commonly known as \emph{Quillen's Theorem B for cocartesian fibrations}:

\begin{lem}[{\cite[Lemma A.1.1]{karlsson_assembly_2026}}, {\cite[Proposition A.3]{arakawa_context_2026}}]\label{lem:Quillen B}
	Let $\cat D'\to \cat D$ be a functor and $p\colon \cat C\to \cat D$ be a cocartesian fibration between $\infty$-categories. If, for each arrow $d\to d'$ in $\cat D$, the induced map between (strict) fibers $\cat C_d \to \cat C_{d'}$ of $p$ is a weak homotopy equivalence,\footnote{This condition is usually referred to as saying that $p$ is a locally constant cocartesian fibration.} the square
	\[\begin{tikzcd}
		{\vert\cat{D}'\times_{\cat{D}}\cat{C}\vert} & {\vert\cat{C}\vert} \\
		{\vert\cat{D}'\vert} & {\vert\cat{D}\vert}
		\arrow[from=1-1, to=1-2]
		\arrow[from=1-1, to=2-1]
		\arrow[from=1-2, to=2-2]
		\arrow[from=2-1, to=2-2]
	\end{tikzcd}\]
	is a pullback in $\cat{S}$.
\end{lem}

\begin{lem}\label{lem:dendroidal Rezk nerve is Segal}
	Let $(\cO,\cW)\in \RelOp_{\infty}$ be a relative $\infty$-operad satisfying the following condition:
	any map $\alpha\colon x\to y$ in $\cW$ induces a weak homotopy equivalence 
	$$
	\alpha_!\colon \Alg_{C_n}(\cO)^{\cW}_{x}\longrightarrow \Alg_{C_n}(\cO)^{\cW}_{y} 
	$$
	between the fibers of the cocartesian fibration $\ev_{\mathrm{root}}\colon \Alg_{C_n}(\cO)^{\cW}\to \cW$.
	Then, $\dN(\cO,\cW)$ is a dendroidal Segal space.
\end{lem}
\begin{proof} 
	We need to show that for any tree $T\in \Omega$, and any grafting  decomposition $T=T'\cup_\ell C_n$ where $T'$ has at least one vertex, the commutative square
	$$
	\begin{tikzcd}
		\vert\Alg_{T}(\cO)^{\cW}\vert \ar[r]\ar[d] & \vert\Alg_{C_n}(\cO)^{\cW}\vert \ar[d,"\ev_{\mathrm{root}}"]\\
		\vert\Alg_{T'}(\cO)^{\cW}\vert \ar[r] & \vert\Alg_{\eta}(\cO)^{\cW}\vert \ar[r, equal] &[-6mm] \vert\cW\vert
	\end{tikzcd}
	$$
	is cartesian. The claim will follow from Lemma \ref{lem:Quillen B} if we check that the cocartesian fibration $\ev_{\mathrm{root}}\colon \Alg_{C_n}(\cO)^{\cW}\to \cW$ is locally constant. This is precisely our hypothesis.  
\end{proof}

Combining these ideas, we arrive at the following criteria:
\begin{prop}\label{prop:corolla-wise_criterion I}
	Let $f\colon (\cO,\cW_{\cO})\to (\cP,\cW_{\cP})$ be a map of relative $\infty$-operads. Suppose the next two conditions are satisfied: 
	\begin{enumerate}
		\item[(1)] $f$ induces a weak homotopy equivalence $\cW_{\cO}\to \cW_{\cP}$.
		\item[(2)] For every $x$ in $\cO$ and every $n\geq 0$, the induced map between fibers, $$ \Alg_{C_n}(\cO)^{\cW_{\cO}}_x\longrightarrow \Alg_{C_n}(\cP)^{\cW_{\cP}}_{f(x)}$$ is a weak homotopy equivalence.
	\end{enumerate}
    Assume further that one of the following statements holds:
    \begin{itemize}
    	\item[(a)] $(\cP,\cW_{\cP})$ fulfills the hypotheses of Lemma \ref{lem:dendroidal Rezk nerve is Segal}.
    	\item[(b)] $(\cO,\cW_{\cO})$ fulfills the hypothesis of Lemma \ref{lem:dendroidal Rezk nerve is Segal}, and $\cW_{\cO}\to \cW_{\cP}$ is essentially surjective.
    \end{itemize}
	Then, $f\colon (\cO,\cW_{\cO})\to (\cP,\cW_{\cP})$ is a local equivalence.
\end{prop}
\begin{proof}
	First, consider the commutative square
	$$
	\begin{tikzcd}
		\Alg_{C_n}(\cO)^{\cW_{\cO}}_x \ar[r,"f_*"]\ar[d,"\alpha_!"'] &[6mm] \Alg_{C_n}(\cP)^{\cW_{\cP}}_{f(x)} \ar[d,"f(\alpha)_!"]\\
		\Alg_{C_n}(\cO)^{\cW_{\cO}}_y\ar[r,"f_*"'] & \Alg_{C_n}(\cP)^{\cW_{\cP}}_{f(y)} 
	\end{tikzcd}
	$$
	associated to any morphism $\alpha\colon x\to y$ in $\cW_{\cO}$. Assuming condition (2) and either (a) or (b), the square implies that Lemma \ref{lem:dendroidal Rezk nerve is Segal} applies to both $(\cO,\cW_{\cO})$ and $(\cP,\cW_{\cP})$. Hence, by Corollary \ref{coro: Segal-equivalence implies local-equivalence}, we are reduced to proving that the functor $f_*\colon \Alg_{C_n}(\cO)^{\cW_{\cO}}\to \Alg_{C_n}(\cP)^{\cW_{\cP}}$ is a weak homotopy equivalence for any $n\geq 0$ (the case $T=\eta$ is just condition (1) in the statement). Notice that this would follow if we show that the commutative square
	$$
	\begin{tikzcd}
		\vert\Alg_{C_n}(\cO)^{\cW_{\cO}}\vert \ar[r, "f_*"]\ar[d,"\vert\ev_{\mathrm{root}}\vert"'] &[6mm] \vert\Alg_{C_n}(\cP)^{\cW_{\cP}}\vert \ar[d,"\vert\ev_{\mathrm{root}}\vert"]\\
		\vert\cW_{\cO}\vert \ar[r,"f"',"\simeq"] & \vert\cW_{\cP}\vert 
	\end{tikzcd}
	$$
	is cartesian (the lower horizontal map is an equivalence by (1)). Applying  Lemma \ref{lem:Quillen B} and condition (2), we conclude this is the case. Just observe that, for any $x\in \vert\cW_{\cO}\vert$, the induced map between fibers over $x$ and $f(x)$ of the vertical arrows in the last square is an equivalence.
\end{proof}

\begin{prop}\label{prop:corolla-wise_criterion II}
	Let $f\colon \cO\to\cP$ be a map of $\infty$-operads and denote by $\cW$ the wide subcategory generated by $f^{-1}(U\cP^\simeq)$. Suppose the next two conditions are satisfied: 
	\begin{enumerate}
		\item $f$ induces a weak homotopy equivalence $\cW\to U\cP^\simeq$.
		\item For every $x$ in $\cO$ and every $n\geq 0$, the induced map between fibers, $$ \Alg_{C_n}(\cO)^\cW_x\longrightarrow \Alg_{C_n}(\cP)^\simeq_{f(x)}$$ is a weak homotopy equivalence.
	\end{enumerate}
	Then $f\colon \cO\to\cP$ exhibits the localization of $\cO$ at $\cW$.
\end{prop}
\begin{proof}
	This follows from Proposition \ref{prop:corolla-wise_criterion I} applied to $f\colon (\cO,\cW)\to \cP^{\,\flat}$. Notice that $\cP^{\,\flat}$ readily verifies the hypotheses of Lemma \ref{lem:dendroidal Rezk nerve is Segal}.
\end{proof}

\begin{rem}
	\label{rem:partial_converse}
	Proposition \ref{prop:corolla-wise_criterion II} admits a partial converse. 
	More precisely, in the situation of Proposition \ref{prop:corolla-wise_criterion I}, suppose that condition (2) is satisfied.
	If $f$ is a localization, then condition (1) must be satisfied.
	This follows from Remark \ref{rem:complete}.
\end{rem}

 A useful observation is that Lemma \ref{lem:dendroidal Rezk nerve is Segal} and Propositions \ref{prop:corolla-wise_criterion I},  \ref{prop:corolla-wise_criterion II} have ``dual'' versions. The essential idea is to work with the fibers of the cartesian fibration
	$\ev_{\mathrm{leaves}}:\Alg_{C_n}(\mathcal{O})\to U\mathcal{O}^{\times n}$
	given by the evaluation at the leaves, instead of the fibers of $\ev_{\mathrm{root}}$. (See Proposition \ref{prop:root_cc} and Remark \ref{rem:restr_at_root}.) For the record and completeness, let us display the dual of Proposition \ref{prop:corolla-wise_criterion II}.
	
	\begin{prop}\label{prop:dual corolla-wise_criterion II}
		Let $f\colon \cO\to\cP$ be a map of $\infty$-operads and denote by $\cW$ the wide subcategory generated by $f^{-1}(U\cP^\simeq)$. Suppose the next two conditions are satisfied: 
	\begin{enumerate}
		\item $f$ induces a weak homotopy equivalence $\cW\to U\cP^\simeq$, and
		\item for each $n\geq 0$ and objects $x_1,\dots,x_n\in \cO$, the functor 
		$$ \Alg_{C_n}(\cO)^\cW_{(x_1,\dots,x_n)}\longrightarrow \Alg_{C_n}(\cP)^\simeq_{(f(x_1),\dots,f(x_n))}
		$$
		is a weak homotopy equivalence,
	\end{enumerate}
    where $\Alg_{C_n}(\cO)^\cW_{(x_1,\dots,x_n)}$ denotes the fiber at $(x_1,\dots,x_n)$ of the cartesian fibration  
    $$
    \ev_{\mathrm{leaves}}\colon\Alg_{C_n}(\cO)\!\,^\cW\longrightarrow \cW^{\times n}.
    $$
	Then $f\colon \cO\to\cP$ exhibits the localization of $\cO$ at $\cW$.
	\end{prop}

\begin{rem}\label{rem:Weak Approx}
	Proposition \ref{prop:corolla-wise_criterion II} gives an efficient
	proof of the main results of Lurie's and Harpaz's (technical) theories
	of ``weak approximation,'' which give sufficient conditions for
	a map of Lurie's $\infty$-operads to be a localization. 
	
	To describe their results, consider the following conditions for a
	map $f\from\cat O^{\t}\to\cat P^{\t}$ of Lurie's $\infty$-operads:
	\begin{itemize}
		\item [(1)]The functor $U\pr{\cat O^{\t}}\to U\pr{\cat P^{\t}}$ is a weak homotopy equivalence,
		where $U\pr{\cat P^{\t}}$ is a Kan complex.
		\item [(1$'$)]The functor $U\pr{\cat O^{\t}}\to U\pr{\cat P^{\t}}$ has weakly contractible fibers
		(in $\curlyCatinfty$).
		\item [(2)]For every object $x\in\cat O^{\t}_{\inp 1}$ and $m\geq 0$, the map
		\[
		f_{x}\from\pr{\cat O^{\t}_{\act}}^{/x}\times_{\Fin_{\ast}}\{\inp m\}\longrightarrow\pr{\cat P^{\t}_{\act}}^{/f\pr x}\times_{\Fin_{\ast}}\{\inp m\}
		\]
		has weakly contractible homotopy fibers	(in $\curlyCatinfty$). 
	\end{itemize}
	Lurie's result asserts that $f$ is a localization if conditions
	(1) and (2) are met \cite[Theorem 2.3.3.23 and Corollary 2.3.3.16]{lurie_higher_nodate},
	while Harpaz's result asserts that $f$ is a localization if conditions
	(1$'$) and (2) are met (\cite[Proposition 4.2.18]{harpaz_little_nodate}; see
	also \cite[Corollary C.16]{carmona_additivity_2025}). 
	
	To deduce these results from Proposition \ref{prop:corolla-wise_criterion II},
	we observe that condition (1$'$) is equivalent to the condition that
	the map $U\pr{\cat O^{\t}}\times_{U\pr{\cat P^{\t}}}U\pr{\cat P^{\t}}\!\,^{\simeq}\to U\pr{\cat P^{\t}}\!\,^{\simeq}$
	have contractible (homotopy) fibers. By Lemma \ref{lem:Quillen B},
	this is equivalent to saying that the latter map is a weak homotopy
	equivalence. We also observe that Proposition \ref{prop:C_n_model}
	identifies the map $f_{x}$ with $\Alg_{C_{m}}\pr{\cat O}_{x}\to\Alg_{C_{m}}\pr{\cat P}_{f\pr x}.$
	Hence, Proposition \ref{prop:corolla-wise_criterion II} recovers both Lurie's
	and Harpaz's results.
	
	Lurie's proof of his result is inherently tied to a specific model
	of $\infty$-operads and is rather technical, to the point that he
	advises readers to skip his proof. Harpaz's proof is built upon Lurie's
	argument and has a similar level of technicalities. In contrast,
	our argument is more synthetic, requires far less technical argument, and gives stronger results (such as Proposition \ref{prop:corolla-wise_criterion I} and Remark \ref{rem:partial_converse}). 
	
	We should mention, however, that Lurie's and Harpaz's weak approximation machinery can be generalized to the case where $\cat O^{\t}$ may not be an $\infty$-operad. This generalization is crucial in many situations, for instance, in Lurie's approach to the additivity of the little cubes operads, as well as its subsequent generalizations. Thus,
	all approaches have complementary advantages. 
\end{rem}

Let us conclude this section with a few obvious, but important, observations, which are the source of numerous examples of localizations. 

The following observation is immediate from the definitions.
\begin{prop}\label{prop:reflective_localization} Any reflective localization $\ell\dashv \iota$ in $\Op_{\infty}$ (resp.\ coreflective localization $\iota\dashv r$) exhibits $\ell$ (resp.\ $r$) as a localization of $\infty$-operads. More generally, any pair of  morphisms $f\colon (\cat O,\cat W_{\cat O})\rightleftarrows (\cat P,\cat W_{\cat P}):\!g$  in $\RelOp_{\infty}$ such that $gf$ (resp.\ $fg$) is connected to the identity through a zigzag of $2$-morphisms in $\RelOp_{\infty}$ whose components are in $\cat W_{\cat O}$ (resp.\ $\cat W_{\cat P}$) induces an equivalence $\cat O[\cat W^{-1}_{\cat O}]\simeq \cat P[\cat W^{-1}_{\cat P}]$.
\end{prop}

For the next one, let us call a map $L\from\cat O\to\cat P$ of $\infty$-operads
a \emph{universal localization} if for every pullback square

\[\begin{tikzcd}
	{\mathcal{O}'} & {\mathcal{O}} \\
	{\mathcal{P}'} & {\mathcal{P}}
	\arrow[from=1-1, to=1-2]
	\arrow["{L'}"', from=1-1, to=2-1]
	\arrow["L", from=1-2, to=2-2]
	\arrow[from=2-1, to=2-2]
\end{tikzcd}\]in $\Op_{\infty}$, the map $L'$ is again a localization. The utility
of universal localization can be seen, for instance, in identifying
suboperads of $\cat P$ as localization of suboperads of $\cat O$.
Also note that if $L$ is a universal localization, then $UL\from U\cat O\to U\cat P$
is again a localization, because it is a pullback of $L$.

\begin{rem}
	The underlying category functor $$ U\colon \Opinfty\to\Catinfty$$ does not commute with localization. Indeed, for a relative $\infty$-operad $(\cO,\cW)$, the universal property of categorical localization yields a map 
	\begin{equation}\label{eqt: Comparison of U-localizations}
	U(\cO)[\cW^{-1}] \longrightarrow U(\cO[\cW^{-1}]).
	\end{equation}  This is not always an equivalence, as it can be seen via the following example.
	
	Let $P$ be the operad given by the union of a tree obtained by attaching a nullary vertex to one leaf $l_2$ of $C_2$ and of a stump representing a nullary operation with the same color as the root $r$ of $C_2$. That is, $P$ is given by the pushout
	$$
	\begin{tikzcd}
		\begin{tikzpicture}
			\draw (0,-.5)--(0,-1) node [midway, right]{$r$};
			\end{tikzpicture} \ar[r, hookrightarrow, shift right=10] \ar[d,shift right=2.9, hookrightarrow] \ar[rd, phantom, "\ulcorner"' near end, shift right=4] & 
			\begin{tikzpicture}
		\draw (0,-1)--(0,-.5) node [midway, right]{$r$};
		\draw (0,-.5)--(0.25,0) node [midway, right]{$l_2$};
		\draw (0,-.5)--(-0.25,0) node [midway, left]{$l_1$};
		\fill (0, -.5) circle (1.5pt);
		\fill (0.25,0) circle (1.5pt);
			\end{tikzpicture} \ar[d, hookrightarrow]\\
			\begin{tikzpicture}
		\draw (0,0.35)--(0,-0.15) node [midway, right]{$r$};
		\fill (0,0.35) circle (1.5pt);
		\end{tikzpicture} \arrow[r, hookrightarrow] & 
		\begin{tikzpicture}
		\node at (0,0) {$P$};
	\end{tikzpicture}
	\end{tikzcd}
$$
%
%
%
%
%
%
%
%
%
in $\mathrm{Op}$. Let $f$ be the unique morphism $l_1 \to r$ in $P$. Then $$ U(P)[f^{-1}]\simeq \{l_2\}\sqcup J_{l_1,r}, $$ where $J_{l_1,r}$ is the groupoid completion of $l_1\to r$. On the other hand, in $U(P[f^{-1}])$ there also exists a morphism $l_2 \to r$, therefore the map (\ref{eqt: Comparison of U-localizations}) cannot be an equivalence.
\end{rem}

\begin{prop}\label{prop:univ_loc}
Let $\pr{\cat O,\cat W}$ be a relative $\infty$-operad. If $\dN\pr{\cat O,\cat W}$
is a complete dendroidal Segal space, then the localization $L\from\cat O\to\cat O[\cat W^{-1}]$
is a universal localization.
\end{prop}

Note that the hypothesis of Proposition \ref{prop:univ_loc} is satisfied, e.g., in
the situation of Proposition \ref{prop:corolla-wise_criterion II}
or its dual Proposition \ref{prop:dual corolla-wise_criterion II}.
\begin{proof}
Set $\cat P=\cat O[\cat W^{-1}]$, and consider a pullback square
\[\begin{tikzcd}
	{\mathcal{O}'} & {\mathcal{O}} \\
	{\mathcal{P}'} & {\mathcal{P}}
	\arrow[from=1-1, to=1-2]
	\arrow["{L'}"', from=1-1, to=2-1]
	\arrow["L", from=1-2, to=2-2]
	\arrow[from=2-1, to=2-2]
\end{tikzcd}\]in $\curlyOpinfty$. We claim that $L'$ exhibits the localization of $\cat O'$ at $\cat W'=\pr{U\cat P'}^{\simeq}\times_{\pr{U\cat P}^{\simeq}}\cat W$.
By Theorem \ref{thm:MG}, it suffices to prove the following: For
each tree $T$, the map
\[
\phi'\from\Alg_{T}\pr{\cat O'}\!\,^{\cat W'}\longrightarrow\Alg_{T}\pr{\cat P}^{\simeq}
\]
is a weak homotopy equivalence. For this, we consider the following
pullback square of $\infty$-categories:
\[\begin{tikzcd}
	{\operatorname{Alg}_{T}(\mathcal{O}')^{\mathcal{W}'}} & {\operatorname{Alg}_{T}(\mathcal{O})^{\mathcal{W}}} \\
	{\operatorname{Alg}_{T}(\mathcal{P}')^{\simeq}} & {\operatorname{Alg}_{T}(\mathcal{P})^{\simeq}.}
	\arrow[from=1-1, to=1-2]
	\arrow["{\phi'}"', from=1-1, to=2-1]
	\arrow["\phi", from=1-2, to=2-2]
	\arrow[from=2-1, to=2-2]
\end{tikzcd}\]By hypothesis and Theorem \ref{thm:MG}, the map $\phi$ is a weak
homotopy equivalence. It follows from Lemma \ref{lem:Quillen B}
that the map $\phi'$ is also a weak homotopy equivalence, as claimed.
\end{proof}

\section{Application I: Deriving tensor products}
\label{sec:modelcats}
Let $\mathbf{V}$ be a symmetric monoidal model category. The homotopy
category $\operatorname{ho}\pr{\mathbf{V}}$ admits a symmetric monoidal structure,
with tensor product given by the derived tensor product. Moreover,
there is a lax symmetric monoidal functor $\mathbf{V}\to\operatorname{ho}\pr{\mathbf{V}}$
that is initial among those that invert weak equivalences.

It is natural to expect that one can lift this to the world of $\infty$-categories.
More precisely, one would expect that the localization $\mathbf{V}\to\mathbf{V}[\weq^{-1}]$
at weak equivalences can be enhanced to a lax symmetric monoidal functor,
and moreover that it is initial among the lax symmetric monoidal functors
inverting weak equivalences of $\mathbf{V}$. This is indeed true,
as proved by Nikolaus and Scholze \cite[Theorem A.7]{nikolaus_topological_2018}.

In this section, we show that this result, as well as its generalization, is a direct consequence of Theorem
\ref{thm:MG}. In fact, our argument allows us to prove an additional assertion on universal
localization, which does not follow from Nikolaus--Scholze's argument (see Theorem \ref{thm:NS}). 

\begin{rem}
\label{rem:SMloc}What makes Nikolaus--Scholze's theorem non-trivial
is that, in a symmetric monoidal model category (Example \ref{exa:symmonmodcat}), weak equivalences
are generally \textit{not} stable under tensor product. Indeed, if
this stability is guaranteed, the theorem can be proved more easily
and in a greater generality. More precisely, let $\cat C$ be a symmetric
monoidal $\infty$-category, and let $\cat W\subset\cat C$ be a subcategory
stable under tensor product in $\cat C$. Then:
\begin{enumerate}
\item The $\infty$-categorical localization $\cat C[\cat W^{-1}]$ acquires
a symmetric monoidal structure, which is uniquely characterized by
the property that the localization $\cat C\to\cat C[\cat W^{-1}]$
can be enhanced to a symmetric monoidal functor. (See \cite[Proposition 4.1.7.4]{lurie_higher_nodate}.)
\item The resulting map $\cat C^{\t}\to\cat C[\cat W^{-1}]^{\t}$ of $\infty$-operads
is a localization of $\infty$-operads. (See \cite[\href{https://kerodon.net/tag/02LW}{Tag 02LW}]{lurie_kerodon_nodate}.)
\end{enumerate}

While this is a very elementary observation, it gives us an important
lesson: When one tries to localize symmetric monoidal $\infty$-categories,
we might as well localize them on the level of $\infty$-operads.
This lesson shapes our approach in this section.
\end{rem}

Our argument does not use the full axioms of symmetric monoidal model categories, and so it is unnatural to restrict to this special class of examples. (In this regard, we should also say that Nikolaus--Scholze's argument does not use all of the axioms either.) Because of this, we make the following definition, which gives all we need for the purpose of this section:

\begin{defn}
A \emph{symmetric monoidal $\infty$-category with cofibrations and weak equivalences} consists of the following data:

\begin{itemize}
\item A symmetric monoidal $\infty$-category $\cat{C}$.
\item Two subcategories of $\cat{C}$ that makes $\cat{C}$ into an $\infty$-category with weak equivalences and cofibrations in the sense of \cite[Definition 4.6.12]{cisinski_higher_2019}. We assume functorial factorization and saturation of weak equivalences.
\end{itemize}
These data are required to satisfy the following conditions:
\begin{enumerate}[label= (\alph*)]
	\label{def:symmonweqcof}
\item Cofibrant objects are stable under finite tensor product.
(In particular, the monoidal unit is cofibrant.)
\item Weak equivalences of cofibrant objects are stable under finite tensor product.
\end{enumerate}
\end{defn}

\begin{rem}
	\label{rem:brown}
Let $\cat{C}$ be a symmetric monoidal $\infty$-category, and suppose that $\cat{C}$ is equipped with the structure of an $\infty$-category with weak equivalences and cofibrations.
By Ken Brown's lemma \cite[Corollary 7.4.14]{cisinski_higher_2019}, condition (b) of Definition \ref{def:symmonweqcof} is equivalent to the following (seemingly weaker) condition:
\begin{itemize}
	\item [(b$'$)]For every cofibrant object $X\in \cat{C}$, the functor $X\otimes - \from \cat{C}\to \cat{C}$ carries trivial cofibrations of cofibrant objects to weak equivalences.
\end{itemize}
In practice, condition (b$'$) is generally easier to verify than condition (b).
\end{rem}

\begin{example}
\label{exa:symmonmodcat}Recall that a \emph{symmetric monoidal model category}
is a closed symmetric monoidal category $\mathbf{V}$ with a model
structure, satisfying the following pair of conditions:
\begin{enumerate}
\item The tensor bifunctor $\otimes\from\mathbf{V}\times\mathbf{V}\to\mathbf{V}$
is a left Quillen bifunctor.
\item The unit object is cofibrant.
\end{enumerate}
Every symmetric monoidal model category is naturally a symmetric monoidal $\infty$-category with weak equivalences and cofibrations, simply by forgetting the information of fibrations. This follows from Remark \ref{rem:brown}.
\end{example}

\begin{rem}
	\label{rem:cofibrantunit}
In some literature, the unit object of a symmetric monoidal model
category is not required to be cofibrant. However, we can almost always
change the class of cofibrations to make the unit object cofibrant,
without affecting the class of weak equivalences  nor the compatibility of the tensor product with (acyclic) cofibrations \cite{muro_unit_2015}.
\end{rem}

Here is the main result of this section.
\begin{thm}
\label{thm:NS} Let $\cat{C}$
be a symmetric monoidal $\infty$-category with weak equivalences and cofibrations, and let $\cat{C}_c\subset \cat{C}$ be the full symmetric monoidal subcategory of cofibrant objects. Then the following statements hold:
\begin{enumerate}
	\item The map $\cat{C}^{\t}_{c}\to\cat{C}^{\t}$ is a local equivalence
of $\infty$-operads.
\item The operadic localization $L\pr{\cat{C}^{\t}}$ of $\cat{C}^{\t}$
at weak equivalences is a symmetric monoidal $\infty$-category.
\item The operadic localization $\cat{C}^{\t}\to L\pr{\cat{C}^{\t}}$
is a universal localization.
\end{enumerate}
\end{thm}

\begin{proof}
	We start with (1). Write $\cat{W}\subset \cat{C}$ for the subcategory
	of weak equivalences, and set $\cat{W}_c=\cat{C}_c\cap \cat{W}$. By Theorem \ref{thm:MG} (and by enlarging the universe if necessary), it suffices to show
that the map
\[
	\dN\pr{\cat{C}^{\t}_{c},\cat{W}_c} \longrightarrow\dN\pr{\cat{C}^{\t},\cat{W}}
\]
is an equivalence of dendroidal spaces. In other words, it suffices
to show that for each tree $T$, the inclusion
\[
	\iota\from\Alg_{T}\pr{\cat{C}_{c}}\!\,^{\cat{W}_c}\xhookrightarrow{\quad}\Alg_{T}\pr{\cat{C}}\!\,^{\cat{W}}
\]
is a weak homotopy equivalence. For notational convenience, we will
prove this in the case where $T$ is a corolla; the general case can
be treated in a similar way, by decomposing trees into corollas (see \cite[Proposition 4.3.4]{hinich_equivalence_2024}).

Suppose that $T=C_{n}$. We will show that $\iota$ is a weak homotopy
equivalence by showing that $\abs{\iota}$ admits a homotopy inverse.
By Corollary \ref{cor:C_n_pb}, the $\infty$-category $\Alg_{T}\pr{\cat{C}}^{\cat{W}}$ is equivalent to the pullback of the cospan
\[
	\cat{C}^{\times n} \xlongrightarrow{\,\;\bigotimes_{i=1}^n\;\,} \cat{C}\xlongleftarrow[\phantom{\,\;\bigotimes_{i=1}^n \;\,}]{\ev_1}\Fun([1],\cat{C}).
\]
Thus, its objects can be identified with tuples of the form $\pr{x_{1},\dots,x_{n},y,f}$,
where $x_{i}$ and $y$ are objects of $\cat{C}$ and $f$ is a
morphism $x_{1}\otimes\cdots\otimes x_{n}\to y$. Via this description, morphisms $\pr{x_{1},\dots,x_{n},y,f}\to\pr{x'_{1},\dots,x'_{n},y',f'}$
are given by weak equivalences $\{a_{i}\from x_{i}\xrightarrow{\sim}x'_{i}\}_{1\leq i\leq n}$
and $b\from y\xrightarrow{\sim}y'$ rendering the diagram 
\[\begin{tikzcd}
	{x_1\otimes \cdots\otimes x_n} & y \\
	{x'_1\otimes \cdots\otimes x'_n} & {y'}
	\arrow["f", from=1-1, to=1-2]
	\arrow["{a_1\otimes \cdots \otimes a_n}"', from=1-1, to=2-1]
	\arrow["b", from=1-2, to=2-2]
	\arrow["{f'}"', from=2-1, to=2-2]
\end{tikzcd}\]commutative. The functor $\iota$ is the inclusion of the full subcategory
spanned by the objects $\pr{x_{1},\dots,x_{n},y,f}$ such that $x_{i}$
and $y$ are cofibrant. With this in mind, we define a functor $\rho\from\Alg_{T}\pr{\cat{C}}\!\,^{\cat{W}}\to\Alg_{T}\pr{\cat{C}_{c}}\!\,^{\cat{W}_c}$
as follows: Given an object $\pr{x_{1},\dots,x_{n},y,f}$, we functorially
replace each $x_{i}$ by a cofibrant object $x^{c}_{i}$, and then
functorially factor the map $x^{c}_{1}\otimes\cdots\otimes x^{c}_{n}\to y$
as a cofibration followed by a weak equivalence. The functors $\rho$
and $\iota$ become inverse equivalences upon geometric realization,
so this shows that $\iota$ is a weak homotopy equivalence, as
desired.

Next, part (2) is immediate from part (1) and Remark \ref{rem:SMloc}.

Finally, to prove part (3), we use Remark \ref{rem:complete} and
Proposition \ref{prop:univ_loc}. According to these results and the proof of (1), it is enough to show that $\dN\pr{\cat{C}^{\t}_{c},\cat{W}_c}$
is a dendroidal Segal space.
According to the dual of Lemma \ref{lem:dendroidal Rezk nerve is Segal}
(see the paragraph preceding Proposition \ref{prop:dual corolla-wise_criterion II}) and Corollary \ref{cor:C_n_pb}, it suffices to show that for every finite
(possibly empty) collection of weak equivalences of cofibrant objects
$\{x_{i}\xrightarrow{\sim}y_{i}\}_{i}$, the induced map
\[
	\pr{\cat{C}_{c}}^{\cat{W}_c}_{\bigotimes_{i}y_{i}/}\longrightarrow\pr{\cat{C}_c}^{\cat{W}_c}_{\bigotimes_{i}x_{i}/}
\]
is a weak homotopy equivalence. For this, let $\gamma\from \cat{C}\to \cat{C}[\cat{W}^{-1}]$ denote the (categorical, not operadic) localization of $\cat{C}$. 
By \cite[Corollary 7.6.9 and Corollary 7.6.13]{cisinski_higher_2019}, we can identify $\abs{\pr{\cat{C}_{c}}^{\cat{W}_c}_{\bigotimes_{i}x_{i}/}}$ with the core of $\cat{C}[\cat{W}^{-1}]_{\gamma(\bigotimes_{i}x_i)/}$, and a similar identification exists for the $y_i$'s.
The claim then follows from the fact that the map $\gamma(\bigotimes_{i}x_i)\to \gamma(\bigotimes_{i}y_i)$ is an equivalence.
\end{proof}

\begin{rem}
In \cite[Theorem A.7 (4)]{nikolaus_topological_2018}, Nikolaus and Scholze
proved a slightly stronger claim using Lurie's $\infty$-operads (and only for symmetric monoidal model categories).
Our argument does not give this stronger claim, but this is not a defect, as the corresponding
claim only makes sense in the language of Lurie's $\infty$-operads.
To the authors' knowledge, the stronger result is only used to show
that the functor $U\pr{\cat{C}^{\t}}\to U\pr{L\pr{\cat{C}^{\t}}}$
is a localization, and this follows from part (1) or part (3) of our
Theorem \ref{thm:NS} anyway.
\end{rem}

\begin{rem}
In the proof of Theorem \ref{thm:NS}, we saw that for any  symmetric monoidal model category $\mathbf{V}$, its relative dendroidal Rezk nerve $\dN\pr{\mathbf{V},\weq}$
is a complete dendroidal Segal space. This can be seen as a dendroidal analog 
of the corresponding results for the relative Rezk nerve of model categories, recorded in \cite[Theorem 8.3]{rezk_model_2001} and \cite{bergner_complete_2009}.
\end{rem}

\section{Application II: Calculus of fractions}
\label{sec:CLFs}
In their influential work \cite{gabriel_calculus_1967}, Gabriel and Zisman introduced
a particularly well-behaved class of relative categories. These relative
categories are said to admit calculus of left or right fractions,
and their localizations can be described very explicitly. Motivated by applications in mathematical physics, an operadic extension of calculus of (left) fractions was developed in \cite[Appendix A]{benini_equivalence_2026}. However, their argument is not quite conceptual, and relies on a specific model of $\infty$-operads. 

In this section, we revisit operadic calculus of fractions from the perspective of Theorem \ref{thm:MG}. As we will see, this point of view allows for a much shorter and more synthetic proof of both left and right calculus of fractions.

We start with the definition of calculus of fractions. We call \emph{relative operad} a relative $\infty$-operad whose underlying $\infty$-operad is an ordinary operad.
\begin{defn}
\label{def:calc}
A relative operad $\pr{\cat O,\cat W}$ is said to
admit an \emph{operadic calculus of left fractions} if it satisfies
the following conditions:

\begin{enumerate}[label=(CL\arabic*)]

\item $\cat W$ has the $2$-out-of-$3$ property.\footnote{This assumption can be relaxed, but we will stick to this stronger version because it simplifies some aspects in the sequel, and since it can always be enforced.}

\item For every multimorphism $\mu\from\pr{x_{i}}_{i\in I}\to y$ in $\cat O$ and
	weak equivalences $\{w_{i}\from x_{i}\xrightarrow{\sim} x'_{i}\}_{i\in I}\subseteq \cat W$, there
is a multimorphism $\mu' \from\pr{x'_{i}}_{i\in I}\to y'$ and a
weak equivalence $w:y\xrightarrow{\sim} y'$ rendering the diagram
\[\begin{tikzcd}
	{(x_i)_{i\in I}} &[6mm] {(x'_{i})_{i\in I}} \\
	y & {y'}
	\arrow["{(w_i)_{i\in I}}", from=1-1, to=1-2]
	\arrow["\mu"', from=1-1, to=2-1]
	\arrow["{\mu'}", dashed, from=1-2, to=2-2]
	\arrow["w"', dashed, from=2-1, to=2-2]
\end{tikzcd}\]commutative.

\item For every finite collection $\{w_{i}\from x'_{i}\xrightarrow{\sim} x{}_{i}\}_{i\in I}\subseteq \cat W$
of weak equivalences and every pair of multimorphisms $\mu,\mu'\from\pr{x_{i}}_{i\in I}\rightrightarrows y$ in $\cat O$
such that $\mu\circ\pr{w_{i}}_{i\in I}=\mu'\circ\pr{w_{i}}_{i\in I}$,
there is a weak equivalence $w\from y\xrightarrow{\sim} y'$ such that $w\circ\mu=w\circ\mu'$:
\[
\begin{tikzcd}
\pr{x'_{i}}_{i\in I} \ar[r,"\pr{w_{i}}_{i\in I}"] &[2mm] \pr{x_{i}}_{i\in I} \ar[r, shift left=1.2,"\mu"] \ar[r, shift right=1.2,"\mu'"'] &[2mm] y \ar[r, dashed,"w"] &[2mm] y'
\end{tikzcd}.
\]

\end{enumerate}

There is also a ``dual'' notion of right fractions. More precisely,
we say that $\pr{\cat O,\cat W}$ admits an \emph{operadic  calculus of right
 fractions} if it satisfies the following conditions:

\begin{enumerate}[label=(CR\arabic*)]

\item $\cat W$ has the $2$-out-of-$3$ property.

\item For every multimorphism $\mu\from\pr{x_{i}}_{i\in I}\to y$ in $\cat O$ and
	every weak equivalence $w\from y'\xrightarrow{\sim} y$, there is a multimorphism $\mu'\from\pr{x'_{i}}_{i\in I}\to y'$
	and weak equivalences $\{w_{i}\from x'_{i}\xrightarrow{\sim} x{}_{i}\}_{i\in I}$
rendering the diagram
\[\begin{tikzcd}
	{(x'_i)_{i\in I}} &[6mm] {(x_{i})_{i\in I}} \\
	{y'} & y
	\arrow["{(w_i)_{i\in I}}", dashed, from=1-1, to=1-2]
	\arrow["{\mu'}"', dashed, from=1-1, to=2-1]
	\arrow["\mu", from=1-2, to=2-2]
	\arrow["w"', from=2-1, to=2-2]
\end{tikzcd}\]commutative.

\item For every pair of multimorphisms $\mu,\mu'\from\pr{x_{i}}_{i\in I}\rightrightarrows y$ in $\cat O$
	and every weak equivalence $w:y\xrightarrow{\sim} y'$ such that $w\circ\mu=w\circ\mu'$,
	there are weak equivalences $\{w_{i}\from x'_{i}\xrightarrow{\sim} x{}_{i}\}_{i\in I}$
such that $\mu\circ\pr{w_{i}}_{i\in I}=\mu'\circ\pr{w_{i}}_{i\in I}$:
\[
\begin{tikzcd}
	\pr{x'_{i}}_{i\in I} \ar[r, dashed, "\pr{w_{i}}_{i\in I}"] &[2mm] \pr{x_{i}}_{i\in I} \ar[r, shift left=1.2,"\mu"] \ar[r, shift right=1.2,"\mu'"'] &[2mm] y \ar[r, "w"] &[2mm] y'
\end{tikzcd}.
\]

\end{enumerate}
\end{defn}

In the presence of calculus of fractions, we can construct a new operad whose morphisms are ``fractions'':
\begin{defn}\label{defn: Operad of left fractions}
Let $\pr{\cat O,\cat W}$ be a relative operad admitting an operadic
calculus of left fractions. Let $\cat W^{-1}\cat O$ be the (ordinary) operad defined as follows:
\begin{enumerate}
\item The objects of $\cat W^{-1}\cat O$ are the objects of $\cat O$.
\item Given a finite collection of objects $\pr{x_{i}}_{i\in I}$ and another
object $y$, the hom-set is defined as the colimit (of sets)
\[
\cat W^{-1}\cat O\binom{\pr{x_{i}}_{i\in I}}{y}=\underset{y'\in\cat W_{y/}}{\colim}\,\cat O\binom{\pr{x_{i}}_{i\in I}}{y'}.
\]
Thus, an operation $\pr{x_{i}}_{i\in I}\to y$ is an equivalence class
of cospans
\[
\pr{x_{i}}_{i\in I}\xrightarrow{\;\;\phantom{\sim}\;\;} y'\xleftarrow{\;\;\sim\;\;}y
\]
whose right-hand arrow is a weak equivalence.
\item Composition is defined as in the case of ordinary calculus of left fractions \cite[2.3]{gabriel_calculus_1967}.
\end{enumerate}
Similarly, if $\pr{\cat O,\cat W}$ is a relative operad admitting
an operadic calculus of right fractions, one can define a new operad
$\cat O\cat W^{-1}$.

\end{defn}

It is immediate from the definitions that $\cat{W}^{-1}\cat{O}$ enjoys the universal property of localization in the world of ordinary operads. 
The main result of this section says that, it is in fact a localization in the world of $\infty$-operads:

\begin{thm}
\label{thm:lcalc}Let $\pr{\cat O,\cat W}$ be a relative operad admitting
a calculus of left fractions. The map $\cat O\to\cat W^{-1}\cat O$
exhibits $\cat W^{-1}\cat O$ as the ($\infty$-operadic) localization
of $\cat O$ at $\cat W$. A similar claim holds for the calculus
of right fractions, too.
\end{thm}

The proof of Theorem \ref{thm:lcalc} relies on the evaluation-at-leaves criterion developed in Section \ref{sec:criteria} and the following well-known special case of relative categories. This latter appears in \cite[Proposition 7.3]{dwyer_calculating_1980}, but for the reader's convenience,
we record an alternative proof here.
\begin{lem}
\label{lem:lcalc}Let $\pr{\cat C,\cat W}$ be a relative category
admitting a calculus of left fractions (in the sense of Definition
\ref{def:calc}). Then the functors $\cat C\to\cat W^{-1}\cat C$
and $\cat W\to\pr{\cat W^{-1}\cat C}\!\,^{\simeq}$ are  ($\infty$-categorical)
localizations.
\end{lem}

\begin{proof}
The claim for the second functor is a special case of the first one,
because we can identify $\pr{\cat W^{-1}\cat C}\!\,^{\simeq}$ with $\cat W^{-1}\cat W$.
(This needs the $2$-out-of-$3$ property of $\cat W$.) It will therefore
suffice to prove the first claim. 

Let $X,Y\in\cat C$ be a pair of objects. The definition of calculus
of fractions ensures that the category $\cat W_{Y/}$ is filtered.
Since weak homotopy equivalences of simplicial sets are stable under
filtered colimits, it follows that the defining colimit
\[
\cat W^{-1}\cat C\pr{X,Y}=\underset{Y'\in\cat W_{Y/}}{\colim}\,\cat C\pr{X,Y'}
\]
of the hom-set of $\cat W^{-1}\cat C$ computes the colimit in $\cat S$.
On the other hand, the derived mapping space lemma \cite[Theorem 2.2]{arakawa_derived_2025}
asserts that the above colimit is the mapping space of the localization.
Hence, the functor $\cat C[\cat W^{-1}]\to\cat W^{-1}\cat C$ is an
equivalence, as claimed.
\end{proof}

\begin{proof}
[Proof of Theorem \ref{thm:lcalc}] We will prove the claim for the calculus of left fractions; the case for the calculus of right fractions can be
treated similarly. According to Proposition \ref{prop:dual corolla-wise_criterion II}, it suffices to prove the following:
\begin{enumerate}
\item The map $\cat W\to U\pr{\cat W^{-1}\cat O}\!\,^{\simeq}$ is a weak homotopy
equivalence.
\item For each $n\geq0$ and objects $x_{1},\dots,x_{n}\in\cat O$, the
functor
\[
\Alg_{C_{n}}\pr{\cat O}^{\cat W}_{\pr{x_{1,}\dots,x_{n}}}\to\Alg_{C_{n}}\pr{\cat W^{-1}\cat O}^{\simeq}_{\pr{x_{1},\dots,x_{n}}}
\]
is a weak homotopy equivalence.
\end{enumerate}
Here, we are denoting by $\Alg_{C_{n}}\pr{\cat O}^{\cat W}_{\pr{x_{1,}\dots,x_{n}}}$ 
the fiber of the cartesian fibration $\ev_{\mathrm{leaves}}\colon \Alg_{C_{n}}\pr{\cat O}\!\,^{\cat W}\to\cat W^{\times n}$
over $\pr{x_{1},\dots,x_{n}}$, given by the evaluation at the leaves
of $C_{n}$. Both of these claims are special cases of Lemma \ref{lem:lcalc},
because $U\cat O$ and $\Alg_{C_{n}}\pr{\cat O}_{\pr{x_{1},\dots,x_{n}}}$
admit calculus of left fractions, with weak equivalences given by
those maps whose images in $U\cat O$ are weak equivalences.
\end{proof}

\begin{example}\label{exa:reflective}
	Any reflective localization
	\[
	\begin{tikzcd}
		\ell\from \cat O \ar[r, shift left=1.8]\ar[r,hookleftarrow, shift right=1.8,"\scalebox{0.8}{$\perp$}"] & \cat P\from \iota
	\end{tikzcd}
	\]
   in $\mathrm{Op}$ (see Proposition \ref{prop:reflective_localization}) yields a relative operad $(\cat O,\cat W)$ which admits an operadic calculus of left fractions, where $\cat W=\ell^{-1}(U\cat P)^{\simeq} $. Dually, any coreflective localization $\iota\dashv r$ in $\mathrm{Op}$ yields an example of operadic calculus of right fractions.
\end{example}

The following is a particular instance of Example \ref{exa:reflective} which is related to the material in Section \ref{sec:modelcats}.

\begin{example} Let $R$ be a (unital) commutative ring. Denote by $\mathrm{Mod}_R$ the associated abelian category of $R$-modules and by $\mathrm{Proj}_R$ its full subcategory of projective $R$-modules. Let us also denote by $\mathrm{K}_+(R)=\mathrm{K}_+(\mathrm{Mod}_R)$ and $\mathrm{D}_+(R)$ the associated homotopy and derived categories of $R$ consisting of chain complexes $X_\bullet$ such that $X_{\ll 0}=0$. Recall that the derived category $\mathrm{D}_+(R)$ was first constructed by Verdier as a categorical localization of $(\mathrm{K}_+(R),\mathrm{qis})$, where $\mathrm{qis}\subset \mathrm{K}_+(R)$ denotes the wide subcategory spanned by quasi-isomorphisms. Technically, back in Verdier's days, it was important to localize the homotopy category and not the abelian category of chain complexes over $R$, since the pair $(\mathrm{K}_+(R),\mathrm{qis})$ admits a categorical calculus of right fractions, while this is not the case for $(\mathrm{Ch}_+(\mathrm{Mod}_R),\mathrm{qis})$.\footnote{A modern point of view clearly explains why: If we have already inverted the homotopy equivalences in a $1$-categorical fashion, inverting quasi-isomorphisms will not create higher-categorical information. In contrast, inverting quasi-isomorphisms without this truncation step produces a stable $\infty$-category with non-trivial higher-categorical information.} 
In fact, we can justify this admissibility by looking at the categorical localization $\mathrm{K}_+(R)\to \mathrm{D}_+(R)$ as a coreflective localization
	\[
\begin{tikzcd}
	\iota\from \mathrm{K}_+(\mathrm{Proj}_R)  \ar[r, shift right=1.8,"\scalebox{0.8}{$\perp$}", leftarrow] & \ar[l, hookleftarrow, shift right=1.8] \mathrm{K}_+(R)\from r
\end{tikzcd},
    \]
   since $r$ induces an equivalence of categories $\mathrm{K}_+(\mathrm{Proj}_R)\simeq \mathrm{D}_+(R)$. The reason behind this last equivalence is that any quasi-isomorphism between objects in $\mathrm{K}_+(\mathrm{Proj}_R)$ is a homotopy equivalence.
   
   To transition into the operadic world, equip both categories with the tensor product $\otimes_R$. We then have a coreflective localization 
   $$
   \begin{tikzcd}
   	\iota\from \mathrm{K}_+(\mathrm{Proj}_R)^{\otimes_R}  \ar[r, shift right=1.8,"\scalebox{0.8}{$\perp$}", leftarrow] & \ar[l, hookleftarrow, shift right=1.8] \mathrm{K}_+(R)^{\otimes_R}\from r
   \end{tikzcd}
   $$
   in $\mathrm{Op}$. This implies that $r\colon \mathrm{K}_+(R)^{\otimes_R}\to \mathrm{K}_+(\mathrm{Proj}_R)^{\otimes_R}$ exhibit the operadic ($\infty$-)localization of $\mathrm{K}_+(R)^{\otimes_R}$ at $r^{-1}(\mathrm{K}_+(R)^{\simeq})=\mathrm{qis}$, and that $(\mathrm{K}_+(R)^{\otimes_R},\mathrm{qis})$ admits an operadic calculus of right fractions. This provides a presentation of the operad underlying the symmetric monoidal category given by $\mathrm{D}_+(R)$ with its derived tensor product $\otimes_R^{\mathbb{L}}$ in exactly the same terms as Verdier's original construction of the derived category $\mathrm{D}_+(R)$ from $\mathrm{K}_+(R)$. 
\end{example}

The next section contains instances of operadic calculus of left fractions not arising from reflective localizations in $\mathrm{Op}$. See Corollary \ref{cor: CLFs for Lorentzian pFAs} and Remark \ref{rem: more on CLFs for Lorentzian operads}.

\section{Application III: Prefactorization algebras}
\label{sec:prefact}
Prefactorization algebras are a very general class of algebraic objects of geometric origin. Their importance stems from their interesting applications in a variety of contexts, ranging from mathematical physics, algebraic and differential topology, to representation theory (see, e.g., \cite{benini_c-categorical_2026,costello_factorization_2017,costello_factorization_2021,karlsson_assembly_2026}). A useful formalization comes from the notion of \textit{orthogonal categories}. An orthogonal category is given by a pair $\mathrm{C}^{\perp}=(\mathrm{C},\perp)$ where $\mathrm{C}$ is an ordinary (small) category and $\perp$ is a suitable binary relation on $\mathrm{C}_1\times_{t,\mathrm{C}_0,t}\mathrm{C}_1$, i.e., the collection of arrows in $\mathrm{C}$ with the same target (see \cite[Definition 3.4]{benini_operads_2021} or \cite[\S 2.1]{benini_strictification_2023}). 
In fact, orthogonal categories can be organized in a 2-category $\mathrm{OrthCat}$ having as $1$-morphisms the orthogonal functors, i.e., functors respecting the binary relation, and natural transformations as $2$-morphisms (see \cite[\S 2.1]{benini_strictification_2023}). Associated with an orthogonal category $\mathrm{C}^{\perp}$ we have an operad $\cP_{\mathrm{C}}^{\perp}$, or simply $\cP_{\mathrm{C}}$ if the omission of $\perp$ causes no confusion, whose algebras are $\mathrm{C}$-\textit{prefactorization algebras} with values in $\cV$, $\mathrm{pFA}_{\mathrm{C}}(\cV):=\Alg_{\cP_{\mathrm{C}}}(\cV)$. (Actually, $\mathrm{C}^{\perp}\mapsto \cP_{\mathrm{C}}^{\perp}$ yields a $2$-functor $\mathrm{OrthCat}\to \mathrm{Op}$.)  
A detailed description of $\cP_{\mathrm{C}}$ can be found in \cite[\textsection 10.3]{yau_homotopical_2020} or \cite[Definition 2.5]{benini_categorification_2021}. Since it will be crucial later on, let us note that (1) the underlying category of $\cP_{\mathrm{C}}$ is always $\mathrm{C}$ (i.e.\ $U\cP_{\mathrm{C}}=\mathrm{C}$), and (2) if the underlying category $\mathrm{C}$ is a poset, we have:  
$$
\cP_{\mathrm{C}}\binom{U_1,\dots,U_n}{V}=\begin{cases}
	*=\{\mu_V^{U_{\bullet}}\} & \text{ if }U_i\perp_V U_j \text{ for any }i\neq j,\\[3mm]
	\diameter & \text{ else},
\end{cases}
$$
for any $U_1,\dots,U_n,V\in \mathrm{C}$. (Under restriction (2), the permutation actions and the composition maps of $\cP_{\mathrm{C}}$ are uniquely determined.)

This section is dedicated to studying different localization problems for operads of the form $\cP_\mathrm{C}$, where $\mathrm{C}^{\perp}$ is an orthogonal category with a Lorentzian or lattice flavor. 

\begin{notation}\label{notat: locally constant pFAs} Since the $\infty$-categories of algebras over the operads $\cP_{\mathrm{C}}$ will appear quite frequently in this subsection, we will opt for using the simpler notation 
	$$
	\mathrm{pFA}_{\mathrm{C}}(\cV):=\Alg_{\cP_{\mathrm{C}}}(\cV).
	$$
Furthermore, when the operad $\cP_{\mathrm{C}}$ is equipped with its maximal relative marking, $\cP_{\mathrm{C}}^{\,\sharp}=(\cP_{\mathrm{C}},\mathrm{C})$, we will denote 
the full subcategory of  $\cP_{\mathrm{C}}$-algebras which send any morphism in $\mathrm{C}$ to an equivalence by
$$
\mathrm{pFA}_{\mathrm{C}}^{\mathrm{lc}}(\cV)=\Alg_{\cP_{\mathrm{C}}}^{\mathrm{C}}(\cV).
$$
The superscript $\mathrm{lc}$ stands for \textit{locally constant}. This notation creates some conflicts with standard references such as \cite{ayala_factorization_2015,costello_factorization_2017} if applied to any orthogonal category $\mathrm{C}^{\perp}$. (In these references, local constancy is related to the behavior of prefactorization algebras with respect to isotopy equivalences between disjoint unions of open disks in an ambient smooth or topological manifold.) However, in the examples below, none of these conflicts will arise and that is why we will stick to our simple convention. The only exception is when we invoke locally constant \emph{factorization algebras} over $\mathbb{S}^{n-1}$ in Proposition \ref{prop: Second Lattice Localization}. Since they are not just ``pre-factorizing'' in that case, we hope this will cause no confusion.  
\end{notation}

\subsection{Lattice localizations} In this subsection, we focus on the comparison between prefactorization algebras arising from discrete geometries --in particular infinite lattices $\mathbb{Z}^n$-- and their continuous counterparts.

\begin{defn}\label{defn: Cone Orthogonal cats} Let $\mathrm{Cone}(\mathbb{R}^n)^{\perp}$ be the orthogonal category given by the following data:
	\begin{itemize}
		\item The underlying category $\mathrm{Cone}(\mathbb{R}^n)$ is the poset of open cones in $\mathbb{R}^n$. That is, open subsets of the form 
		\begin{align*}
			C_{(p,t,\alpha)} &= \left\{x\in \mathbb{R}^n\backslash\{p\}:\quad\mathrm{angle}(x-p,t)<\alpha \right\}\\
			&= \left\{ x\in \mathbb{R}^n: \quad (x-p)\cdot t>\vert\vert x-p\vert\vert \mathrm{cos}(\alpha)\right\}\subseteq \mathbb{R}^n \quad,
		\end{align*}
	    where $p\in \mathbb{R}^n$ is the apex of the cone, $t\in \mathbb{S}^{n-1}\subseteq \mathbb{R}^n$ is the normalized center direction, and $\alpha \in (0,\pi)$ is the opening angle. Here we denote $\cdot$ and $\vert\vert-\vert\vert$ the standard inner product and its associated norm in $\mathbb{R}^n$.\footnote{Open cones in $\mathbb{R}^n$ are faithfully parametrized by triples $(p,t,\alpha)$ if $\alpha\neq \pi/2$. When $\alpha=\pi/2$, there is a mild redundancy in the choice of apex, since those cones correspond to half-spaces.}
		
		\item The orthogonality relation is just disjointness, i.e.\ $U\perp_{V} U'$ if and only if $U\cap U'=\diameter$.

	\end{itemize}	

We will denote by $\mathrm{Cone}_{\circ}(\mathbb{R}^n)^{\perp}$ the full orthogonal subcategory of $\mathrm{Cone}(\mathbb{R}^n)^{\perp}$ whose objects are cones centered at the origin, i.e., open cones $C_{(p,t,\alpha)}$ such that $p=0$.
\end{defn}  

The orthogonal category  $\mathrm{Cone}(\mathbb{Z}^n)^{\perp}$ is defined exactly as $\mathrm{Cone}(\mathbb{R}^n)^{\perp}$ but using the subposet of parts of $\mathbb{Z}^n$ spanned by subsets $\overline{V}\subseteq \mathbb{Z}^n$ of the form $C_{(p,t,\alpha)}\cap\mathbb{Z}^n$, where $C_{(p,t,\alpha)}\subseteq \mathbb{R}^n$ is an open cone. Also, it is clear that we have an orthogonal functor $\delta_{\mathbb{Z}}\colon \mathrm{Cone}(\mathbb{R}^n)^{\perp}\to \mathrm{Cone}(\mathbb{Z}^n)^{\perp}$ induced by the rule $U\mapsto U\cap\, \mathbb{Z}^n$. Accordingly, there is an operad map between the associated prefactorization operads which we denote by the same symbol.

\begin{rem}\label{rem: Cones, fixed direction and angle} For an open cone $C_{(p,t,\alpha)}\subseteq \mathbb{R}^n$, its restriction to the lattice fixes both $t\in \mathbb{S}^{n-1}$ and $\alpha\in (0,\pi)$. That is, if there is another open cone $C_{(p',t',\alpha')}\subseteq \mathbb{R}^n$ such that $C_{(p,t,\alpha)}\cap \mathbb{Z}^n=C_{(p',t',\alpha')}\cap \mathbb{Z}^n$, then $t=t'$ and $\alpha=\alpha'$. Of course, there are also restrictions on the apex, but they are more involved (they also depend on the opening angle $\alpha\in (0,\pi)$). Since we will not need them, we omit their discussion. Also notice that any two open cones which share their direction $t$ and angle $\alpha$ have non-empty intersection (e.g., it contains the ray $\{\lambda t\in \mathbb{R}^n:\; \lambda>\lambda_0\}$ for sufficiently big $\lambda_0\gg0$).
\end{rem}

In light of the previous remark, we define $\cW_{\mathbb{R}}$ to be the wide subcategory  of inclusions in $\mathrm{Cone}(\mathbb{R}^n)$ between cones with the same normalized direction $t\in \mathbb{S}^{n-1}$ and opening angle $\alpha\in (0,\pi)$. We will use the notation $U\xhookrightarrow{\;\infty\;} U'$ to represent that an inclusion $U\subseteq U'$ belongs to $\cW_{\mathbb{R}}$. We will employ the same notation for the wide subcategory  $\cW_{\mathbb{Z}}=\delta_{\mathbb{Z}}(\cW_{\mathbb{R}})$ of $\mathrm{Cone}(\mathbb{Z}^n)$. Associated to these sets of unary maps, there are relative operads 
$$
\cP_{\mathrm{Cone}(\mathbb{R}^n)}^{\cW_{\mathbb{R}}}=\big(\cP_{\mathrm{Cone}(\mathbb{R}^n)},\cW_{\mathbb{R}}\big) \;\text{ and }\; \cP_{\mathrm{Cone}(\mathbb{Z}^n)}^{\cW_{\mathbb{Z}}}=\big(\cP_{\mathrm{Cone}(\mathbb{Z}^n)},\cW_{\mathbb{Z}}\big).
$$

\begin{cor}\label{cor: First and a half Lattice Localization}
The relative operad map $\delta_{\mathbb{Z}}\colon \cP_{\mathrm{Cone}(\mathbb{R}^n)}^{\cW_{\mathbb{R}}}\to \cP_{\mathrm{Cone}(\mathbb{Z}^n)}^{\cW_{\mathbb{Z}}}$ is a local equivalence. In particular,
we have a commutative diagram of $\infty$-categories
$$
\begin{tikzcd}
 \mathrm{pFA}_{\mathrm{Cone}(\mathbb{R}^n)}^{\cW_{\mathbb{R}}}(\cV)\simeq &[-10mm] \mathrm{pFA}_{\mathrm{Cone}(\mathbb{Z}^n)}^{\cW_{\mathbb{Z}}}(\cV)\\
 \mathrm{pFA}_{\mathrm{Cone}(\mathbb{R}^n)}^{\mathrm{lc}}(\cV) \simeq \ar[u, hookrightarrow] &  \mathrm{pFA}^{\mathrm{lc}}_{\mathrm{Cone}(\mathbb{Z}^n)}(\cV) \ar[u, hookrightarrow]
\end{tikzcd}\quad,
$$
for any symmetric monoidal $\infty$-category $\cV$.
\end{cor}
\begin{proof}
We prove the claim by applying Proposition \ref{prop:corolla-wise_criterion I}. Regarding condition (1), which asserts that $\cW_{\mathbb{R}}\to \cW_{\mathbb{Z}}$ is a weak homotopy equivalence, we will use Quillen's Theorem A. To apply this theorem, we must show that for each $\overline{V}\in \mathrm{Cone}(\mathbb{Z}^n)$, the slice category
	$$
	\big(\cW_{\mathbb{R}}\big)_{/\overline{V}}=\left\{ 
	U\in \cW_{\mathbb{R}}:\quad \delta_{\mathbb{Z}}(U)=
	U\cap \mathbb{Z}^n\xhookrightarrow{\;\infty\;}\overline{V}\right\}
	$$
	is cofiltered. Such a poset is clearly non-empty. Moreover, it satisfies that given two objects $U,U'\in \big(\cW_{\mathbb{R}}\big)\!\,_{/\overline{V}}$, we can find a third one $U''$ such that $U\xhookleftarrow{\;\infty\;}U''\xhookrightarrow{\;\infty\;} U'$ by simply picking the apex of $U''$ to be any point in the non-empty intersection $U\cap U'$ (see Remark \ref{rem: Cones, fixed direction and angle}). Hence, $\big(\cW_{\mathbb{R}}\big)\!\,_{/\overline{V}}$ is cofiltered, as claimed.
	
	For condition (2), letting $V\in \mathrm{Cone}(\mathbb{R}^n)$ and $\overline{V}=\delta_{\mathbb{Z}}(V)$, we must show that
	$$
	\Alg_{C_t}(\cP_{\mathrm{Cone}(\mathbb{R}^n)})_{V}^{\cW_{\mathbb{R}}} \longrightarrow \Alg_{C_t}(\cP_{\mathrm{Cone}(\mathbb{Z}^n)})_{\overline{V}}^{\cW_{\mathbb{Z}}}
	$$
	is a weak homotopy equivalence for any $t\geq 0$.
	This is another consequence of Quillen's Theorem A.  Simply note that the slice category over a multimorphism $\mu^{\overline{U}_{\bullet}}_{\overline{V}}\colon (\overline{U}_1,\dots,\overline{U}_t)\to \overline{V}$ in $\cP_{\mathrm{Cone}(\mathbb{Z}^n)}$ can be identified with
$$
\big( \Alg_{C_t}(\cP_{\mathrm{Cone}(\mathbb{R}^n)})_{V}^{\cW_{\mathbb{R}}}  \big)_{/\mu^{\overline{U}_{\bullet}}_{\overline{V}}}\simeq \left\{ U_{\bullet}\in \big(\cW_{\mathbb{R}/V}\big)^{\times t}:
\begin{matrix}
	&U_i\cap U_j=\diameter\text{ for }i\neq j,\text{ and}\\[1mm]
	& \delta_{\mathbb{Z}}(U_i)\xhookrightarrow{\;\infty\;}\overline{U}_i \text{ for all }i
\end{matrix}
\right\},
$$
where $\cW_{\mathbb{R}/V}$ denotes the subposet of $\cW_{\mathbb{R}}$ spanned by open cones which are contained in $V$. To see that the latter is cofiltered, we use again Remark \ref{rem: Cones, fixed direction and angle} to find a span  $U_i\xhookleftarrow{\;\infty\;}U''_i\xhookrightarrow{\;\infty\;} U'_i$ for $1\leq i\leq t$ given any two objects $U_{\bullet}$, $U'_{\bullet}$ in the poset. Non-emptiness follows from the fact that we can choose the apex for $U_i$ to be any element of $\overline{U}_i$.

The remaining hypothesis, i.e., in this case condition (b):
$$
	\Alg_{C_t}(\cP_{\mathrm{Cone}(\mathbb{R}^n)})_{V}^{\cW_{\mathbb{R}}} \longrightarrow 	\Alg_{C_t}(\cP_{\mathrm{Cone}(\mathbb{R}^n)})_{V'}^{\cW_{\mathbb{R}}}
$$
is a weak homotopy equivalence for any $V\xhookrightarrow{\;\infty\;}V'$ (note that essential surjectivity of $\cW_{\mathbb{R}}\to \cW_{\mathbb{Z}}$ is clear), follows by a similar application of Quillen's Theorem A.
\end{proof}

Our next result, Proposition \ref{prop: Second Lattice Localization}, connects the previous prefactorization algebras over cones with ordinary factorization algebras over the $(n-1)$-sphere $\mathbb{S}^{n-1}$ \cite[Definitions 3.1.13 and 3.3.8]{karlsson_assembly_2026}. It builds on the extremely useful \emph{Lurie--Seifert--van Kampen Theorem}, whose statement we recall.

\begin{thm}[{\cite[Theorem A.3.1]{lurie_higher_nodate}}]\label{thm:LSvK}
	Let $X$ be a topological space and $\mathrm{Open}(X)$ be the poset of open subsets of $X$. Let $\chi\colon \mathrm{C}\to \mathrm{Open}(X)$ be a functor from a small category $\mathrm{C}$. For a point $x\in X$, denote by $\mathrm{C}_x$ the full subcategory of $\mathrm{C}$ spanned by objects $c\in \mathrm{C}$ such that $x\in \chi(c)$. If $\mathrm{C}_x$ is weakly contractible for any $x\in X$, the canonical map
	%
	$$
	\underset{c\in \mathrm{C}}{\mathrm{hocolim}}\,\chi(c)\longrightarrow X
	$$
	is a weak homotopy equivalence.
\end{thm}

\begin{prop}\label{prop: Second Lattice Localization}
The inclusion of orthogonal categories $\mathrm{Cone}_{\circ}(\mathbb{R}^n)^{\perp}\hookrightarrow \mathrm{Cone}(\mathbb{R}^n)^{\perp}$ induces a local equivalence $\cP_{\mathrm{Cone}_{\circ}(\mathbb{R}^n)}^{\,\sharp}\hookrightarrow \cP_{\mathrm{Cone}(\mathbb{R}^n)}^{\,\sharp}$. In particular, we have equivalences of $\infty$-categories
$$
\mathrm{FA}^{\mathrm{lc}}_{\mathbb{S}^{n-1}}(\cV)\simeq \mathrm{pFA}_{\mathrm{Cone}_{\circ}(\mathbb{R}^n)}^{\mathrm{lc}}(\cV) \simeq 	\mathrm{pFA}_{\mathrm{Cone}(\mathbb{R}^n)}^{\mathrm{lc}}(\cV) \simeq  \mathrm{pFA}^{\mathrm{lc}}_{\mathrm{Cone}(\mathbb{Z}^n)}(\cV) \;,
$$
for any presentably symmetric monoidal $\infty$-category $\cV$, where  $\mathrm{FA}^{\mathrm{lc}}_{\mathbb{S}^{n-1}}(\cV)$ denotes the $\infty$-category of locally constant \textit{factorization algebras} over the $(n-1)$-sphere.
\end{prop}
\begin{proof} Let us begin by contemplating the sequence of equivalences of $\infty$-categories. The last one already appeared in Corollary \ref{cor: First and a half Lattice Localization}. The first one follows from identifying the orthogonal category $\mathrm{Cone}_{\circ}(\mathbb{R}^n)^{\perp}$, via $U\mapsto U\cap\mathbb{S}^{n-1}$, with the orthogonal category of (non-empty and connected) Euclidean disks in $\mathbb{S}^{n-1}$ (see \cite[Proposition 4.5]{benini_c-categorical_2026} and \cite[Corollary 5.3.4]{karlsson_assembly_2026}). Therefore, it just remains to check the first claim:  $\cP_{\mathrm{Cone}_{\circ}(\mathbb{R}^n)}^{\,\sharp}\hookrightarrow \cP_{\mathrm{Cone}(\mathbb{R}^n)}^{\,\sharp}$ is a local equivalence.
	
For such an endeavor, we introduce an auxiliary full orthogonal subcategory $\mathrm{Cone}_R(\mathbb{R}^n)^{\perp}$ of $\mathrm{Cone}(\mathbb{R}^n)^{\perp}$ for any $R\in \mathbb{R}_{\geq 0}$, whose objects are those open cones with apex $p$ satisfying $\vert\vert p\vert\vert \leq R$. These operads determine a functor
$$
\mathbb{R}_{\geq 0}\longrightarrow \RelOpinfty,\quad R\longmapsto \cP_{\mathrm{Cone}_{R}(\mathbb{R}^n)}^{\,\sharp}
$$
whose colimit is $\cP_{\mathrm{Cone}(\mathbb{R}^n)}^{\,\sharp}$. (Note that filtered colimits of ordinary operads and categories are homotopy colimits, essentially because filtered colimits of simplicial sets are homotopy colimits in both the Kan--Quillen model structure or the Joyal model structure.)
 Therefore, using that the localization functor $L\colon \RelOpinfty\to \Opinfty$ is a left adjoint and 2-out-of-3 for local equivalences, we may conclude the claim by showing that  $\cP_{\mathrm{Cone}_{\circ}(\mathbb{R}^n)}^{\,\sharp}\hookrightarrow \cP_{\mathrm{Cone}_{R}(\mathbb{R}^n)}^{\,\sharp}$ is a local equivalence. Let us demonstrate that Proposition \ref{prop:corolla-wise_criterion I} applies in this case.

To check condition (1),  $\mathrm{Cone}_{\circ}(\mathbb{R}^n)\to \mathrm{Cone}_{R}(\mathbb{R}^n)$ is a weak homotopy equivalence, apply Theorem \ref{thm:LSvK} to the commutative triangle
$$
\begin{tikzcd}
	\mathrm{Cone}_{\circ}(\mathbb{R}^n) \ar[rr, hookrightarrow]\ar[rd,"\theta"'] && \mathrm{Cone}_{R}(\mathbb{R}^n)\ar[ld,"\theta"]\\
	& \mathrm{Open}\big(\mathbb{S}^{n-1}_{R+1}\big)
\end{tikzcd},
$$  	
where $\mathbb{S}^{n-1}_{R+1}=\{x\in \mathbb{R}^n:\; \vert\vert x\vert\vert =R+1\}$ denotes the $(n-1)$-sphere of radius $R+1$, and $\theta\colon U\mapsto U\cap \mathbb{S}^{n-1}_{R+1}$. Since we are going to repeat the argument once again below, we leave the details to the reader.

Condition (2) asserts that, for any $V\in \mathrm{Cone}_{\circ}(\mathbb{R}^n)$ and $t\geq 0$, the map
 $$
 \begin{tikzcd}
 	\Alg_{C_t}(\cP_{\mathrm{Cone}_{\circ}(\mathbb{R}^n)})_{V} \ar[r, hookrightarrow]&[6mm] \Alg_{C_t}(\cP_{\mathrm{Cone}_{R}(\mathbb{R}^n)})_{V} 
 \end{tikzcd}
 $$
 is a weak homotopy equivalence. We see that this is the case by applying again Theorem \ref{thm:LSvK} to identify the homotopy type of both categories. Consider the following commutative triangle of posets
 $$
 \begin{tikzcd}
 	\Alg_{C_t}(\cP_{\mathrm{Cone}_{\circ}(\mathbb{R}^n)})_{V} \ar[rr,hookrightarrow]\ar[rd,"\theta_{\circ}"'] && \Alg_{C_t}(\cP_{\mathrm{Cone}_{R}(\mathbb{R}^n)})_{V}\ar[ld,"\theta"]\\
 	& \mathrm{Open}\big(\mathrm{Conf}_t(V\cap\mathbb{S}^{n-1}_{R+1})\big)
 \end{tikzcd}
 $$  	
 determined by setting $\theta\big(\mu^{U_{\bullet}}_{V}\big):=\prod_{i}(U_i\cap \mathbb{S}^{n-1}_{R+1})$ (consult the discussion preceding Notation \ref{notat: locally constant pFAs} for our conventions). Notice that, for any $i=1,\dots,t$, the topological space $U_i\cap \mathbb{S}^{n-1}_{R+1}$ is contractible. Also, for any $x_{\bullet}\in \mathrm{Conf}_t(V\cap \mathbb{S}^{n-1}_{R+1})$, the subposet
 $$
 \left\{\mu^{U_{\bullet}}_{V}\in \Alg_{C_t}(\cP_{\mathrm{Cone}_{R}(\mathbb{R}^n)})_{V}:\quad x_{\bullet}\in \prod_i(U_i\cap \mathbb{S}^{n-1}_{R+1})\right\}\subseteq \Alg_{C_t}(\cP_{\mathrm{Cone}_{R}(\mathbb{R}^n)})_{V}
 $$
 is cofiltered, and thus weakly contractible. The same holds replacing $\mathrm{Cone}_{R}(\mathbb{R}^n)$ by $\mathrm{Cone}_{\circ}(\mathbb{R}^n)$. Since we just checked the hypothesis of Theorem \ref{thm:LSvK} plus the natural equivalences  $\theta_{\circ}\xrightarrow{\;\sim\;}\ast\xleftarrow{\;\sim\;}\theta$, we obtain a commutative diagram of spaces 
 $$
 \begin{tikzcd}
 	\vert\Alg_{C_t}(\cP_{\mathrm{Cone}_{\circ}(\mathbb{R}^n)})_{V}\vert \ar[rr]\ar[d, leftarrow,"\rotatebox{180}{$\simeq$}"' sloped] && \vert\Alg_{C_t}(\cP_{\mathrm{Cone}_{R}(\mathbb{R}^n)})_{V}\vert\ar[d, leftarrow, "\simeq" sloped]\\
 	\mathrm{hocolim}\,\theta_{\circ} \ar[rr]\ar[rd,"\simeq"' sloped] && \mathrm{hocolim}\,\theta \ar[ld,"\simeq"' sloped]\\
 	& \mathrm{Conf}_t(V\cap \mathbb{S}^{n-1}_{R+1})
 \end{tikzcd}
 $$  
 proving that the upper horizontal map is a weak homotopy equivalence.
 
 Finally, a similar argument shows that 
 $$
 \vert\Alg_{C_t}(\cP_{\mathrm{Cone}_{R}(\mathbb{R}^n)})_{V}\vert \longrightarrow \vert\Alg_{C_t}(\cP_{\mathrm{Cone}_{R}(\mathbb{R}^n)})_{V'}\vert
 $$
 is a weak homotopy equivalence for any $ V\hookrightarrow V'$ in $\mathrm{Cone}_{R}(\mathbb{R}^n)$, proving (a).
\end{proof}

A direct consequence of Proposition \ref{prop: Second Lattice Localization} is that the locally constant prefactorization algebra $\mathbf{SSS}_{(\mathfrak{A},\pi_0)}\in \mathrm{pFA}_{\mathrm{Cone}(\mathbb{Z}^n)}^{\mathrm{lc}}\big(\Alg_{\mathbb{E}_1}(C^*\EuScript{C}\mathrm{at}_{\mathbb{C}}) \big)$ constructed in \cite[Corollary 4.3]{benini_c-categorical_2026} yields \textit{exactly} the same information as its underlying locally constant factorization algebra structure over $\mathbb{S}^{n-1}\times \mathbb{R}$ from \cite[Corollary 4.8]{benini_c-categorical_2026}; see loc.cit.\ for more details.

\subsection{Lorentzian localizations} One important source of orthogonal categories, and hence of prefactorization algebras, comes from Lorentz geometry. Let us focus on a couple of orthogonal categories and refer the interested reader to \cite{benini_operads_2021,benini_strictification_2023,benini_equivalence_2026} for more examples and details. For the rest of this subsection, we are going to fix a spacetime $M$ (i.e., a globally hyperbolic, connected, oriented, and time-oriented Lorentzian $m$-manifold; with or without time-like boundary), and later on, even restrict to the case $M=\mathbb{M}$ is flat Minkowski spacetime.    

Given $M$, we can define an orthogonal category $\mathrm{COpen}(M)^{\perp}$ as follows: (1) Its underlying category is the full subposet of $\mathrm{Open}(M)$ spanned by causally convex open subsets--i.e.\ $U\subseteq M$ open subspace such that any future/past-directed causal curve that starts and ends in $U$ completely lies in $U$--. (2) $U_1\perp_V U_2$ if and only if there is no future-directed causal curve in $V$ passing through both $U_1$ and $U_2$.
The associated operad $\cP_{\mathrm{COpen}(M)}$ comes with an interesting relative structure, which we denote by $(\cP_{\mathrm{COpen}(M)}, \cat W_{\mathrm{ts}})$. Here,  $\cW_{\mathrm{ts}}\subseteq \mathrm{COpen}(M)$ denotes the wide subcategory of Cauchy morphisms, i.e., those inclusions $U\hookrightarrow V$ such that $U$ contains a Cauchy surface of $V$.\footnote{The subscript ``$\mathrm{ts}$'' stands for  \textit{time-slice}, as a short-hand notation for Cauchy morphisms motivated by their relation to \textit{the time-slice axiom} for AQFTs.} (See \cite{benini_higher_2019,benini_strictification_2023} for details, references, and motivation for these notions coming from Quantum Field Theory.) 

In this restricted setting, we can compute the localization of $(\cP_{\mathrm{COpen}(M)},\cat W_{\mathrm{ts}})$ rather explicitly because of the following construction. The orthogonal category $\mathrm{COpen}(M)^{\perp}$ admits a (pointed) idempotent orthogonal endofunctor $$\mathrm{D}\colon\mathrm{COpen}(M)^{\perp}\longrightarrow \mathrm{COpen}(M)^{\perp},$$ called \textit{Cauchy-development}, given by
$$
\mathrm{D}(U):=\left\{	p\in M:
\begin{matrix}
	&\text{any inextensible future/past-directed}\\
	&\text{causal curve stemming from }p\text{ intersects }U
\end{matrix}
\right\},
$$ 
where $\eta_U\colon U\hookrightarrow \mathrm{D}(U)$ is the obvious inclusion (see \cite[Definition 2.2 and Lemma 3.2]{benini_algebraic_2018}). By setting the essential image of $\mathrm{D}$ to be $\mathrm{COpen}_{\mathrm{D}}(M)$, we obtain an adjunction of orthogonal categories
\[
\begin{tikzcd}
	\mathrm{D}\from \mathrm{COpen}(M)^{\perp} \ar[r, shift left=1.8]\ar[r,hookleftarrow, shift right=1.8,"\scalebox{0.8}{$\perp$}"] & \mathrm{COpen}_{\mathrm{D}}(M)^{\perp}\from \mathrm{inc}
\end{tikzcd}.
\]
This induces an adjunction, that we denote with the same symbols, between the associated prefactorization operads. The upshot is that we have an identification of relative operads $(\cP_{\mathrm{COpen}_{\mathrm{D}}(M)},\cat W_{\mathrm{ts}})=\cP_{\mathrm{COpen}_{\mathrm{D}}(M)}^{\,\flat}$, and that $\eta_U\colon U\hookrightarrow \mathrm{D}(U)$ is clearly a Cauchy morphism for any $U\in \mathrm{COpen}(M)$.

More precisely, we obtain:
\begin{prop}\label{prop: First Lorentzian localization} The operad map $\mathrm{D}\colon \cP_{\mathrm{COpen}(M)}\to \cP_{\mathrm{COpen}_{\mathrm{D}}(M)}$ exhibits the localization of $\cP_{\mathrm{COpen}(M)}$ at  $\cW_{\mathrm{ts}}$. In particular, there is an equivalence of $\infty$-categories
	$$
	\mathrm{pFA}_{\mathrm{COpen}(M)}^{\cW_{\mathrm{ts}}}(\cV)\simeq \mathrm{pFA}_{\mathrm{COpen}_{\mathrm{D}}(M)}(\cV),
	$$
induced by restriction along $\mathrm{D}\dashv\mathrm{inc}$, for any symmetric monoidal $\infty$-category $\cV$. 
\end{prop}
\begin{proof} Apply Proposition \ref{prop:reflective_localization} to the adjunction of relative operads
	\[
	\begin{tikzcd}
		\mathrm{D}\from (\cP_{\mathrm{COpen}(M)},\cat W_{\mathrm{ts}}) \ar[r, shift left=1.8]\ar[r,hookleftarrow, shift right=1.8,"\scalebox{0.8}{$\perp$}"] & \cP_{\mathrm{COpen}_{\mathrm{D}}(M)}^{\,\flat}\from \mathrm{inc}
	\end{tikzcd},
	\]
 to see that we have a local equivalence.
\end{proof}

\begin{rem}\label{rem: operad for AQFTs} The same arguments apply to the $2$-functor $$
	\mathrm{OrthCat}\longrightarrow \mathrm{Op},\quad \mathrm{C}^{\perp}\longmapsto \cO_{\mathrm{C}}^{\perp},
$$
	 which associates to an orthogonal category the operad encoding AQFTs over $\mathrm{C}$ used in \cite{benini_strictification_2023}. This way we recover Theorem 3.6 in loc.cit.\ as an analog of Proposition \ref{prop: First Lorentzian localization} for AQFTs. Also observe that Proposition \ref{prop: First Lorentzian localization} admits, at least, as many variations as the examples of reflective localizations of orthogonal categories listed in \cite[Example 3.4]{benini_strictification_2023}.
\end{rem}

In contrast to this pleasant example, even simple modifications of the orthogonal category $\mathrm{COpen}(M)^{\perp}$ yield relative operads whose localization at Cauchy morphisms is not so simple. An instance of this phenomenon was observed in \cite{benini_haag-kastler_2025}. There, the authors noticed that the Cauchy-development functor $\mathrm{D}$ does not restrict, in general, to the full orthogonal subcategories 
$$
\mathrm{RC}^{\diamond}(M)^{\perp}\hookrightarrow \mathrm{RC}(M)^{\perp}\hookrightarrow\mathrm{COpen}(M)^{\perp}
$$
spanned by causally convex open subsets $U\subseteq M$ which are additionally relatively compact (resp.\ and diffeomorphic to $\mathbb{R}^m$ where $m=\mathrm{dim}(M)$). Nevertheless, as we will see below, these examples are under control thanks to the \emph{operadic calculus of left fractions} (see Section \ref{sec:CLFs}).

\begin{rem} Restricting to flat Minkowski spacetime $M=\mathbb{M}$, we recover this nice feature: The orthogonal endofunctor $\mathrm{D}$ restricts to both $\mathrm{RC}(\mathbb{M})^{\perp}$ and $\mathrm{RC}^{\diamond}(\mathbb{M})^{\perp}$. Hence, $\mathrm{D}\colon \cP_{\mathrm{RC}^{(\diamond)}(\mathbb{M})}\to \cP_{\mathrm{RC}^{(\diamond)}_{\mathrm{D}}(\mathbb{M})}$ exhibits the localization of $\cP_{\mathrm{RC}^{(\diamond)}(\mathbb{M})}$ at $\cW_{\mathrm{ts}}$. 
\end{rem}

It was shown in \cite[Corollary B.5]{benini_haag-kastler_2025} that the pair $(\mathrm{RC}(M),\cW_{\mathrm{ts}})$ admits a (categorical) calculus of left fractions. (The same proof demonstrates that this is also the case for $(\mathrm{RC}^{\diamond}(M),\cW_{\mathrm{ts}})$.) The following result can be employed to upgrade this into an operadic calculus of left fractions.
\begin{lem}\label{lem:categorical to operadic CLF}
	Let $\mathrm{C}^{\perp}$ be an orthogonal category equipped with a wide subcategory $\cW\subseteq \mathrm{C}$ admitting a (categorical) calculus of left (resp.\ right) fractions\footnote{We implicitly assume that $\cW$ has the $2$-out-of-$3$ property, since that is part of our Definition \ref{def:calc}. Of course, this assumption can be relaxed.}. Assume the following compatibility condition holds: for any commutative diagram in $\mathrm{C}^{\perp}$
	$$
\begin{tikzcd}
	U'_1 \ar[r, leftarrow, "w_1"]\ar[rrdd,"f'_1"'] & U_1\ar[rd,"f_1"'] && U_2 \ar[ld,"f_2"]\ar[r, "w_2"] & U'_2\ar[lldd,"f'_2"]\\[-6mm]
&  & V \ar[d,"w"]& & \\
&  & V'& &
\end{tikzcd},
	$$
	where $w,w_1,w_2$ belong to $\cW$, $f_1\perp f_2$ if, and only if, $f'_1\perp f'_2$. Then, $(\cP_{\mathrm{C}},\cW)$ admits an operadic calculus of left (resp.\ right) fractions. 
\end{lem}
\begin{proof} First, recall that for a general orthogonal category $\mathrm{C}^{\perp}$, $m$-ary multimorphisms in $\cP_{\mathrm{C}}$ are given by $m$-tuples of maps in $\mathrm{C}$ with the same target, denoted $\mu(f_{\bullet})=(f_1,\dots,f_m)$, subject to $f_i\perp f_j$ for any $i\neq j$; see \cite[Definition 2.5]{benini_categorification_2021}. Then, (CL2) in Definition \ref{def:calc} follows by iterating the same axiom for $(\mathrm{C},\cW)$. 
Let us see how this works when $m=2$. Consider given $\{w_i\colon U_i\to U'_i\}_{i=1,2}$ in $\cW$ and $\mu(f_1,f_2)\colon (U_1,U_2)\to V$. Applying (CL2) for $(\mathrm{C},\cat W)$ on each component of the multimorphism we arrive at a commutative diagram in $\cP_{\mathrm{C}}$  
\[\begin{tikzcd}
	(U_1,U_2) &[6mm] (U_1',U_2) &[6mm] (U'_1,U'_2) \\
	V & \bullet & V'
	\arrow["{(w_1\text{,}\id)}", from=1-1, to=1-2]
	\arrow["{(\id\text{,}w_2)}", from=1-2, to=1-3]
	\arrow["\mu(f_{1}\text{,}f_{2})"', from=1-1, to=2-1]
	\arrow["\mu(f'_{1}\text{,}w'_1f_{2})", dashed, from=1-2, to=2-2]
	\arrow["\mu(f''_{1}\text{,}f''_{2})", dashed, from=1-3, to=2-3]
	\arrow["{w_1'}"', dashed, from=2-1, to=2-2]
	\arrow["{w_2'}"', dashed, from=2-2, to=2-3]
\end{tikzcd}.\]	
Notice the dashed vertical multimorphisms $\mu(f'_1,w'_1f_2)$ and $\mu(f''_1,f''_2)$ are well-defined due to the compatibility condition in the statement. A similar argument proves (CL3). The case of calculus of right fractions is dual. 
\end{proof}

\begin{cor}\label{cor: CLFs for Lorentzian pFAs}
The relative operad $(\cP_{\mathrm{RC}^{(\diamond)}(M)},\cW_{\mathrm{ts}})$ admits an operadic calculus of left fractions. In particular $\cP_{\mathrm{RC}^{(\diamond)}(M)}\to \cW^{-1}_{\mathrm{ts}}\cP_{\mathrm{RC}^{(\diamond)}(M)}$ (see Definition \ref{defn: Operad of left fractions}) exhibits the ($\infty$-)localization of $\cP_{\mathrm{RC}^{(\diamond)}(M)}$ at the Cauchy morphisms $\cat W_{\mathrm{ts}}$.
\end{cor}
\begin{proof} In light of Lemma \ref{lem:lcalc}, we are reduced to checking the compatibility condition in Lemma \ref{lem:categorical to operadic CLF} for $(\mathrm{RC}^{(\diamond)}(M),\cW_{\mathrm{ts}})$. This follows from the fact that we can test causal orthogonality in the ambient spacetime $M$, and since Cauchy-development detects it \cite[Proposition 2.5]{benini_algebraic_2018}: 
	$$
	U\perp_V U'
    \;\Longleftrightarrow\; U\perp_M U' \;\Longleftrightarrow\; \mathrm{D}(U)\perp_M \mathrm{D}(U'),
	$$
	for any $U,U'\in \mathrm{RC}^{(\diamond)}(M)$. On the right-hand side, we are looking at $\mathrm{D}(U),\mathrm{D}(U')$ as objects in $\mathrm{COpen}(M)$.  
\end{proof}

\begin{rem}\label{rem: more on CLFs for Lorentzian operads} Various observations are in order. First, Corollary \ref{cor: CLFs for Lorentzian pFAs} implies strictification results, in the spirit of \cite{benini_strictification_2023}, for the time-slice axiom over $\mathrm{RC}^{(\diamond)}(M)$-prefactorization algebras. Second, in \cite{benini_equivalence_2026}, an interesting variation of $\mathrm{RC}(M)^{\perp}$, denoted $\mathrm{Loc}^{\mathrm{rc},\, \perp}_{/M}$ in loc.cit., and the relation between these orthogonal categories with \textit{additivity} for AQFTs and time-orderable prefactorization algebras was presented. The previous discussion also applies to show that $(\mathrm{Loc}^{\mathrm{rc},\, \perp}_{/M},\cW_{\mathrm{ts}})$ admits a categorical calculus of left fractions that can be upgraded to an operadic calculus of left fractions for the associated relative prefactorization operad (see \cite[Appendix B]{benini_equivalence_2026}). Lastly, replacing $\mathrm{C}^{\perp}\mapsto \cP_{\mathrm{C}}$ by $\mathrm{C}^{\perp}\mapsto \cO_{\mathrm{C}}$ (c.f.\ Remark \ref{rem: operad for AQFTs}), the analog to Lemma \ref{lem:categorical to operadic CLF} shows that the AQFT operad admits an operadic calculus of left fractions with respect to Cauchy morphisms, recovering \cite[Theorem 5.1]{benini_equivalence_2026}. 
\end{rem}

In the remainder of this section, we restrict our attention to ($m$-dimensional) flat Minkowski spacetime $M=\mathbb{M}$. In fact, our last example is dedicated to studying a further localization of $\cP_{\mathrm{RC}^{\diamond}(\mathbb{M})}$ which connects to the recent work \cite{benini_prefactorization_2026} and the theory of localized sectors. 

\begin{defn}\label{defn: Caus orthogonal category} Let $\mathrm{DCone}(\mathbb{M})^{\perp}$ be the orthogonal category given by the following data:
	\begin{itemize}
	\item Its objects are ordered pairs of points $p^{\pm}=(p^+,p^-)\in \mathbb{M}^{\times 2}$ satisfying $p^+\in \mathrm{I}^+(p^{-})$, i.e.\ $p^+$ lives in the chronological future of $p^-$.
 	\item As morphisms, it has
	$$
	\mathrm{DCone}(\mathbb{M})(p^{\pm},q^{\pm})=\begin{cases}
		* & \text{ if }p^+\in \mathrm{J}^-(q^+)\text{ and }p^-\in \mathrm{J}^+(q^{-}),\\[3mm]
		\diameter & \text{ else}.
	\end{cases}
	$$
	\item Its orthogonality relation is defined as follows:  $p_1^{\pm}\perp_{q^{\pm}}p_2^{\pm}$ if, and only if, $p_1^+\notin \mathrm{J}^+(p_2^-)$ and $p_1^{-}\notin \mathrm{J}^{-}(p_2^+)$.
	\end{itemize}	
\end{defn}  
\begin{rem}\label{rem:Connection of Caus and RC} Note that the conditions on the morphisms of $\mathrm{DCone}(\mathbb{M})$ read: $p^+$ lives in the causal past of $q^+$, denoted $p^+\in \mathrm{J}^-(q^+)$, and $p^-$ lives in the causal future of $q^-$, denoted $p^-\in \mathrm{J}^{+}(q^-)$. A similar translation applies to the orthogonality relation above. Furthermore, the rule $p^{\pm}\mapsto \diamondsuit(p^{\pm})=\mathrm{I}^-(p^+)\cap \mathrm{I}^+(p^-)$, where $\mathrm{I}^+(p)$ denotes the chronological future of $p$ (resp.\ $\mathrm{I}^{-}(p)$ denotes the chronological past of $p$), identifies $\mathrm{DCone}(\mathbb{M})^{\perp}$ with a full orthogonal subcategory of $\mathrm{RC}^{\diamond}(\mathbb{M})^{\perp}$. (See \cite[\textsection 5]{benini_prefactorization_2026}.)
\end{rem}

\begin{defn} Let $V\in \mathrm{RC}^{\diamond}(\mathbb{M})$. The \textit{space of causal (ordered) configurations of }$t$-\textit{points in }$V$, denoted $\mathrm{cConf}_{t}(V)$, is the subspace of $V^{\times t}$ spanned by the tuples $p_{\bullet}=(p_1,\dots,p_t)$ satisfying $p_i\notin \mathrm{J}(p_j)=\mathrm{J}^+(p_j)\cup \mathrm{J}^-(p_j)$ for any $i\neq j$. In words, $p_i$ does not belong to the causal future or past of $p_j$.   
\end{defn}
\begin{rem}\label{rem: causal configurations are homotopy equivalent}
 A simple argument shows that $\mathrm{cConf}_{t}(\diamondsuit(p^{\pm}))$ and $\mathrm{cConf}_{t}(\diamondsuit(q^{\pm}))$ are homotopy equivalent for any $p^{\pm}\leq q^{\pm}$ in $\mathrm{DCone}(\mathbb{M})$. Similarly,  $\mathrm{cConf}_{t}(\diamondsuit(p^{\pm}))$ is homotopy equivalent to $\mathrm{Conf}_{t}(\mathbb{R}^n)$ with $n=\mathrm{dim}(\mathbb{M})-1$ for any $p^{\pm}\in \mathrm{DCone}(\mathbb{M})$. See \cite[\textsection 5]{benini_prefactorization_2026}. A variation of the argument proves that the same holds for any $V\subseteq V'$ in $\mathrm{RC}^{\diamond}(\mathbb{M})$.
\end{rem}

\begin{prop}\label{prop: Second Lorentzian localization}
The orthogonal functor $\diamondsuit\colon \mathrm{DCone}(\mathbb{M})^{\perp}\hookrightarrow \mathrm{RC}^{\diamond}(\mathbb{M})^{\perp}$ from Remark \ref{rem:Connection of Caus and RC} induces a local equivalence of relative operads $\cP_{\mathrm{DCone}(\mathbb{M})}^{\,\sharp}\to \cP_{\mathrm{RC}^{\diamond}(\mathbb{M})}^{\,\sharp}$. In particular, we have equivalences of $\infty$-categories
$$
\mathrm{pFA}^{\mathrm{lc}}_{\mathrm{RC}^{\diamond}(\mathbb{M})}(\cV)\simeq\mathrm{pFA}^{\mathrm{lc}}_{\mathrm{DCone}(\mathbb{M})}(\cV)\simeq \Alg_{\mathbb{E}_{n}}(\cV)\;,
$$
for any symmetric monoidal $\infty$-category $\cV$, where $n=\mathrm{dim}(\mathbb{M})-1$ and $\mathbb{E}_n$ denotes the little $n$-disks operad.
\end{prop}
\begin{proof} The second claim immediately follows from the first one due to \cite[Appendix A]{benini_prefactorization_2026}. Thus, let us show that Proposition \ref{prop:corolla-wise_criterion I} applies to the map of relative operads $ \cP_{\mathrm{DCone}(\mathbb{M})}^{\,\sharp}\to \cP_{\mathrm{RC}^{\diamond}(\mathbb{M})}^{\,\sharp}$. 
	
First, observe that condition (1), $\diamondsuit\colon \mathrm{DCone}(\mathbb{M})\to \mathrm{RC}^{\diamond}(\mathbb{M})$ is a weak homotopy equivalence, follows from the fact that both posets are filtered since we work over flat Minkowski spacetime.

Regarding condition (2), we need to show that, for any  $p^{\pm}\in \mathrm{DCone}(\mathbb{M})$ and $t\geq 0$, the map
$$
\begin{tikzcd}
	\Alg_{C_t}(\cP_{\mathrm{DCone}(\mathbb{M})})_{p^{\pm}} \ar[r,"\diamondsuit_*"] &[6mm] \Alg_{C_t}(\cP_{\mathrm{RC}^{\diamond}(\mathbb{M})})_{\diamondsuit(p^{\pm})} 
\end{tikzcd}
$$
is a weak homotopy equivalence. For this purpose, we will apply Lurie--Seifert--van Kampen Theorem \ref{thm:LSvK}. Consider the following commutative triangle of posets
$$
\begin{tikzcd}
\Alg_{C_t}(\cP_{\mathrm{DCone}(\mathbb{M})})_{p^{\pm}} \ar[rr,"\diamondsuit_*"]\ar[rd,"\theta"'] && \Alg_{C_t}(\cP_{\mathrm{RC}^{\diamond}(\mathbb{M})})_{\diamondsuit(p^{\pm})}\ar[ld,"\theta'"]\\
& \mathrm{Open}\big(\mathrm{cConf}_t(\diamondsuit(p^{\pm}))\big)
\end{tikzcd}
$$  	
determined by setting $\theta'\big(\mu^{U_{\bullet}}_{\diamondsuit(p^{\pm})}\big):=\prod_{i}U_i$ and $\theta\big(\mu^{p^{\pm}_{\bullet}}_{p^{\pm}}\big):=\prod_i\diamondsuit(p^{\pm}_i)$	(consult the discussion preceding Notation \ref{notat: locally constant pFAs} for our conventions). Notice that, for any $i=1,\dots,t$, the topological spaces $U_i$ and $\diamondsuit(p^{\pm}_i)$ are contractible as they are diffeomorphic to $\mathbb{R}^{n+1}$. Also, for any $p_{\bullet}\in \mathrm{cConf}_t(\diamondsuit(p^{\pm}))$, the subposet
$$
\left\{\mu^{U_{\bullet}}_{\diamondsuit(p^{\pm})}\in \Alg_{C_t}(\cP_{\mathrm{RC}^{\diamond}(\mathbb{M})})_{\diamondsuit(p^{\pm})}:\quad p_{\bullet}\in \prod_iU_i\right\}\subseteq \Alg_{C_t}(\cP_{\mathrm{RC}^{\diamond}(\mathbb{M})})_{\diamondsuit(p^{\pm})}
$$
is cofiltered, and thus weakly contractible. The same holds replacing $\mathrm{RC}^{\diamond}(\mathbb{M})$ by $\mathrm{DCone}(\mathbb{M})$. Since we just checked the hypothesis of Theorem \ref{thm:LSvK} plus the natural equivalences  $\theta\xrightarrow{\;\sim\;}\ast\xleftarrow{\;\sim\;}\theta'$, we obtain a commutative diagram of spaces 
$$
\begin{tikzcd}
	\vert\Alg_{C_t}(\cP_{\mathrm{DCone}(\mathbb{M})})_{p^{\pm}}\vert \ar[rr,"\vert\diamondsuit_*\vert"]\ar[d, leftarrow,"\rotatebox{180}{$\simeq$}"' sloped] && \vert\Alg_{C_t}(\cP_{\mathrm{RC}^{\diamond}(\mathbb{M})})_{\diamondsuit(p^{\pm})}\vert\ar[d, leftarrow, "\simeq" sloped]\\
	\mathrm{hocolim}\,\theta \ar[rr]\ar[rd,"\simeq"' sloped] && \mathrm{hocolim}\,\theta' \ar[ld,"\simeq"' sloped]\\
	& \mathrm{cConf}_t(\diamondsuit(p^{\pm}))
\end{tikzcd}
$$  
proving that $\diamondsuit_*$ is a weak homotopy equivalence. Finally, for condition (a), we observe that a similar argument relates, for any $V\hookrightarrow V'$ in $\mathrm{RC}^{\diamond}(\mathbb{M})$, the map 
$$
 \vert\Alg_{C_t}(\cP_{\mathrm{RC}^{\diamond}(\mathbb{M})})_{V}\vert \longrightarrow \vert\Alg_{C_t}(\cP_{\mathrm{RC}^{\diamond}(\mathbb{M})})_{V'}\vert
$$
with the inclusion $ \mathrm{cConf}_t(V)\hookrightarrow \mathrm{cConf}_t(V')$. The last map is a homotopy equivalence by Remark \ref{rem: causal configurations are homotopy equivalent}. 
\end{proof}

To connect Proposition \ref{prop: Second Lorentzian localization} with the results in \cite{benini_prefactorization_2026}, let us first observe that the proof of the cited proposition carries over for any orthogonal category $\mathrm{C}^{\perp}$ sitting in a chain of orthogonal full subcategory inclusions
$$
\diamondsuit\colon \mathrm{DCone}(\mathbb{M})^{\perp}\hookrightarrow \mathrm{C}^{\perp}\hookrightarrow \mathrm{RC}^{\diamond}(\mathbb{M})^{\perp}.
$$
If we further assume that $\mathrm{C}^{\perp}$ is filtered and satisfies \cite[Assumptions 3.5 (1) and 3.9]{benini_prefactorization_2026}, we could apply \cite[Theorem 3.11]{benini_prefactorization_2026} to produce a locally constant prefactorization algebra  $\mathbf{SSS}_{(\mathfrak{A},\pi_0)}\in \mathrm{pFA}_{\mathrm{C}}^{\mathrm{lc}}\big(\Alg_{\mathbb{E}_1}(C^*\EuScript{C}\mathrm{at}_{\mathbb{C}}) \big)$. In this case, a direct application of Proposition \ref{prop: Second Lorentzian localization} implies that $\mathbf{SSS}_{(\mathfrak{A},\pi_0)}$ \textit{exactly} equips the global $C^*$-category of superselection sectors of $\mathfrak{A}$ relative to $\pi_0$ from \cite[Definition 3.1]{benini_prefactorization_2026} with an $\mathbb{E}_{m}$-algebra structure, where $m=\mathrm{dim}(\mathbb{M})$. Thus, we recover the same information as in \cite[Theorem 5.11]{benini_prefactorization_2026}. (See Remark 5.12 in loc.cit.)

\begin{rem} The need to consider an additional orthogonal subcategory $\mathrm{C}^{\perp}$ in the previous paragraph comes from the fact that $\mathrm{RC}^{\diamond}(\mathbb{M})^{\perp}$ fails to satisfy \cite[Assumption 3.9]{benini_prefactorization_2026}. Hence, one cannot simply apply \cite[Theorem 3.11]{benini_prefactorization_2026} to obtain a locally constant prefactorization algebra of superselection sectors defined over $\mathrm{RC}^{\diamond}(\mathbb{M})$. A counterexample, i.e., an orthogonal pair $U_1\perp_{\widetilde{U}}U_2$ in $\mathrm{RC}^{\diamond}(\mathbb{M})$ that cannot be completed to a diagram as \cite[(3.25)]{benini_prefactorization_2026}, in spacetime dimension $m=3$ is given by the following procedure: (1) In $\mathbb{R}^2$, consider the euclidean ball centered at the origin with radius $1$ (resp.\ $2$),
	$
\Sigma_1=\{x\in\mathbb{R}^2:\; \vert\vert x\vert\vert<1\}
	$
(resp.\ $\widetilde{\Sigma}$),	
and the annulus with internal (resp.\ external) radius $1$ (resp.\ $2$) minus a segment,
$$
\Sigma_1=\{x\in\mathbb{R}^2:\; 1<\vert\vert x\vert\vert<2\text{ and }x\notin \{0\}\times (1,2)\}.
$$  
(2) Embed $\mathbb{R}^2$ as a horizontal Cauchy surface in $\mathbb{M}$, and take Cauchy-developments
 $$ 
 U_1=\mathrm{D}(\Sigma_1), \quad \widetilde{U}=\mathrm{D}(\widetilde{\Sigma}), \quad U_2=\mathrm{D}(\Sigma_2).
 $$ 	
Then, the three regions $U_1,\widetilde{U},U_2\in \mathrm{RC}(\mathbb{M})$ are diffeomorphic to $\mathbb{R}^m$, relatively compact, and satisfy $U_1\perp_{\widetilde{U}}U_2$. However, there cannot exist $V_1\perp_W V_2$ in $\mathrm{RC}^{\diamond}(\mathbb{M})$ with $U_1\subseteq V_1$ and $U_2\subseteq V_2$ fulfilling \cite[Assumption 3.9]{benini_prefactorization_2026}, since, in this case, there is no room for an orthogonal pair of the form $U'_1\perp_{V_1}U_1$.	
\end{rem}

\section{Application IV: Cyclic operads through modules}
\label{sec:cyclicoperads}
\textit{Cyclic operads} are a generalization of operads whose inputs
are indistinguishable from their outputs. A closely related notion
is that of (right) \textit{module} over an operad, which receives inputs from
operads but has no output. (A review of the definitions of these structures
can be found in Section \ref{sec:cyc_mod_def}.) Whereas compositions
of cyclic operads are governed by \textit{unrooted} trees, those for
modules over operads are governed by rooted forests that are
``planted'' at a single vertex (Figure \ref{fig:planted}). Such
data can be interpreted as an unrooted tree equipped with a distinguished
vertex, suggesting a connection between cyclic operads and operadic
modules. 

\begin{figure}[h]

	\begin{center}
		\begin{tikzpicture}
			
			\draw 
			(-2,1.2) -- (-2,0.5); 
			\draw
			(-2,0.5) to[out=-90, in=180] (0,-1);
			\draw
			(0,-1) to[out=0, in=-90] (2,0.5);
			\draw
			(2,0.5) -- (2,1.2);
			
			\fill (-2,0.5) circle (2pt);
			\fill (2,0.5) circle (2pt);
			\fill (0,-1) circle (2pt);
			
			\node at (0,-1.5) {$v$};
			
			\draw (-2,0.5) -- (-2,1.2);          
			\draw (-2,0.5) -- (-2.8,1.2);      
			\draw (-2,0.5) -- (-1.2,1.2);      
			
			\draw (2,0.5) -- (2,1.2);            
			\draw (2,0.5) -- (1.2,1.2);        
			\draw (2,0.5) -- (2.8,1.2);        
			
		\end{tikzpicture}
	\end{center}
	
	\caption{Two rooted trees planted at the vertex $v$}
	\label{fig:planted}
\end{figure}
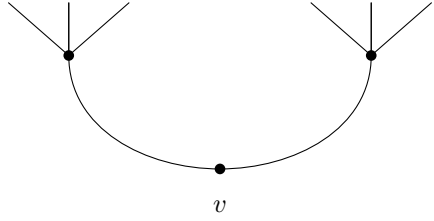

A closer inspection confirms this intuition, showing that the category
of cyclic operads in a symmetric monoidal category embeds fully faithfully
into that of ``pointed'' operadic modules. A natural question is:
Can we lift this to the homotopical setting? A recent work by Willwacher
solves a special case of this in the setting of rational chain complexes,
which led to the computation of the homotopy automorphism group of
the Batalin--Vilkovisky cooperad as a cyclic dg Hopf cooperad \cite{willwacher_cyclic_2024}.
In this section, we use the criteria developed in Section \ref{sec:criteria} to show that Willwacher's
result holds in arbitrary symmetric monoidal $\infty$-categories,
giving a complete answer to this question (Theorem \ref{thm:cyc_opd_mod}).

Our strategy is to define colored operads whose algebras are cyclic
operads (Definition \ref{def:cycopd}) and pointed operadic modules
(Definition \ref{def:opdmod}), and then apply one of our criteria. 

To define these colored operads, we need the notion of unrooted trees,
which we now introduce.
\begin{defn}
	An \emph{unrooted tree} is a finite set $T$, whose elements are called the \emph{points} of $T$, equipped with the following
	data:
	\begin{itemize}
		\item A map $T\to\{V,L\}$, where $V$ and $L$ are symbols. Elements of
		$T$ that map to $V$ are called \emph{vertices}, and those that
		map to $L$ are called \emph{leaves}. We write $\Vert\pr T$ and
		$\Leaf\pr T$ for the sets of vertices and leaves of $T$. 
		\item A collection $\Edge\pr T$ of subsets of $T$ of cardinality $2$,
		whose elements are called \emph{edges}.
	\end{itemize}
	These data are required to satisfy the following conditions:
	\begin{enumerate}
		\item Every leaf is contained in \textit{exactly one} edge.
		\item Every pair of distinct points
		is connected by \textit{exactly one} path.
		Here a \emph{path} in $T$ is a finite non-empty sequence $x_{1},\dots,x_{k}$
		of \textit{distinct} points of $T$, such that for each $1< i\leq k$,
		there is an edge having $x_{i-1}$ and $x_{i}$ as endpoints. 
	\end{enumerate}
	For each
	vertex $v\in\Vert\pr T$, we write $\Adj\pr v$ for the set of points
	$x\in T$ that are \emph{adjacent to} $v$, i.e., $\{v,x\}$ is an edge
	of $T$.
	
	An \emph{isomorphism} of unrooted trees $\phi\from T\cong T'$ is
	a bijection between the underlying set that respects the maps into
	$\{V,L\}$, and such that for every pair of points $x,y\in T$, the
	set $\{x,y\}$ is an edge if and only if $\{\phi\pr x,\phi\pr y\}$
	is an edge.
\end{defn}

\begin{rem}
	When drawing pictures of (un)rooted trees, we will represent vertices by
	black dots ``$\bullet$'', but we won't depict leaves.
\end{rem}

\begin{example}
	A \emph{star} is an unrooted tree with a single vertex.
	Such an unrooted tree is isomorphic to an \emph{$n$-star $S_{n}$} for some $n\geq 0$, which is the unrooted tree
	that has $n$ leaves $1,\dots,n$, a single vertex, and $n$ edges
	connecting the vertex with the leaves. The following is a picture
	of $S_{4}$:
	
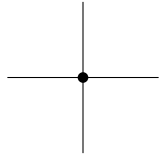
\begin{center}
		\begin{tikzpicture}
			\fill (0,0) circle (2pt);
			\draw (-1,0) -- (1,0);
			\draw (0,-1) -- (0,1);
		\end{tikzpicture}
	\end{center}
	Given an unrooted tree $T$ and a vertex $v\in T$, there is an associated star $S_v$ with set of leaves $\Adj\pr{v}$ connected to the unique vertex $v$.  
\end{example}

\begin{example}
	There is an unrooted tree with no vertex, two leaves, and a single edge joining
	them. It looks like a ``bar'', and is denoted by $\eta$. 
	Up to (unique) isomorphism, $\eta$ is the only unrooted tree without vertices.
\end{example}

\begin{example}
	\label{exa:substitution}An important operation of unrooted trees
	is \emph{substitution}. The main ingredients of this operation are
	an unrooted tree $T$, a vertex $v\in\Vert\pr T$, another unrooted
	tree $T_{v}$, and a bijection $\Leaf\pr{T_{v}}\cong\Adj\pr v$.
	From these data, we can define a new tree denoted by $T\lhd_{v}T_{v}$
	as follows: 
	\begin{enumerate}
		\item The underlying set of $T\lhd_{v}T_{v}$ is the pushout $\pr{T\setminus\{v\}}\amalg_{\Leaf\pr{T_{v}}}T{}_{v}$.
		\item The edges of $T\lhd_{v}T_{v}$ are the images 
		of the edges of $T_{v}$ and the images of edges of $T$ connecting
		points in $T\setminus\{v\}$. 
	\item The leaves of $T\lhd_{v}T_{v}$ are the images of the leaves
		of $T$.
	\end{enumerate}
	
	The following picture illustrates one instance of substitution:
	
	
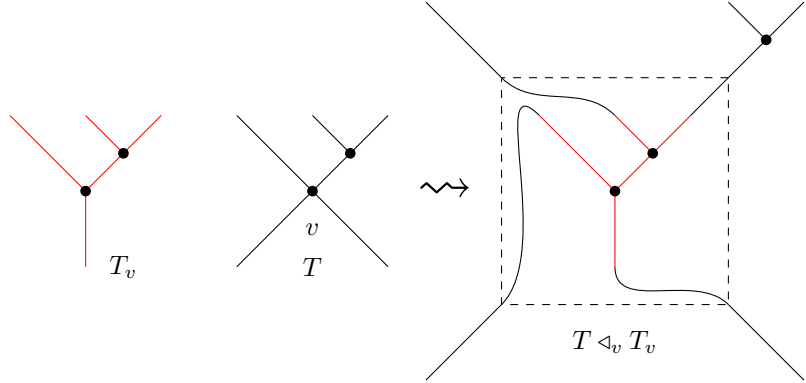
\begin{center}
		\begin{tikzpicture}
			
			\begin{scope}
				\draw[red] (0,0)--(0,-1);
				\draw[red] (0,0)--(1,1);
				\draw[red] (0.5, 0.5)--(0,1);
				\draw[red] (0,0)--(-1,1);
				\fill (0,0) circle (2pt);
				\fill (0.5,0.5) circle (2pt);
				\node at (.5,-1) {$T_v$};
			\end{scope}
			
			\begin{scope}[shift = {(3,0)}]
				\draw (-1,-1)--(1,1);
				\draw (.5,.5)--(0,1);
				\draw (-1,1)--(1,-1);
				\fill (0,0) circle (2pt);
				\fill (.5,.5) circle (2pt);
				\node at (0,-1) {$T$};
				\node at (0,-.5) {$v$};
			\end{scope}
			
			\begin{scope}[shift= {(5,0)}]
				\node at (-.25,0) {\Huge$\rightsquigarrow$};
			\end{scope}
			
			\begin{scope}[shift = {(7,0)}]
				
				\draw[red] (0,0)--(0,-1);
				\draw[red] (0,0)--(1,1);
				\draw[red] (0.5, 0.5)--(0,1);
				\draw[red] (0,0)--(-1,1);
				\fill (0,0) circle (2pt);
				\fill (0.5,0.5) circle (2pt);
				
				\draw (1.5, 1.5) --(2.5,2.5);
				\draw (-2.5, 2.5)--(-1.5, 1.5);
				\draw (-2.5,-2.5)--(-1.5, -1.5);
				\draw (2.5, -2.5)--(1.5, -1.5);
				\draw (2,2)--(1.5,2.5);
				\fill (2,2) circle (2pt);
				
				\draw (1,1) to[out=45, in =225] (1.5, 1.5);
				\draw (0,1) to[out=135, in =-45] (-1.5, 1.5);
				\draw (-1,1) to[out=135, in=45] (-1.5,-1.5);
				\draw (0,-1) to[out=-90,in=135] (1.5, -1.5);
				
				\draw[dashed] (1.5,1.5)--(-1.5, 1.5)--(-1.5,-1.5)--(1.5,-1.5)--(1.5, 1.5);
				
				\node at (0,-2) {$T\triangleleft_{v}T_v$};
				
			\end{scope}
			
		\end{tikzpicture}
	\end{center}
	
	Clearly, if we are to apply this construction twice at two different vertices of $T$, it doesn't matter in which order we perform the operations: First applying the construction at one vertex and then at the other, or in the reverse order. Thus, if we are given, for each $v\in\Vert\pr T$, an unrooted tree $T_{v}$
	and a bijection $\Leaf\pr{T_{v}}\cong\Adj\pr v$, we can define a
	new tree $T\lhd_{\Vert\pr T}( T_{v} )_{v\in \Vert(T)}$
	by iterating the ``substitution of $S_{v}$ by $T_{v}$''  above. 
\end{example}

Having defined unrooted trees, we can now define the colored operad
for cyclic operads.
\begin{defn}
	\label{def:cyc}We define a colored operad $\Cyc$ as follows: Its
	objects are the $n$-stars $S_{n}$, for $n\geq0$. Given a finite
	set $V$ and integers $n_{v}\geq0$
	for $v\in V$, and $n\geq0$, the
	set $\Cyc\binom{\pr{S_{n_{v}}}_{v\in V}}{S_{n}}$ is the set of equivalence
	classes of unrooted trees $T$ equipped with the following data:
	\begin{itemize}
		\item A bijection $V\cong\Vert\pr T$, by which we identify $V$ with $\Vert\pr T$.
		\item For each $v\in V$, a bijection $\Leaf\pr{S_{n_{v}}}\cong\Adj\pr v$.
		\item A bijection $\Leaf\pr{S_{n}}\cong\Leaf\pr T$.
	\end{itemize}
	Two such data $T$ and $T'$ are regarded as equivalent if there is
	an isomorphism of unrooted trees $T\cong T'$ that respects all the
	bijections involved. Composition is defined by the substitution operation
	described in Example \ref{exa:substitution}. 
\end{defn}

\begin{rem} Notice that in Definition \ref{def:cyc}, there are two possibilities for the set of multimorphisms $\Cyc\binom{\pr{S_{n_{v}}}_{v\in V}}{S_{n}}$ when $V$ is non-empty:  (1) $n_{v'}=0$ for just one element $v'\in V$; and (2) $n_{v}\geq 1$ for all $v\in V$. In Case (1), we have $n=0$, and $V$ and $\Cyc\binom{\pr{S_{n_{v}}}_{v\in V}}{S_{n}}$ are singletons; the latter only contains the identity morphism $\id\colon S_0\to S_0$. In Case (2), $n$ and $V$ are unobstructed.
	
	Also, observe that when $V$ is empty, we have	
	$$
	\Cyc\binom{\diameter}{S_n}=\begin{cases}
		*=\{\eta\} & \text{ if }n=2,\\[3mm]
		\diameter & \text{ else}.
	\end{cases}
	$$
\end{rem}

\begin{rem}
	If $\mathbf{C}$ is a symmetric monoidal category, then the category $\Alg_{\Cyc}\pr{\mathbf{C}}$
	of $\Cyc$-algebras in $\mathbf{C}$ is equivalent to the category of cyclic operads
	in $\mathbf{C}$. 
	For instance, the $\Sigma_2$-invariance of the unit (i.e., the unitality axiom) is encoded in the map
	$$
	\circ\colon \Cyc\binom{S_2}{S_2}\times \Cyc\binom{\diameter}{S_2}\cong \Sigma_2^{\op}\times *\xrightarrow{\quad!\quad}*\cong\Cyc\binom{\diameter}{S_2} . 
	$$
	Also, if we restrict our attention to the full suboperad
	spanned by $S_{n}$ with $n\geq2$, we get a colored operad for cyclic
	operads with $C\pr 0=C\pr 1=\diameter$.
\end{rem}

We now turn to the definition of the colored operad for pointed operadic
modules. For this, we need a few auxiliary notions.
\begin{defn}
	Let $T$ be an unrooted tree, and let $x\in T$ be a point. For each
	vertex $v\in\Vert\pr T\setminus\{x\}$, we write $\In_{x}\pr v=\Adj\pr v\setminus\{v_{1}\}$
	for the set of \emph{input vertices of $v$ relative to $x$}, where
	$v,v_{1},\dots,v_{k},x$ is the unique path from $v$ to $x$. 
	
	We can define the set of\emph{ input vertices relative to an edge}
	$e=\{x_{1},x_{2}\}$ of $T$ by introducing a new vertex $x_{+}$
	in the middle of $x_{1}$ and $x_{2}$, and then setting $\In_{e}\pr v=\In_{x_{+}}\pr v$
	for $v\in\Vert\pr T$.
\end{defn}

\begin{defn}
	A \emph{rooted tree} is an unrooted tree $T$ equipped with a distinguished
	element $r\in\Leaf\pr T$, called the \emph{root}. We write $\Leaf^{r}\pr T=\Leaf\pr T\setminus\{r\}$.
	
	If $T$ is a rooted tree, then for each vertex $v\in\Vert\pr T$,
	we write $\In\pr v=\In_{r}\pr v$ for the set of \emph{input vertices}
	of $v$. 
\end{defn}
\begin{example}
	For each $n\geq0$, we define the \emph{$n$-corolla} $C_n$ as the rooted tree whose underlying unrooted tree is the $(n+1)$-star $S_{n+1}$, and whose root is $n+1$.
\end{example}
\begin{defn}
	We define a colored operad $\OpdMod^{\pt}$ as follows. Its objects
	are $C_{n}$ for $n\geq0$ and $S_{n}$ for $n\geq0$. To describe
	its hom-sets, let $V$ be a finite set, and suppose we are given a
	symbol $X^{v}\in\{C,S\}$ for each $v\in V$. The maps into $C_{m}$
	can be described as follows:
	\begin{enumerate}
		\item If $X^{v}=S$ for some $v\in V$, then $\OpdMod^{\pt}\binom{\pr{X^{v}_{n_{v}}}_{v\in V}}{C_{m}}=\diameter$.
		\item Suppose that $X^{v}=C$ for all $v$. Then the set $\OpdMod^{\pt}\binom{\pr{C_{n_{v}}}_{v\in V}}{C_{m}}$
		is given by the equivalence classes (as in Definition \ref{def:cyc})
		of rooted trees $T$ equipped with the following data:
		\begin{enumerate}
			\item A bijection $V\cong\Vert\pr T$, by which we identify $V$ with $\Vert\pr T$.
			\item For each $v\in V$, a bijection $\In\pr{C_{n_{v}}}\cong\In\pr v$.
			\item A bijection $\In\pr{C_{m}}\cong\Leaf^{r}\pr T$.
		\end{enumerate}
	\end{enumerate}
	Next, the maps into $S_{m}$ can be described as follows:
	\begin{enumerate}
		\item If $X^{v}=S$ for more than one $v\in V$, then $\OpdMod^{\pt}\binom{\pr{X^{v}_{n_{v}}}_{v\in V}}{S_{m}}=\diameter$.
		\item If $X^{v}=S$ for exactly one $v\in V$, say $v_{0}$, then $\OpdMod^{\pt}\binom{\pr{X^{v}_{n_{v}}}_{v\in V}}{S_{m}}$
		is the set of equivalence classes of unrooted trees $T$ equipped
		with the following data:
		\begin{enumerate}
			\item A bijection $V\cong\Vert\pr T$, by which we identify $V$ with $\Vert\pr T$.
			\item A bijection $\Leaf\pr{S_{n_{v_{0}}}}\cong\Adj\pr{v_{0}}$.
			\item For each vertex $v\in V\setminus\{v_{0}\}$, a bijection $\In\pr{C_{n_{v}}}\cong\In_{v_{0}}\pr v$.
			\item A bijection $\Leaf\pr{S_{m}}\cong\Leaf\pr T$.
		\end{enumerate}
		\item If $X^{v}=C$ for all $v\in V$, then $\OpdMod^{\pt}\binom{\pr{X^{v}_{n_{v}}}_{v\in V}}{S_{m}}$
		is the set of equivalence classes of unrooted trees $T$ equipped
		with the following data:
		\begin{enumerate}
			\item A bijection $V\cong\Vert\pr T$, by which we identify $V$ with $\Vert\pr T$.
			\item An edge $e\in\Edge\pr T$.
			\item For each vertex $v\in V$, a bijection $\In\pr{C_{n_{v}}}\cong\In_{e}\pr v$.
			\item A bijection $\Leaf\pr{S_{m}}\cong\Leaf\pr T$.
		\end{enumerate}
	\end{enumerate}
	Again, composition is defined by substitution of trees. 
\end{defn}

\begin{rem}
	Let $(\mathbf{C},\otimes ,I)$ be a symmetric monoidal category. The category $\Alg_{\OpdMod^{\pt}}\pr{\mathbf{C}}$
	of algebras over $\OpdMod^{\pt}$ in $\mathbf{C}$ is equivalent to the category whose
	objects are pairs $\pr{O,M}$, where $O$ is an operad and $M$ is
	a pointed $O$-module, and whose morphisms $\pr{O,M}\to\pr{P,N}$
	are pairs of an operad map $O\to P$ and an $O$-module map $M\to N$
	under $I$.
\end{rem}

\begin{rem}
	Here is a picture of a typical morphism $\pr{S_{3},C_{2},C_{3},C_{2},C_0,C_{3}}\to S_{9}$
	in $\OpdMod^{\pt}$: 
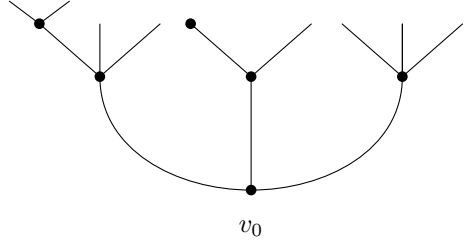
	\begin{center}
		\begin{tikzpicture}
			
			\draw 
			(-2,1.2) -- (-2,0.5); 
			\draw
			(-2,0.5) to[out=-90, in=180] (0,-1);
			\draw
			(0,-1) to [out=0, in=-90] (2,0.5);
			\draw
			(2,0.5) -- (2,1.2);
			
			\fill (-2,0.5) circle (2pt);
			\fill (2,0.5) circle (2pt);
			\fill (0,-1) circle (2pt);
			\fill (0,.5) circle (2pt);
			\fill (-2.8, 1.2) circle (2pt);
			\fill (-0.8, 1.2) circle (2pt);
			
			\draw
			(0,-1)--(0,0.5); 
			\draw
			(0, 0.5)--(0.8,1.2); 
			\draw
			(0,0.5)--(-0.8,1.2); 
			
			\node at (0,-1.5) {$v_0$};
			
			\draw (-2,0.5) -- (-1.2,1.2);          
			\draw (-2,0.5) -- (-2.8,1.2);      
			\draw (-2.8, 1.2) -- (-3.2, 1.5); 
			\draw (-2.8, 1.2) -- (-2.4, 1.5); 
			
			\draw (2,0.5) -- (2,1.2);            
			\draw (2,0.5) -- (1.2,1.2);        
			\draw (2,0.5) -- (2.8,1.2);        
			
		\end{tikzpicture}
	\end{center}
	Implicit in this picture is that each black dot (a vertex) comes equipped
	with a choice of a labelling of edges that are attached above it, as
	well as a choice of labelling on the leaves.
	
	The following picture is a morphism $\pr{C_{2},C_{3},C_{2},C_{1}}\to S_{6}$:

	
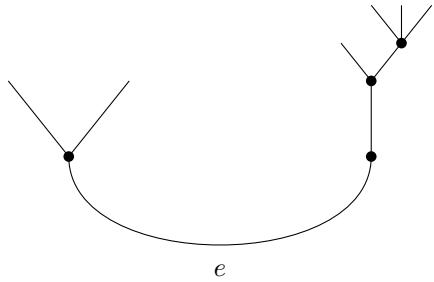
\begin{center}
		\begin{tikzpicture}
			\draw 
			(-2,0) to[out=-90, in=-90] (2,0);
			
			\draw
			(-2.8,1)--(-2,0); 
			\draw
			(-1.2,1)--(-2,0); 
			
			\draw
			(2,0) -- (2,1); 
			\draw
			(1.6,1.5) -- (2,1); 
			\draw
			(2.4, 1.5)--(2,1); 
			\draw
			(2,2)--(2.4, 1.5); 
			\draw
			(2.4,2)--(2.4, 1.5); 
			\draw
			(2.8,2)--(2.4, 1.5); 
			
			\fill (-2,0) circle (2pt);
			\fill (2,0) circle (2pt);
			\fill (2,1) circle (2pt);
			\fill (2.4, 1.5) circle (2pt);
			
			\node at (0,-1.5) {$e$};
			
		\end{tikzpicture}
	\end{center}Again, we suppressed the data of labelling. If $O$ is an operad and
	$M$ is a pointed $O$-module in a symmetric monoidal category $\pr{\mathbf{C},\t,I}$,
	this operation corresponds to the map
	\[
	I\otimes O\pr 2\otimes O\pr 3\otimes O\pr 2\otimes O\pr 1\xrightarrow{\eta\otimes \id} M\pr 2\otimes O\pr 2\otimes O\pr 3\otimes O\pr 2\otimes O\pr 1\to M\pr 6,
	\]
	where the first map is given by the pointing of $M$, and the second
	map is given by the operad structure on $O$ and the module structure
	of $M$.
\end{rem}

By construction, there is a forgetful functor 
\[
\OpdMod^{\pt}\longrightarrow\Cyc
\]
which maps $C_{n}\mapsto S_{n+1}$ and $S_{n}\mapsto S_{n}$, and
which modifies the labelling corresponding to the corollas $C_{n}$
by adding the label $n+1$ to its root. The main result of this section
asserts that this is a localization:
\begin{thm}
	\label{thm:cyc_opd_mod}The forgetful functor
	\[
	\OpdMod^{\pt}\longrightarrow\Cyc,
	\]
	is a localization at all morphisms.
\end{thm}

We will give a proof of Theorem \ref{thm:cyc_opd_mod} at the end of this section. We first look at its corollaries.
\begin{cor}
	\label{cor:cyc_opd_mod_infty}Let $\cat C$ be a symmetric monoidal
	$\infty$-category. The functor
	\[
	\Alg_{\Cyc}\pr{\cat C}\longrightarrow\Alg_{\OpdMod^{\pt}}\pr{\cat C}
	\]
	is fully faithful, and its essential image consists of those algebras
	that carry the maps $\{C_{n}\to S_{n+1}\}_{n\geq0}$ to equivalences
	of $\cat C$.
\end{cor}

\begin{proof}
	This follows from Theorem \ref{thm:cyc_opd_mod} and the definition
	of localization of $\infty$-operads.
\end{proof}

\begin{rem} 
	Notice that the set of unary maps $\OpdMod^{\pt}\binom{C_n}{S_{n+1}}$, i.e., the maps inverted in Theorem \ref{thm:cyc_opd_mod}, can be identified with $\Sigma_{n+1}$.	Since, for any $n\geq0$, 
	$$
	\circ\colon \OpdMod^{\pt}\binom{S_{n+1}}{S_{n+1}}\times \OpdMod^{\pt}\binom{C_n}{S_{n+1}}\longrightarrow \OpdMod^{\pt}\binom{C_n}{S_{n+1}}
	$$
	yields a free and transitive group action, and since the endomorphisms of $S_{n+1}$ are all automorphisms, it suffices to invert just one map $C_n\to S_{n+1}$ in $\OpdMod^{\pt}$ for each $n\geq 0$. 
	Thus, in the situation of Corollary \ref{cor:cyc_opd_mod_infty}, a pointed module $(O,M)\in \Alg_{\OpdMod^{\pt}}(\cat C)$ lies in the essential image of the restriction along $\OpdMod^{\pt}\to \Cyc$ if and only if, for each $n\geq0$, the map
	$$
	O(n)\simeq I\otimes O(n) \xrightarrow{\;\eta\otimes \id \;} M(2)\otimes O(n)\xrightarrow{\;\phantom{\eta\otimes \id}\;} M(n+1) 	
	$$
is invertible, where the second component is the $O$-module action associated to the function $\underline{n}\to \underline{1}\hookrightarrow \underline{1}\!\,_+=\underline{2}$.
\end{rem}

For the next corollary, recall that a \emph{symmetric monoidal model
	category} is a closed symmetric monoidal category equipped with a model structure, which
	satisfies the pushout-product axiom and whose unit object is cofibrant (Example \ref{exa:symmonmodcat}).
\begin{cor}
	\label{cor:cyc_opd_mod_modelcats}Let $\mathbf{V}$ be a cofibrantly
	generated symmetric monoidal model category.
	Then the right Quillen functor
	\[
	\Alg_{\Cyc}\pr{\mathbf{V}}\longrightarrow\Alg_{\OpdMod^{\pt}}\pr{\mathbf{V}}
	\]
	of semi-model categories \cite[Definition 2.1.1]{white_smith_2024} has a fully
	faithful total right derived functor.
\end{cor}

\begin{proof}
	This follows from Corollary \ref{cor:cyc_opd_mod_infty} and the rectification
	theorem of algebras over $\infty$-operads \cite[Theorem 7.3.1]{white_smith_2024}.
	(The unit object needs to be cofibrant to ensure that $\Cyc$ and $\OpdMod^{\pt}$
	are $\Sigma$-cofibrant in $\mathbf{V}$. But again, we can almost always force this to be true without changing the class of
	weak equivalences; see Remark \ref{rem:cofibrantunit}.)
\end{proof}

\begin{rem}
	Corollary \ref{cor:cyc_opd_mod_modelcats} remains valid, with the
	same proof, if we omit $C_{0}$ and $S_{0},S_{1}$ from the definitions
	of $\Cyc$ and $\OpdMod^{\pt}$. One then obtains the main theorem
	of \cite{willwacher_cyclic_2024} by slicing the right Quillen functor of Corollary
	\ref{cor:cyc_opd_mod_modelcats} by the unit cyclic operad.
\end{rem}

We now turn to the proof of Theorem \ref{thm:cyc_opd_mod}.

\begin{proof}
	[Proof of Theorem \ref{thm:cyc_opd_mod}]By Proposition \ref{prop:corolla-wise_criterion II}, it
	suffices to prove the following assertions:
	\begin{enumerate}
		\item The map $\phi\from U\pr{\OpdMod^{\pt}}\to U\pr{\Cyc}^{\simeq}$ is a weak
		homotopy equivalence.
		\item For each $m\geq1$, the map $\Alg_{C_{n}}\pr{\OpdMod^{\pt}}_{C_{m-1}}\to\Alg_{C_{n}}\pr{\Cyc}^{\simeq}_{S_{m}}$
		is a weak homotopy equivalence.
		\item For each $m\geq0$, the map $\Alg_{C_{n}}\pr{\OpdMod^{\pt}}_{S_{m}}\to\Alg_{C_{n}}\pr{\Cyc}^{\simeq}_{S_{m}}$
		is a weak homotopy equivalence.
	\end{enumerate}
	
	We start with (1). For each $m\geq0$, let $\cat X_{m}\subset U\pr{\OpdMod^{\pt}}$
	denote the full subcategory spanned by the objects $C_{m-1}$ and
	$S_{m}$ (when $m=0$, we only consider $S_{m}$), and let $\cat Y_{m}\subset U\pr{\Cyc}^{\simeq}$
	denote the full subcategory spanned by the object $S_{m}$. Then $U\pr{\OpdMod^{\pt}}=\coprod_{m\geq0}\cat X_{m}$
	and $U\pr{\Cyc}^{\simeq}=\coprod_{m\geq0}\cat Y_{m}$, so it suffices
	to show that the restriction $\phi_{m}\from\cat X_{m}\to\cat Y_{m}$
	of $\phi$ is a weak homotopy equivalence. We can identify $\cat Y_{m}$
	with a full subcategory of $\cat X_{m}$, and $\phi_{m}$ is a retraction
	to the inclusion. Moreover, the composite $\cat X_{m}\xrightarrow{\phi_{m}}\cat Y_{m}\hookrightarrow\cat X_{m}$
	admits a natural transformation from the identity functor. Hence $\phi_{m}$
	is a weak homotopy equivalence, as desired.
	
	For part (2), let $\mathrm{Tree}\pr{n,m}$ denote the set of equivalence
	classes of unrooted trees $T$ equipped with bijections $\Vert\pr T\cong\{1,\dots,n\}$
	and $\Leaf\pr T\cong\{1,\dots,m\}$. Two such trees are regarded as
	equivalent if there is an isomorphism of unrooted trees between them that commutes
	with these bijections. We then observe that the map
	\[
		\pi_{0}\pr{\Alg_{C_{n}}\pr{\OpdMod^{\pt}}_{C_{m-1}}}\longrightarrow\pi_{0}\pr{\Alg_{C_{n}}\pr{\Cyc}^{\simeq}_{S_{m}}}
	\]
	can be identified with the identity map of $\mathrm{Tree}\pr{n,m}$.
	Since the components of $\Alg_{C_{n}}\pr{\OpdMod^{\pt}}_{C_{m-1}}$ and
	$\Alg_{C_{n}}\pr{\Cyc}^{\simeq}_{S_{m}}$ are contractible groupoids,
this shows that $\psi$ is a weak homotopy equivalence (or even an
equivalence of categories).

Finally, we prove (3). For each element $[T]\in\mathrm{Tree}\pr{n,m}$
represented by $T$, let $P_{[T]}$ denote the poset whose elements
are the edges and vertices of $T$. The relation is generated by $v<e$,
where $e$ is an edge of $T$ containing $v$. We then observe that
$\Alg_{C_{n}}\pr{\OpdMod^{\pt}}_{S_{m}}$ and $\Alg_{C_{n}}\pr{\Cyc}^{\simeq}_{S_{m}}$
are thin categories (i.e., there is at most one morphism from one
object to another), whose skeletons can be identified with $\coprod_{ \mathrm{Tree}\pr{n,m}}P^{\op}_{[T]}$
and $\mathrm{Tree}\pr{n,m}$. Thus, we are reduced to showing that the projection
\[
\coprod_{[T]\in\mathrm{Tree}\pr{n,m}}P^{\op}_{[T]}\to\mathrm{Tree}\pr{n,m}.
\]
is a weak homotopy equivalence. But this follows from the observation that the geometric
realization of $P_{[T]}$ matches exactly with the picture of $T$
we depict, which is contractible by the definition of unrooted trees.
\end{proof}

\appendix
\section{Definition of cyclic operads and operadic modules}
\label{sec:cyc_mod_def}
The definitions of cyclic operads and operadic modules are well-known
to experts, but it is difficult to find a single reference containing
both definitions in the way that illuminates its connection with (un)rooted 
trees. Moreover, the definitions are not entirely standardized. (For
example, some authors do not consider nullary operations for cyclic operads.)
For these reasons, we collect in this section the definitions.
\begin{notation}
	We write $\FB$ for the groupoid of finite sets and bijections. We also write $\underline{n}=\{1,\dots ,n\}.$
\end{notation}
\begin{notation}
	Given a pair of finite set $X,Y$ and an element $(x,y)\in X\times Y$, we write $X\ps{\amalg}xyY= (X\setminus \{x\}) \amalg (Y\setminus \{y\})$. 
\end{notation}

\begin{defn}
\cite[Definition 1.3.2]{lukacs_cyclic_2010}\label{def:cycopd} Let $(\mathbf{C}, \otimes , I)$
be a symmetric monoidal category. A \emph{cyclic collection} in $\mathbf{C}$ is a functor $C\from\FB^{\op}\to\mathbf{C}$. A cyclic
operad is a cyclic collection $C\from\FB^{\op}\to\mathbf{C}$ equipped
with maps
\[
\ps{\circ}xy\from C\pr X\otimes C\pr Y\longrightarrow C\pr{X\ps{\amalg}xyY}
\]
for each $X,Y\in\FB$ and $\pr{x,y}\in X\times Y$, satisfying the
following properties: 
\begin{itemize}
\item (\textbf{Associativity}) The following diagrams commute: 
\[\begin{tikzcd}
		{C(X)\otimes C(Y)\otimes C(Z)} &[6mm] {C(X\ps{\amalg}{x}{y}Y)\otimes C(Z)} \\
	{C(X)\otimes C(Y\ps{\amalg}{\bar{y}}{z}Z)} & {C(X\ps{\amalg}{x}{y}Y\ps{\amalg}{\bar{y}}{z}Z)} \\
	{C(X)\otimes C(Y)\otimes C(Z)} & {C(X\ps{\amalg}{x}{y}Y)\otimes C(Z)} \\
	{C(X)\otimes C(Z)\otimes C(Y)} \\
	{C(X\ps{\amalg}{\bar{x}}{z}Z)\otimes C(Y)} & {C(X\ps{\amalg}{x}{y}Y\ps{\amalg}{\bar{x}}{z}Z)}
	\arrow[from=1-1, to=1-2,"\ps{\circ}{x}{y}\otimes \id"]
	\arrow[from=1-1, to=2-1,"\id\otimes \ps{\circ}{\bar{y}}{z}"']
	\arrow[from=1-2, to=2-2,"\ps{\circ}{\bar{y}}{z}"]
	\arrow[from=2-1, to=2-2,"\ps{\circ}{x}{y}"']
	\arrow[from=3-1, to=3-2,"\ps{\circ}{x}{y}\otimes \id"]
	\arrow["\cong" sloped, from=3-1, to=4-1]
	\arrow[from=3-2, to=5-2,"\ps{\circ}{\bar{x}}{z}"]
	\arrow[from=4-1, to=5-1,"\ps{\circ}{\bar{x}}{z}\otimes \id"']
	\arrow[from=5-1, to=5-2, "\ps{\circ}{x}{y}"']
\end{tikzcd}\]
where $x,\bar{x}\in X$, $y,\bar{y}\in Y$, $z\in Z$, and $x\neq \bar{x}$ and $y\neq \bar{y}$.
\item (\textbf{Equivariancy}) For any $x\in X$, $y\in Y$, and any maps
$\sigma\from X\to X'$ and $\tau\from Y\to Y'$ in $\FB$, the diagram
\[\begin{tikzcd}
	{C(X')\otimes C(Y')} &[6mm] {C(X'\ps{\amalg}{\sigma(x)}{\tau(y)}Y')} \\
	{C(X)\otimes C(Y)} & {C(X\ps{\amalg}xyY)}
	\arrow[from=1-1, to=1-2,"\ps{\circ}{\sigma(x)}{\tau(y)}"]
	\arrow[from=1-1, to=2-1,"\sigma^*\otimes\tau^*"']
	\arrow[from=1-2, to=2-2,"(\sigma\ps{\circ}xy\tau)^*"]
	\arrow[from=2-1, to=2-2,"\ps{\circ}{x}{y}"']
\end{tikzcd}\]commutes. 
\item (\textbf{Unitality}) For each set $U=\{u_{0},u_{1}\}$ with two elements,
there is a map $\eta_{U}\from I\to P\pr U$ with the following properties:
For every $X\in\FB$, $x\in X$, and $i\in\{0,1\}$, the diagram 
\[\begin{tikzcd}
	{I\otimes C(X)} &[3mm] {C(U)\otimes C(X)} & {C(X)\otimes C(U)} &[3mm] {C(X)\otimes I} \\
	& {C(X)} & {C(X)}
	\arrow[from=1-1, to=1-2,"\eta_U\otimes \id"]
	\arrow["\cong"' sloped, from=1-1, to=2-2]
	\arrow["{\ps{\circ}{u_i\!}{x}}", from=1-2, to=2-2]
	\arrow["{\ps{\circ}{x}{u_i}}"', from=1-3, to=2-3]
	\arrow[from=1-4, to=1-3,"\id\otimes\eta_U"']
	\arrow["\cong"' sloped, from=1-4, to=2-3]
\end{tikzcd}\]commutes. Moreover, for any map $U\to U'$ in $\FB$, the diagram
\[\begin{tikzcd}
	& I & \\
	{C(U')} && {C(U)}
	\arrow[from=1-2, to=2-1,"\eta_{U'}"']
	\arrow[from=1-2, to=2-3,"\eta_U"]
	\arrow["\cong"', from=2-1, to=2-3]
\end{tikzcd}\]commutes.
\end{itemize}
\emph{Morphisms of cyclic operads} are natural transformations of
underlying cyclic collections that respect the unit and composition
maps.
\end{defn}

\begin{rem}\label{rem: ordinary operad from cyclic one} 
	Recall that a monochromatic operad can be seen as a functor $\FB^{\op}\to\mathbf{C}$
equipped with composition maps and unit maps. (See, e.g., \cite[3.1.1]{fresse_modules_2009}.)
Every cyclic operad $C$ has an underlying monochromatic operad $C^{\mathrm{opd}}$,
defined by 
\[
C^{\mathrm{opd}}\pr S=C\pr{S_{+}},
\]
where $S_{+}=S\amalg\{+\}$ denotes the set $S$ with a newly added
element $+$. The composition map $\circ_{s}\from C^{\mathrm{opd}}\pr S\otimes C^{\opd}\pr T\to C^{\opd}\pr{S\setminus\{s\}\amalg T}$
is simply the map $_{s}\circ_{+}$, and the unit map $I\to C^{\opd}\pr{\underline{1}}=C\pr{\underline{1}\!\,_{+}}$.

From the perspective of cyclic operads, it might be more natural to
use the category $\FB_{\ast}$ of finite pointed sets and pointed
bijections for monochromatic operads. Under this point of view, the functor $\pr -^{\opd}$ is induced by precomposition with the forgetful functor $\FB_{\ast}\to\FB$. The connection between both descriptions is simply due to the equivalence of categories $\pr -_{+}\from\FB\xrightarrow{\simeq}\FB_{\ast}$. This
is the approach adopted in \cite{lukacs_cyclic_2010}.
\end{rem}

\begin{defn}
\cite[5.1.1]{fresse_modules_2009}\label{def:opdmod} Let $O$ be a (monochromatic)
operad in a symmetric monoidal category $\mathbf{C}$ (seen as a functor $O\colon \FB^{\op}\to \mathbf{C}$). A (right) \emph{module
over} $O$, or an \emph{$O$-module}, is a functor $M\from\FB^{\op}\to\mathbf{C}$
equipped with maps
\[
 M\pr T \otimes \pr{\bigotimes_{t\in T}O\pr{S_{t}}}\longrightarrow M\pr S
\]
for each map $S\to T$ of finite sets with fibers $\{S_{t}\}_{j\in J}$,
satisfying the following axioms:
\begin{itemize}
\item (\textbf{Associativity}) For every pair of set maps $S\to T\to U$,
the diagram
\[\begin{tikzcd}
	{M(U)\otimes \left(\bigotimes_{u\in U}O(T_u)\right)}\otimes \left(\bigotimes_{t\in T}O(S_t)\right) & {M(T)\otimes\left(\bigotimes_{t\in T}O(S_t)\right)} \\
	{M(U)\otimes \left(\bigotimes_{u\in U}O(S_u)\right)} & {M(S)}
	\arrow[from=1-1, to=1-2]
	\arrow[from=1-1, to=2-1]
	\arrow[from=1-2, to=2-2]
	\arrow[from=2-1, to=2-2]
\end{tikzcd}\]commutes.
\item (\textbf{Equivariancy}) Suppose we are given a commutative diagram
\[\begin{tikzcd}
	S & {S'} \\
	T & {T'}
	\arrow["\cong", from=1-1, to=1-2]
	\arrow[from=1-1, to=2-1]
	\arrow[from=1-2, to=2-2]
	\arrow["\cong"', from=2-1, to=2-2]
\end{tikzcd}\]of finite sets, where the horizontal arrows are bijections. Then the
diagram 
\[\begin{tikzcd}
	{ M(T')\otimes\left(\bigotimes_{t'\in T'}O(S'_{t'})\right)} & { M(T)\otimes\left(\bigotimes_{t\in T}O(S_{t})\right)} \\
	{M(S')} & {M(S)}
	\arrow["\cong", from=1-1, to=1-2]
	\arrow[from=1-1, to=2-1]
	\arrow[from=1-2, to=2-2]
	\arrow["\cong"', from=2-1, to=2-2]
\end{tikzcd}\]commutes.
\item (\textbf{Unitality}) For every finite set $S$, the diagram 
\[\begin{tikzcd}
	{M(S)\otimes I^{\otimes S}} & {M(S)\otimes\left(\bigotimes_{s\in S }O(\{s\})\right)} \\
	& {M(S)}
	\arrow[from=1-1, to=1-2]
	\arrow["\cong"' sloped, from=1-1, to=2-2]
	\arrow[from=1-2, to=2-2]
\end{tikzcd}\]commutes.
\end{itemize}
\emph{Morphisms of $O$-modules} are maps between the underlying
symmetric collections that respect the action of $O$.
\end{defn}

An $O$-module $M$ is said to be \emph{pointed} if, for any set $U$ with two elements, it comes equipped
with a map $\eta_{U}\colon I\to M(U)$ with the following property: For any map $U\to U'$ in $\FB$, the following triangle commutes
\[\begin{tikzcd}
	& I & \\
	{M(U')} && {M(U).}
	\arrow[from=1-2, to=2-1, "\eta_{U'}"']
	\arrow[from=1-2, to=2-3, "\eta_{U}"]
	\arrow["\cong"', from=2-1, to=2-3]
\end{tikzcd}.\]
In other words, $M$ comes with a preferred element in $M(2)^{\Sigma_2}$. \emph{Morphisms of pointed modules} are additionally required to preserve the pointing.

\begin{example}
Let $C$ be a cyclic operad in a symmetric monoidal category $\mathbf{C}$.
Then, as explained in Remark \ref{rem: ordinary operad from cyclic one}, we can extract an operad $C^{\opd}$ from $C$ whose value over a finite set $S$ is  
$
C^{\opd}\pr S=C\pr{S_{+}},
$
where $S_{+}$ denotes the set $S$ with a disjoint element.

We can also define a $C^{\opd}$-module $C^{\mod}$ by setting $C^{\mod}\pr S=C\pr S$.
This module is canonically pointed by the unit of $C$.
\end{example}

\begin{example}
Let $O$ be an operad in a symmetric monoidal category $\mathbf{C}$,
and let $M$ be a pointed $O$-module. Suppose that, for each finite
set $S$, the map
\[
O\pr S\cong I\otimes O\pr S \xrightarrow{\;\eta\otimes\id\;} M\pr{\underline{1}\!\,_{+}}\otimes O\pr S\xrightarrow{\;\phantom{\eta\otimes \id}\;} M\pr{S_{+}}
\]
is an isomorphism, where the second map is induced by the map $S\to\underline{1}\hookrightarrow\underline{1}\,\!_{+}$.
Then there is no distinction between inputs and outputs of $O$, because
$M\pr{S_{+}}$ only knows ``inputs''. This suggests that the pair
$\pr{O,M}$ comes from a cyclic operad, and Theorem \ref{thm:cyc_opd_mod}
confirms this intuition.
\end{example}

\section{Evaluation at the root edge is a cocartesian fibration}
\label{sec:ev_root}

The goal of this section is to prove that for an $\infty$-operad $\cat O$, the evaluation at the root edge gives a cocartesian fibration
\[
	\Alg_{C_n}(\cat O)\longrightarrow  U\cat O,
\]
and the dual statement for the evaluation at the leaves (Proposition \ref{prop:root_cc}).

These results are so basic that proper proofs most likely need a concrete model of $\infty$-operads. For this reason, we will use  Lurie's model of $\infty$-operads throughout this section.
In accordance with the notation in \cite{lurie_higher_nodate}, typical $\infty$-operads will be denoted with a tensor symbol on it, like $\cat{O}^\t$, and the underlying $\infty$-category will be denoted by dropping the tensor. Thus $\cat O=  U\pr{\cat{O}^{\t}}=\cat{O}^{\t}_{\langle 1\rangle}$.

We start with a precise statement of the main result of this section.
\begin{prop}
\label{prop:root_cc}Let $\cat O^{\t}$ be an $\infty$-operad, and
let $n\geq0$. Consider the functors
\begin{align*}
	\ev_{\mathrm{root}}\from\Alg_{C_{n}}\pr{\cat O} & \longrightarrow\cat O,\\
	\ev_{\mathrm{leaves}}\from \Alg_{C_{n}}\pr{\cat O} & \longrightarrow\cat O^{\times n},
\end{align*}
determined by the evaluation at the root and leaves. The following holds:
\begin{enumerate}
    \item The functor $\ev_{\mathrm{root}}$ is a cocartesian fibration, and the cocartesian morphisms are those whose images under $\ev_{\mathrm{leaves}}$ are equivalences.
    \item The functor $\ev_\mathrm{leaves}$ is a cartesian fibration, and the cartesian morphisms are those whose images under $\ev_{\mathrm{root}}$ are equivalences.
\end{enumerate}

\end{prop}

\begin{rem}\label{rem:restr_at_root}
	Let $(\cO^\otimes,\cW)$ be a relative $\infty$-operad. Since $\cat {W}$ contains all equivalences, Proposition \ref{prop:root_cc} implies that the functor
	\[
\ev_{\mathrm{root}}\from\Alg_{C_{n}}\pr{\cat O}\!\,^{\cat W} \longrightarrow\cW
	\]
is a cocartesian fibration, and that the functor
\[
\ev_{\mathrm{leaves}}\from \Alg_{C_{n}}\pr{\cat O}\!\,^{\cat W} \longrightarrow\cW^{\times n}
\]
is a cartesian fibration.
\end{rem}

The proof of Proposition
\ref{prop:root_cc} needs a few preliminaries.
\begin{notation}
Let $n\geq0$. We write $C^{\t}_{n}$ for the $\infty$-operad corresponding
to the free operad on $C_{n}$. Explicitly, it is the category whose
objects are the finite (possibly empty) tuples $\pr{i_{1},\dots,i_{k}}$,
where $i_{1},\dots,i_{k}\in\{1,\dots,n,n+1\}$. A morphism $\pr{i_{1},\dots,i_{k}}\to\pr{j_{1},\dots,j_{l}}$
is given by a pointed map $\alpha\from\inp k\to\inp l$ such that,
for each $a\in\inp l^{\circ}$, one of the following conditions hold:
\begin{itemize}
\item We have $\abs{\alpha^{-1}\pr a}=n$, $\{i_{t}\mid t\in\alpha^{-1}\pr a\}=\{1,\dots,n\}$,
and $j_{a}=n+1$.
\item We have $\abs{\alpha^{-1}\pr a}=1$ and $i_{\alpha^{-1}\pr a}=j_{a}$.
\end{itemize}
\end{notation}

\begin{notation}
We write $\cat X\subset C^{\t}_{n}$ for the full subcategory of $C^{\t}_{n}$
spanned by the objects $\pr{1,\dots,n},1,\dots,n,n+1$. Equivalently,
$\cat X$ is the poset depicted as

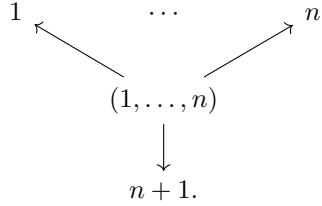
\[\begin{tikzcd}
	1 & \cdots & n \\
	& {(1,\dots,n)} \\
	& {n+1}.
	\arrow[from=2-2, to=1-1]
	\arrow[from=2-2, to=1-3]
	\arrow[from=2-2, to=3-2]
\end{tikzcd}\]
We also write $\cat X_{0}\subset\cat X$ for the full subcategory
spanned by the objects other than $n+1$.
\end{notation}
\begin{defn}
    Let $\cat O^{\t}$ be an $\infty$-operad. A functor $F\from\cat X\to\cat O^{\t}$
over $\Fin_{\ast}$ is called an \emph{$\cat X$-algebra} if it carries the
maps $\{\pr{1,\dots,n}\to a\}_{1\leq a\leq n}$ to inert maps. We
write $\Alg_{\cat X}\pr{\cat O}\subset\Fun_{\Fin_{\ast}}\pr{\cat X,\cat O^{\t}}$
for the full subcategory spanned by the $\cat X$-algebras. 
\end{defn}
\begin{notation}
	For each $n\geq 0$, we will write $\mu_n\from \inp{n} \to \inp{1}$ for the unique active morphism in $\Fin_\ast$. We also write $([1],{\mu_n})$ to mean the category $[1]$ equipped with the functor $[1]\to \Fin_\ast$ corresponding to the morphism $\mu_n$. 
\end{notation}

\begin{prop}
\label{prop:C_n_model}Let $n\geq0$. Consider the fully faithful
functors
\[
[1]\longrightarrow\cat X\longrightarrow C^{\t}_{n},
\]
where the left map corresponds to the active morphism $\pr{1,\dots,n}\to n+1$. For
every $\infty$-operad $p\from\cat O^{\t}\to\Fin_{\ast}$, precomposing
with these functors induces trivial Kan fibrations
\[
\Alg_{C_{n}}\pr{\cat O}\xrightarrow[\;\;\phi\;\;]{\simeq}\Alg_{\cat X}\pr{\cat O}\xrightarrow[\;\;\psi\;\;]{\simeq}\Fun_{\Fin_{\ast}}\pr{( [1],\mu_n ),\cat O}.
\]
\end{prop}
\begin{proof}
We observe that a functor $F\from\cat X\to\cat O^{\t}$ over $\Fin_{\ast}$
is an $\cat X$-algebra if and only if it is $p$-left Kan extended
from $F\vert[1]$, and that every functor $([1],\mu_n)\to \cat{O}$ can be $p$-left Kan extended to $\cat{X}$.
It follows from \cite[Proposition 4.3.2.15]{lurie_higher_2009}
that $\psi$ is a trivial fibration. Consequently, it suffices to
show that the functor $\phi$ is a trivial fibration. For this, again
by \cite[Proposition 4.3.2.15]{lurie_higher_2009}, it suffices to prove the following
pair of assertions:
\begin{enumerate}
\item Every $\cat X$-algebra $F\from\cat X\to\cat O^{\t}$ over $\Fin_{\ast}$
admits a $p$-right Kan extension along $i$.
\item Let $G\from C^{\t}_{n}\to\cat O^{\t}$ be a functor such that $G\vert\cat X$
is an $\cat X$-algebra. Then $G$ is a $p$-right Kan extension of
$G\vert\cat X$ if and only if it is a map of $\infty$-operads.
\end{enumerate}

To prove (1), let $\pr{i_{1},\dots,i_{k}}\in C^{\t}_{n}$ be an arbitrary
object. The category $\cat X\times_{C^{\t}_{n}}\pr{C^{\t}_{n}}_{\pr{i_{1},\dots,i_{k}}/}$
contains an initial full subcategory $\cat I$, spanned by the following
objects:
\begin{itemize}
\item Maps $\pr{i_{1},\dots,i_{k}}\to\pr{1,\dots,n}$, if any.
\item Maps $\pr{i_{1},\dots,i_{k}}\to a$ for $a\in\{1,\dots,n\}$, if any. Note that these maps are necessarily inert.

\item \textit{inert} maps $\pr{i_{1},\dots,i_{k}}\to n+1$, if any. 
\end{itemize}
Let $\cat I'\subset\cat I$ denote the full subcategory spanned by
the objects that are \textit{not} of the form $\pr{i_{1},\dots,i_{k}}\to\pr{1,\dots,n}$.
Then $G\vert\cat I$ is $p$-right Kan extended from $G\vert\cat I'$.
Consequently, it suffices to show that $G\vert\cat I'$ admits a $p$-limit
diagram lying over the inert maps $\{\pr{i_{1},\dots,i_{k}}\to a\}_{a\in\{i_1,\dots ,i_k\}}$,
which is clear. A similar argument proves (2), and the proof is complete.
\end{proof}

For the next lemma, we use the following terminology: We say that
a functor $p\from \cat{E}\to \cat{B} $ of $\infty$-categories is \emph{essentially a cartesian
fibration} if there is a diagram in $\Catinfty$
\[\begin{tikzcd}
	{\mathcal{E}} & {\mathcal{E}'} \\
	{\mathcal{B}} & {\mathcal{B}',}
	\arrow["f", "\simeq"', from=1-1, to=1-2]
	\arrow["p"', from=1-1, to=2-1]
	\arrow["{p'}", from=1-2, to=2-2]
	\arrow["g"',"\simeq", from=2-1, to=2-2]
\end{tikzcd}\]
where $f$ and $g$ are equivalences and $p$ is a cartesian fibration of $\infty$-categories (quasicategories).
In this situation, we say that a morphism of $\cat{E}$ is \emph{$p$-cartesian} if its image in $\cat{E}'$ is $p'$-cartesian. 
\begin{lem}
\label{lem:pb}Consider the following diagram in $\curlyCatinfty$
\[\begin{tikzcd}
	& {\mathcal{B}_0} && {\mathcal{B}_1} \\
	{\mathcal{E}_0} && {\mathcal{E}_1} \\
	& {\mathcal{B}_2} && {\mathcal{B}_3.} \\
	{\mathcal{E}_2} && {\mathcal{E}_3}
	\arrow[from=1-2, to=1-4]
	\arrow[from=1-2, to=3-2]
	\arrow[from=1-4, to=3-4]
	\arrow["{p_0}", from=2-1, to=1-2]
	\arrow[from=2-1, to=2-3, crossing over]
	\arrow[from=2-1, to=4-1]
	\arrow["{p_1}", from=2-3, to=1-4]
	\arrow[from=3-2, to=3-4]
	\arrow["{p_2}", from=4-1, to=3-2]
	\arrow[from=4-1, to=4-3, "f_2"']
	\arrow["{p_3}", from=4-3, to=3-4]
	\arrow[from=2-3, to=4-3, crossing over, "f_1"{pos=0.2}]
\end{tikzcd}\]
Suppose that the following conditions hold:
\begin{enumerate}
	\item The functors $p_{1},p_{2},p_{3}$ are essentially cartesian fibrations.
	\item For each $i\in \{1,2\}$, the functor $f_i$ carries $p_i$-cartesian maps to $p_3$-cartesian maps.
	\item The front and the back faces are cartesian.
\end{enumerate}
Then $p_{0}$ is essentially a cartesian fibration, and $p_0$-cartesian morphisms are those maps whose images in $\cat{E}_i$ are $p_i$-cartesian for every $i\in\{1,2,3\}.$
\end{lem}

\begin{proof}
	Pulling back $\cat{E}_i$ along the maps $\cat B_{0}\to\cat B_{i}$, we may assume
that the back face is constant. In this case, the claim follows from
the fact that the inclusion of the subcategory $\Cart\pr{\cat B}\subset\pr{\curlyCatinfty}_{/\cat B}$
of essentially cartesian fibrations and maps preserving cartesian
maps preserves limits, being a right adjoint.
\end{proof}

\begin{proof}
[Proof of Proposition \ref{prop:root_cc}] We will prove (2); the
proof of part (1) is similar. The functor $\ev_{\mathrm{leaves}}$
is a categorical fibration, so it is a cartesian fibration if and
only if it is essentially a cartesian fibration. Now consider the
diagram:
\[\begin{tikzcd}
	{\operatorname{Alg}_{C_n}(\mathcal{O})} & {\operatorname{Alg}_{\mathcal{X}}(\mathcal{O})} & {\operatorname{Fun}_{\mathrm{Fin}_\ast}(( [1],\mu_n ),\mathcal{O}^{\otimes})} \\
	{\mathcal{O}^{\times n}} & {\operatorname{Fun}'_{\mathrm{Fin}_\ast}(\mathcal{X}_0,\mathcal{O}^{\otimes})} & {\mathcal{O}^{\otimes}_{\langle n \rangle}.}
	\arrow["\phi", "\simeq"', from=1-1, to=1-2]
	\arrow["{\mathrm{ev}_{\mathrm{leaves}}}"', from=1-1, to=2-1]
	\arrow["\psi", "\simeq"', from=1-2, to=1-3]
	\arrow["\rho", from=1-2, to=2-2]
	\arrow["{\mathrm{ev}_0}", from=1-3, to=2-3]
	\arrow["{\phi'}", "\simeq"', from=2-2, to=2-1]
	\arrow["{\psi'}"', "\simeq", from=2-2, to=2-3]
\end{tikzcd}\]
Here $\Fun'_{\Fin_{\ast}}\pr{\cat X_{0},\cat O^{\t}}\subset\Fun_{\Fin_{\ast}}\pr{\cat X_{0},\cat O^{\t}}$
denotes the full subcategory spanned by the functors carrying all
morphisms to inert maps, and $\rho$ denotes the restriction functor.
We have seen that $\phi$ and $\psi$ are equivalences in Proposition
\ref{prop:C_n_model}. Also, the functors $\phi'$ and $\psi'$ are
equivalences by \cite[Proposition 4.3.2.15]{lurie_higher_2009}. Thus, it suffices
to show that the functor
\[
\ev_{0}\from\Fun_{\Fin_{\ast}}\pr{( [1],\mu_n ),\cat O^{\t}}\longrightarrow\cat O^{\t}_{\inp n}
\]
is a cartesian fibration, and that its cartesian edges are the maps whose images under $\ev_1\from \Fun_{\Fin_\ast}(([1],\mu_n),\cat O^\t)\to \cat O$ are equivalences. 

To prove that $\ev_{0}$ is a cartesian fibration, we consider the
following diagram of $\infty$-categories
\[\begin{tikzcd}
	& {\mathcal{O}^{\otimes }_{\langle n\rangle}} && {\mathcal{O}^\otimes} \\
	{\operatorname{Fun}_{\mathrm{Fin}_\ast}(( [1],\mu_n ),\mathcal{O}^\otimes )} && {\operatorname{Fun}([1],\mathcal{O}^\otimes )} \\
	& {\{\langle n\rangle\}} && {\mathrm{Fin}_\ast,} \\
	{\{\mu_n\}} && {\operatorname{Fun}([1],\mathrm{Fin}_\ast )}
	\arrow[from=1-2, to=1-4]
	\arrow[from=1-2, to=3-2]
	\arrow[from=1-4, to=3-4]
	\arrow["{\operatorname{ev}_0}", from=2-1, to=1-2]
	\arrow[from=2-1, to=4-1]
	\arrow["{\operatorname{ev}_0}", from=2-3, to=1-4]
	\arrow[from=3-2, to=3-4]
	\arrow[from=4-1, to=3-2]
	\arrow[from=4-1, to=4-3]
	\arrow["{\operatorname{ev}_0}", from=4-3, to=3-4]
 	\arrow[from=2-3, to=4-3, crossing over]
	\arrow[from=2-1, to=2-3, crossing over]
\end{tikzcd}\]
where $\mu_{n}\from\inp n\to\inp 1$ denotes the unique active map.
We observe that the two slanted arrows on the right are both
cartesian fibrations \cite[\href{https://kerodon.net/tag/0478}{Tag 0478}]{lurie_kerodon_nodate}, whose cartesian edges are the maps whose images under $\ev_1$ are equivalences.
Moreover, the front and back faces are pullback in $\curlyCatinfty$.
Thus, the claim follows from Lemma \ref{lem:pb}.
\end{proof}

We conclude this section with another consequence of Proposition \ref{prop:C_n_model}, which we will use in Section \ref{sec:modelcats}.

\begin{cor}
\label{cor:C_n_pb}Let $\cat C^{\t}\to\Fin_{\ast}$ be a symmetric monoidal $\infty$-category.
For every $n\geq0$, the $\infty$-category $\Alg_{C_{n}}\pr{\cat C}$ fits into a pullback of the form
\[\begin{tikzcd}
	{\Alg_{C_n}(\cat{C})} & {\Fun([1],\cat{C})} \\
	{\cat{C}^{\times n}} & {\cat{C}}
	\arrow[from=1-1, to=1-2]
	\arrow["{\ev_{\mathrm{leaves}}}"', from=1-1, to=2-1]
	\arrow["{\ev_0}", from=1-2, to=2-2]
	\arrow["{\bigotimes_{i=1}^n}"', from=2-1, to=2-2]
\end{tikzcd}\]
in $\Cat_{\infty}$, where the composite $\Alg_{C_{n}}\pr{\cat C}\to\Fun\pr{[1],\cat C}\xrightarrow{\ev_{1}}\cat C$ is equivalent to $\ev_{\mathrm{root}}$.
\end{cor}

Corollary \ref{cor:C_n_pb} is a consequence of Proposition \ref{prop:C_n_model} and the following lemma:

\begin{lem}
Let $p\from\cat M\to[1]$ be a cocartesian fibration of $\infty$-categories with fibers $\cat A=p^{-1}\pr 0$ and $\cat B=p^{-1}\pr 1$.
Choose a cocartesian natural transformation $h\from\cat M\times[1]\to\cat M$ satisfying the following conditions:

\begin{itemize}
\item $h\vert\cat M\times\{0\}=\id_{\cat M}$.
\item $h\vert\cat M\times\{1\}$ factors through $\cat B$.
\end{itemize}
Set $r=h\vert\cat M\times\{1\}$ and $f=r\vert\cat A$.
The square 
\[\begin{tikzcd}
	{\Fun_{[1]}([1],\cat{M})} & {\Fun([1],\cat{B})} \\
	{\cat{A}} & {\cat{B}}
	\arrow["{r_\ast}", from=1-1, to=1-2]
	\arrow["{\ev_0}"', from=1-1, to=2-1]
	\arrow["{\ev_0}", from=1-2, to=2-2]
	\arrow["f"', from=2-1, to=2-2]
\end{tikzcd}\]
is a pullback in $\Cat_{\infty}$.
\end{lem}

\begin{proof}
Recall that over a $1$-category, there is an explicit model of unstraightening, called the relative nerve \cite[\S 3.2.5]{lurie_higher_2009}.
Without loss of generality, we may assume that $\cat M$ is the relative nerve of the functor $[1]\to\Cat_{\infty}$ classifying $f$.
This means that $\cat M$ is characterized by the following (1-categorical) universal property: Given an $\infty$-category $\cat X$ equipped with a map $\cat X\to[1]$, the data of a functor $\cat X\to\cat M$ over $[1]$ is equivalent to giving a commutative diagram
\[\begin{tikzcd}
	{\cat{X}_0} & {\cat{X}} \\
	{\cat{A}} & {\cat{B}.}
	\arrow[from=1-1, to=1-2]
	\arrow[from=1-1, to=2-1]
	\arrow[from=1-2, to=2-2]
	\arrow["f"', from=2-1, to=2-2]
\end{tikzcd}\]
Here, $\cat{X}_0$ denotes the strict fiber of $\cat{X}$ over $0\in [1]$.
In particular, a morphism of $\cat M$ lying over $0\to1$ consists of an object $A\in\cat A$ equipped with a morphism $\alpha\from f\pr A\to B$ in $\cat B$.
Such a morphism is cocartesian if and only if $\alpha$ is an equivalence.

There is a canonical retraction $\cat M\to\cat B$, and the above description of cocartesian morphisms shows that we can choose $r$ to be this natural retraction.
With this choice, the square in the statement is a strict pullback of quasi-categories.
Since its vertical arrows are categorical fibrations, this is in fact a homotopy pullback, and we are done.
\end{proof}

\bibliographystyle{amsalpha}
\bibliography{References}

\vspace{0.5cm}

\address{\textsc{K.A.: Research Institute for Mathematical Sciences, Kyoto University, Kitashirakawa-Oiwakecho, 606-8502
Kyoto (Japan)}} \newline \email{\texttt{karakawa@kurims.kyoto-u.ac.jp}}

\vspace{0.2cm}
\address{\textsc{V.C.: Max Planck Institute for Mathematics in the Sciences, Inselstrasse 22, 04103 Leipzig (Germany)}} \newline \email{\texttt{vcarmonamath@gmail.com}}

\vspace{0.2cm}
\address{\textsc{F.P.: Utrecht Geometry Center, Utrecht University, Hans Freudenthalgebouw
	Budapestlaan 6,	3584 CD Utrecht (The Netherlands)}} \newline \email{\texttt{ f.pratali@uu.nl}}
\end{document}